\documentclass[12pt]{amsart}
\usepackage{amsfonts, amssymb, amsmath, amsthm, eucal, latexsym, array, mathabx}
\usepackage{fullpage}
\usepackage{verbatim}

\usepackage{xcolor}
\definecolor{darkblue}{rgb}{0,0,.5}
\definecolor{darkgreen}{rgb}{.2,0.5,.2}
\usepackage[colorlinks=true,linkcolor=darkblue,urlcolor=violet,citecolor=magenta]{hyperref}

\numberwithin{equation}{section}

\input{cyracc.def}

\font\tencyr=wncyr10 
\font\tencyi=wncyi10 
\font\tencysc=wncysc10 

\def\rus{\tencyr\cyracc}
\def\rusi{\tencyi\cyracc}
\def\rusc{\tencysc\cyracc}

\newtheorem{thm}{Theorem}[section]
\newtheorem{lm}[thm]{Lemma}
\newtheorem{cl}[thm]{Corollary}
\newtheorem{prop}[thm]{Proposition}

\theoremstyle{remark}
\newtheorem{ex}[thm]{Example}
\newtheorem{rmk}[thm]{Remark}
\newtheorem{qn}[thm]{Question}
\newtheorem*{rem}{Remark}

\theoremstyle{definition}

\newcommand{\gt}{\mathfrak}

\newcommand{\GL}{{\rm GL}}

\newcommand{\id}{{\rm id}}

\newcommand{\diag}{{\rm diag}}

\newcommand{\rk}{\mathrm{rk\,}}

\newcommand{\Lie}{\mathrm{Lie\,}}

\newcommand{\gr}{\mathrm{gr\,}}

\newcommand{\ad}{\mathrm{ad}}

\newcommand{\esi}{\varepsilon}

\newcommand {\ca}{{\mathcal A}}

\newcommand {\cS}{{\mathcal S}}

\newcommand {\cG}{{\mathcal G}}
\newcommand {\Gz}{{\mathcal G}(\bar z)}

\newcommand{\ta}{{\tt a}}
\newcommand{\tb}{{\tt b}}
\newcommand{\tc}{{\tt c}}
\newcommand{\tal}{{\tt \alpha}}
\newcommand{\tbe}{{\tt \beta}}
\newcommand{\tga}{{\tt \gamma}}

\newcommand {\mK}{{\mathbb C}}

\newcommand {\Z}{{\mathbb Z}}

\renewcommand{\le}{\leqslant}
\renewcommand{\ge}{\geqslant}

\newfam\eusfam%
\font\euszw=eusm10 scaled 1200%
\font\eusac=eusm7 scaled 1200%
\font\eusacc=eusm7 scaled 1000%
\textfont\eusfam=\euszw\scriptfont\eusfam=\eusac%
\scriptscriptfont\eusfam=\eusacc%
\newcommand{\ts}{\,}

\newcommand{\U}{{\mathcal U}}

\newcommand{\wg}{{\widehat{\gt g}}}

\newcommand {\eus}{\EuScript}

\newcommand{\lf}{{\boldsymbol f}}
\newcommand{\lb}{{\boldsymbol b}}
\newcommand{\lH}{{\boldsymbol H}}
\newcommand{\bb}{{\boldsymbol B}}

\newcommand{\gl}{\mathfrak{gl}}

\newcommand{\ggot}{\mathfrak{g}}
\newcommand{\z}{\mathfrak{z}}

\begin{document}
\hfill {\scriptsize May 6, 2020} 
\vskip1ex

\title{Symmetrisation and the Feigin--Frenkel centre}
\author{Oksana Yakimova}
\address{Institut f\"ur Mathematik, Friedrich-Schiller-Universit\"at Jena, Jena, 07737, Deutschland} 
\email{oksana.yakimova@uni-jena.de}
\thanks{This work is 
funded by the Deutsche Forschungsgemeinschaft (DFG, German Research Foundation) --- project number  404144169.} 
\subjclass[2010]{16S30, 17B67, 17B20, 17B35, 17B63}
\maketitle

\begin{abstract}
For complex simple Lie algebras of types {\sf B}, {\sf C}, and {\sf D}, we provide new explicit formulas for the generators of
the  commutative 
subalgebra $\gt z(\widehat{\gt g})\subset\U(t^{-1}\gt g[t^{-1}])$ known as the  
{\it Feigin--Frenkel centre}. These formulas make use of the symmetrisation map as well as of some well-chosen symmetric invariants of $\gt g$.  There are some general results on the r\^ole of the symmetrisation map in the explicit description of the Feigin--Frenkel  centre. Our method  reduces questions about elements of $\gt z(\widehat{\gt g})$ to questions on the structure of the symmetric invariants in a type-free way. As an illustration, we deal with type 
{\sf G}$_2$ by hand. One of our technical tools is  the map ${\sf m}\!\!: \cS^{k}(\gt g)\to \Lambda^2\gt g \otimes \cS^{k-3}(\gt g)$ introduced here. As the results show, a better understanding of this map will lead to a better understanding of 
$\gt z(\widehat{\gt g})$. 
\end{abstract} 

\section*{Introduction}

Let $G$ be a  complex reductive group.  
Set $\gt g=\Lie G$.  
As is well-known, the algebra $\cS(\gt g)^{\gt g}$ of {\it symmetric} $\gt g$-invariants and
the centre ${\mathcal Z}(\gt g)$ of the enveloping algebra $\U(\gt g)$ are polynomial algebras 
with $\rk\gt g$ generators. Therefore there are several isomorphisms between them. 
Two of these isomorphisms can be distinguished, the one given by the {\it symmetrisation map}, which is a homomorphism of $\gt g$-modules,  and the Duflo 
isomorphism, which is a homomorphism of algebras. Both of them exist for any finite-dimensional 
complex Lie algebra. 

The symmetrisation map is defined in the infinite dimensional case as well. However, no 
analogue of the Duflo isomorphism for  Lie algebras $\gt q$ with $\dim\gt q=\infty$ is known. 
Furthermore, one may need to complete $\U(\gt q)$ in order to replace ${\mathcal Z}(\gt q)$ 
with an interesting related object, see e.g. \cite{Kac-L}. 
In this paper,  we are dealing with the most notable class of infinite dimensional Lie algebras, namely  
affine Kac--Moody algebras $\wg$, and  the related centres at the critical level.

The Feigin--Frenkel centre $\gt z(\widehat{\gt g})$  is a remarkable commutative 
subalgebra of the enveloping algebra $\U(t^{-1}\gt g[t^{-1}])$.  
The central elements of the completed enveloping 
algebra $\widetilde{\U}_{\kappa}(\wg)$ at the critical level $\kappa=-{\tt h}\!^\vee$
can be obtained from the elements of $\gt z(\wg)$ by employing the vertex algebra structure \cite[Sect.~4.3.2]{f:lc}. 
The structure of  $\gt z(\widehat{\gt g})$ is described by a theorem of Feigin and Frenkel \cite{ff:ak}, hence the name. This algebra provides a quantisation of the local Hitchin system \cite[Sect.~2]{bd}. 
Elements $S\in\gt z(\wg)$ give rise to higher Hamiltonians of the Gaudin model, which  describes 
a completely integrable quantum spin chain \cite{FFRe}. 

The classical counterpart of $\gt z(\wg)$ is the Poisson-commutative subalgebra of 
$\gt g[t]$-in\-va\-riants in $\cS(\gt g[t,t^{-1}])/(\gt g[t])\cong \cS(t^{-1}\gt g[t^{-1}])$, which is 
a polynomial ring with infinitely  many generators according to a direct generalisation of 
a Ra\"is--Tauvel theorem \cite{rt}. Explicit formulas for the elements of $\gt z(\wg)$ appeared first in
type {\sf A}~\cite{ct:qs,cm:ho} following Talalaev's discovery \cite{Tal} of  explicit higher 
Gaudin Hamiltonians. Then they were extended to all classical types in \cite{m:ff}.
The construction of \cite{m:ff} relies on the Schur--Weyl duality involving the Brauer algebra.  Type 
{\sf G}$_2$ is covered by \cite{mrr}.  
The subject is beautifully summarised in \cite{book2}. 

Unlike the finite-dimensional case,  no natural isomorphism between the algebras $\cS(t^{-1}\gt g[t^{-1}])^{\gt g[t]}$ and $\gt z(\wg)$ is known. Also, 
generally speaking, an element of $\gt z(\widehat{\gt g})$ cannot be obtained by the symmetrisation $\varpi$ from 
a homogeneous 
$\gt g[t]$-invariant in $\cS(t^{-1}\gt g[t^{-1}])$. At the same time, some of the elements do come in this way, 
see Example~ \ref{paf}, which is dealing with the Pfaffians of $\gt{so}_{2n}$. 
In this paper, we show that for all classical Lie algebras, $\varpi$ can produce generators of $\gt z(\wg)$. 
The symmetrisation map is not a homomorphism of algebras. 
However, it is a homomorphism of $\gt g[t^{-1}]$-modules and it behaves well with respect to taking
various limits.  

According to a striking result of L.\,Rybnikov \cite{r:un}, $\gt z(\wg)$ is the centraliser in 
$\U(t^{-1}\gt g[t^{-1}])$ of a single quadratic element ${\mathcal H}[-1]$, see Section~\ref{s-zen}.  
This fact is crucial for our considerations. 

\vskip0.5ex

Any ${\mathcal Y}\in \U(t^{-1}\gt g[t^{-1}])$ can be expressed as 
a sum 
\begin{equation} \label{sim-alg}
\varpi(Y_k)+\varpi(Y_{k{-}1})+\ldots + Y_1 + Y_0 \ \ \text{with} \ \
 Y_j\in\cS^j(t^{-1}\gt g[t^{-1}]).
\end{equation}
Here $Y_k=\gr\!({\mathcal Y})$ if $Y_k\ne 0$.  Note that 
$\sum\limits_{0\le j\le k}\varpi(Y_j)$ is a $\gt g$-invariant 
if and only if each $Y_j$ is a $\gt g$-invariant. In the following, we consider only elements with 
$Y_0=0$. 

A {\it polarisation} of a $\gt g$-invariant $F\in\cS(\gt g)$ is a $\gt g$-invariant in 
$\cS(t^{-1}\gt g[t^{-1}])$, see Section~\ref{s-fs} for the definition of  a polarisation. However,  
$\cS(t^{-1}\gt g[t^{-1}])^{\gt g}$ is not generated by elements of this sort, see \eqref{-3} for an example.

There are finite sets of elements 
$\{S_1,\ldots,S_{\ell}\}\subset\gt z(\wg)$ with $\ell=\rk\gt g$, called {\it complete sets of Segal--Sugawara vectors}, see 
Section~\ref{ssv} for the definition,  that are of vital importance for the understanding of $\gt z(\wg)$. 
We prove that if $\gt g$ is either a classical Lie algebra or an exceptional Lie algebra of type {\sf G}$_2$, then 
there is a complete set $\{S_k\}$ of Segal--Sugawara vectors such that 
all the terms $Y_j$ occurring  in presentations~\eqref{sim-alg} for $S_k$ 
are polarisations of 
symmetric  invariants of $\gt g$. 
The map {\sf m}, defined in Section~\ref{sec-m},  plays a crucial r\^ole in the 
selection  of suitable $\gt g$-invariants. 
In particular, if $F[-1]\in\cS^k(\gt g t^{-1})$ is obtained from $F\in\cS^k(\gt g)^{\gt g}$
using the canonical isomorphism $\gt g t^{-1}\cong\gt g$, then 
$\varpi(F[-1])\in\gt z(\wg)$ if and only if ${\sf m}(F)=0$, see Theorem~\ref{m-sym} and the remark after it. More generally,
if $H\in\cS^k(\gt g)^{\gt g}$ is such that 
\begin{equation} \label{property}
{\sf m}^d(H)\in \cS(\gt g) \ \ \text{ for all } \ \ 1\le d < k/2, 
\end{equation} 
then there is a way to produce an element of $\gt z(\wg)$ corresponding to $H$,  
see Theorem~\ref{thm-form} and  \eqref{form-m}.

First for $F=\xi_1\ldots \xi_m\in\cS^m(\gt g)$ and $\bar a=(a_1,\ldots,a_m)\in\Z_{<0}^m$, set 
\begin{equation} \label{a}
\varpi(F)[\bar a]=\frac{1}{m!} \sum_{\sigma\in{\tt S}_m} \xi_{\sigma(1)} t^{a_1}\ldots \xi_{\sigma(m)} t^{a_m} \in\U(t^{-1}\gt g[t^{-1}]),
\end{equation}
then extend this notation to all elements  $F\in\cS^m(\gt g)$ by linearity. According to 
Lemma~\ref{lm-sym-a}, $\varpi(F)[\bar a]=\varpi(F[\bar a])$ for the $\bar a$-polarisation 
$F[\bar a]\in\cS^m(t^{-1}\gt g[t^{-1}])$ of $F$. The expression  $\varpi(\tau^r F[-1]){\cdot}1$ encodes a sum of 
$\frac{1}{(m+r)!} c(r,\bar a) \varpi(F)[\bar a]$, where the vectors $\bar a\in\Z_{<0}^m$ are  such that 
$\sum_{j=1}^m a_j=-m-r$ and $c(r,\bar a)\in\mathbb N$ are certain combinatorially defined coefficients, which 
we do not compute explicitly. The combinatorial aspects of our approach remain hidden for the moment, 
cf. Remark~\ref{rm-comb}. 

For each classical Lie algebra $\gt g$, there is a set of generators 
$\{H_1,\ldots,H_\ell\}\subset\cS(\gt g)^{\gt g}$ such that ${\sf m}(H_k)\in \mK H_j$ for some $j$ depending on $k$, see 
Sections~\ref{sec-A}, \ref{sec-C}, \ref{sec-ort} and in particular Propositions~\ref{works}, \ref{sp-m}, \ref{prop-D}. 
In types {\sf A} and {\sf C}, we are using the coefficients of the characteristic polynomial, 
in the orthogonal case, one has to work with $\det(I_n-q(F_{ij}))^{-1}$ instead. 
In type {\sf A}$_{n-1}$, ${\sf m}(\tilde\Delta_k)=\frac{(n-k+2)(n-k+1)}{k(k-1)} \tilde\Delta_{k-2}$; in type {\sf C}$_n$, 
${\sf m}(\Delta_{2k})=\frac{(2n-2k+3)(2n-2k+2)}{2k(2k-1)} \Delta_{2k-2}$; 
and finally for $\gt g=\gt{so}_n$, we have ${\sf m}(\Phi_{2k})=\frac{1}{k(2k-1)} \big(\binom{n}{2} + 2n(k-1) + (k-1)(2k-3) \big) \Phi_{2k-2}$. This leads to the following complete sets of Segal--Sugawara vectors:\\[.5ex]
$\{\tilde S_{k-1}=\varpi(\tilde\Delta_k[-1]) + \sum\limits_{1\le r<(k-1)/2} \binom{n-k+2r}{2r} \varpi(\tau^{2r} \tilde\Delta_{k-2r}[-1]){\cdot}1 \mid 2\le k\le n\}$ in type ${\sf A}_{n-1}$; \\[.3ex]
$\{S_k=\varpi(\Delta_{2k}[-1]) + 
\sum\limits_{1\le r<k} \binom{2n{-}2k{+}2r{+}1}{2r}\varpi(\tau^{2r}\Delta_{2k{-}2r}[-1]){\cdot}1\mid 1\le k\le n\}$ in type ${\sf C}_n$; \\[.3ex]
$\{S_{k}=\varpi(\Phi_{2k}[-1]) +\sum\limits_{1\le r<k} R(k,r) \varpi(\tau^{2r} \Phi_{2k-2r}[-1]){\cdot 1} \mid 1\le k <\ell \}$ for $\gt{so}_n$ with $n=2\ell-1$  \\[.3ex]
with  the addition of $S_{\ell}=\varpi({\rm Pf}[-1])$ for $\gt{so}_{n}$ with $n=2\ell$, where   
$$
R(k,r)=\frac{2^r}{(2r)!} \prod\limits_{u=1}^r \left( \binom{n}{2} + 2n(k-u) +
(k-u)(2k-2u-1) \right).
$$
The result in type {\sf A} is not new. It follows via a careful rewriting from the formulas of~\cite{ct:qs,cm:ho}. 
We are not giving a new proof, quite to the contrary, we use the statement in type {\sf A} in order to 
extend the formula to other types. 

It would be important to find out, whether our formulas for $\gt{so}_n$ and $\gt{sp}_{2n}$ 
describe the same elements as \cite{m:ff}. This is indeed the case for the 
Pfaffian-type Segal--Sugawara vector, see Section~\ref{sec-Spec}.
The advantage of our method is that 
it reduces questions about elements of $\gt z(\wg)$ to questions on the structure of $\cS(\gt g)^{\gt g}$ in a 
type-free way. For example, it is possible to deal with type ${\sf G}_2$ by hand unlike \cite{mrr}, see \eqref{g2-7}.  
It is quite probable, that other exceptional types can be handled on a computer. 
Conjecturally, each exceptional Lie algebra possesses a set $\{H_k\}$ of generating symmetric invariants 
such that each $H_k$ satisfies \eqref{property}.

\vskip0.5ex

One of the significant applications of the Feigin--Frenkel centre is related to {\it Vinberg's quantisation problem}. 
The symmetric algebra $\cS(\gt g)$ carries a Poisson structure extended from the Lie bracket on $\gt g$ by the 
Leibniz rule.  To each $\mu\in\gt g^*\cong\gt g$, one associates the  
{\it Mishchenko--Fomen\-ko subalgebra} $\ca_\mu\subset\cS(\gt g)$, which is an extremely  
interesting 
Poisson-commutative subalgebra \cite{mf:ee}. 
In \cite{v:sc}, Vinberg proposed to find a commutative subalgebra ${\mathcal C}_\mu\subset\U(\gt g)$ such that the graded image $\gr\!({\mathcal C}_\mu)$ coincides with $\ca_\mu$. 
Partial solutions to this problem are obtained in \cite{no:bs,t:cs}.
The breakthrough came in \cite{r:si}, where a certain commutative subalgebra 
$\tilde\ca_\mu\subset\U(\gt g)$ is constructed as an image of $\gt z(\wg)$, see \eqref{ff-map}.

In \cite[Sect.~3.3]{my}, sets of generators $\{H_k\mid 1\le k\le \ell\}$ of $\cS(\gt g)^{\gt g}$ such that 
$\tilde\ca_\mu$ is generated by $\varpi(\partial_\mu^m H_k)$, cf.~\eqref{MF-sym}, 
are exhibited in types {\sf B}, {\sf C}, and {\sf D}.
For the symplectic Lie algebra, $H_k=\Delta_{2k}$, in the orthogonal case $H_k=\Phi_{2k}$ 
with the exception of $H_\ell={\rm Pf}$ in type ${\sf D}_\ell$.  Results of this paper 
provide a different proof for \cite[Thm~3.2]{my}.   
We have pushed the symmetrisation map to the level of $\U(t^{-1}\gt g[t^{-1}])$. 

In Section~\ref{G-A}, we briefly consider Gauding algebras $\cG$. If $\gt g$ is a classical Lie algebra, then  
the 
{\it two points} Gauding subalgebra $\cG\subset\U(\gt g{\oplus}\gt g)$ is generated by the symmetrisations of 
certain bi-homogeneous $\gt g$-invariants in $\cS(\gt g{\oplus}\gt g)$, see Theorem~\ref{2-Gau}.

\section{Preliminaries and notation}

Let $\gt g=\Lie G$ be a non-Abelian complex reductive Lie algebra. 
The Feigin--Frenkel centre  $\gt z(\wg)$ 
is the centre of the universal affine vertex algebra
associated with the affine Kac--Moody algebra $\wg$
at the critical level \cite{ff:ak,f:lc}.  There is an injective homomorphism 
$\gt z(\wg) \hookrightarrow \U(t^{-1}\ggot[t^{-1}])$
and $\gt z(\wg)$ can be viewed as a commutative subalgebra of $\U(t^{-1}\ggot[t^{-1}])$ \cite[Sect.~3.3]{f:lc}. 
Each element of $\z(\widehat{\ggot})$ is annihilated by the
adjoint action of $\gt g$, cf. \cite[Sect.~6.2.]{book2}.

\subsection{The Feigin--Frenkel centre as a centraliser}  \label{s-zen}
We set $\gt g[b]:=\gt g t^{b}$ and $x[b]:=xt^b$ for $x\in\gt g$. Furthermore, 
$\widehat{\gt g}^{-}:= t^{-1}\ggot[t^{-1}]$. 
According to \cite{r:un}, $\z(\widehat{\ggot})$ 
is the centraliser in $\U(\widehat{\gt g}^{-})$ of the following 
quadratic element 
$$
{\mathcal H}[-1]=\sum_{a=1}^{\dim\gt g} x_a[-1]x_a[-1], 
$$
where  $\{x_1,\ldots,x_{\dim\gt g}\}$ is any basis of $\gt g$ that is 
orthonormal w.r.t. a fixed $\gt g$-invariant  non-degenerate 
scalar product $(\,\,,\,)$.  

\subsection{The symmetrisation map} For any complex Lie algebra $\gt q$, let 
$\varpi\!:\cS^k(\gt q)\to \gt q^{\otimes k}$ 
be the canonical symmetrisation map. Following the usual convention, we let $\varpi$ stand also for the symmetrisation map from $\cS(\gt q)$ to $\U(\gt q)$. 
Let $\gr\!(X)\in\cS(\gt q)$ be the symbol of $X\in\U(\gt q)$. Then  
$\gr\!(\varpi(Y))=Y$ for $Y\in\cS^k(\gt q)$ by the construction.  

\subsection{The antipode} Let us define the  anti-involution $\omega$ on ${\mathcal U}(\widehat{\gt g}^{-})$ 
to be the $\mathbb C$-linear map such that $\omega(\xi[k])=-\xi[k]$ for each $\xi\in\gt g$ and 
$$
\omega(\xi_1[k_1]\xi_2[k_2]\ldots \xi_m[k_m])=(-\xi_m[k_m])\ldots (-\xi_2[k_2])(-\xi_1[k_1]).
$$  
Let also $\omega$ be the analogues  anti-involution on $\U(\gt q)$ for any complex Lie algebra $\gt q$.  

Clearly, $\omega({\mathcal H}[-1])={\mathcal H}[-1]$. Therefore $\omega$ acts on  
$\z(\widehat{\ggot})$.  
For $Y_j\in\cS^j(\wg^-)$, we have $\omega(\varpi(Y_j))=(-1)^j\varpi(Y_j)$. 
A non-zero element ${\mathcal Y}\in\U(\wg^-)$ presented in the form~\eqref{sim-alg}
is an eigenvector of $\omega$ if and only if 
either all $Y_j$ with even $j$ or all $Y_j$ with odd $j$ are zero.

\subsection{The map {\sf m}} \label{sec-m}
For $\gt{gl}_N=\gt{gl}_N(\mK)={\rm End}(\mK^N)$ and $1\le r\le k$, 
consider the linear map 
$$
{\sf m}_r\!: \gt{gl}_N^{\otimes k}\to \gt{gl}_N^{\otimes (k{-}r{+}1)} \ \ 
\text{ that  sends } \ \ y_1{\otimes}\ldots{\otimes}\,y_k \ \ \text{ to } \ \ 
y_1y_2\ldots y_r\otimes y_{r+1}{\otimes}\ldots{\otimes}\,y_k. 
$$
Note that clearly ${\sf m}_r{\circ}{\sf m}_s={\sf m}_{r+s-1}$. 
Via the adjoint representation of $\gt g$, the map ${\sf m}_r$ leads to a map 
$\gt g^{\otimes k}\to \gl(\gt g) \otimes \gt g^{\otimes (k{-}r)}$, which we denote by the same symbol. 
Observe that 
$$
\ad(y_1)\ad(y_2)\ldots\ad(y_{2r+1})+\ad(y_{2r+1})\ldots\ad(y_2)\ad(y_1) \in\gt{so}(\gt g)\cong \Lambda^2\gt g.
$$
We embed $\cS^k(\gt g)$ in $\gt g^{\otimes k}$ via $\varpi$. 
Set ${\sf m}={\sf m}_3$. Then
${\sf m}\!: \cS^{k}(\gt g)\to \Lambda^2\gt g \otimes \cS^{k-3}(\gt g)$. 
For example, if $Y=y_1y_2y_3\in\cS^3(\gt g)$, then 
\begin{align*}
& {\sf m}(Y)=\frac{1}{6}\big(\ad(y_1)\ad(y_2)\ad(y_3)+\ad(y_3)\ad(y_2)\ad(y_1)+ \ad(y_1)\ad(y_3)\ad(y_2) + \\
& \qquad  \qquad + \ad(y_2)\ad(y_3)\ad(y_1) 
+ \ad(y_2)\ad(y_1)\ad(y_3) + \ad(y_3)\ad(y_1)\ad(y_2)\big)\in\gt{so}(\gt g).
\end{align*}
Similarly one defines ${\sf m}_{2r+1}\!: \cS^{k}(\gt g)\to \Lambda^2\gt g \otimes \gt g^{\otimes (k{-}2r{-}1)}$
for each odd $2r+1\le k$.   
Note  that each ${\sf m}_{2r+1}$ is $G$-equivariant. 
It is convenient to put ${\sf m}(\cS^k(\gt g))=0$ for $k\le 2$. 

Suppose that $\gt g$ is simple. There is a $G$-stable decomposition  $\Lambda^2\gt g=\gt g\oplus V$. This $V$ will be called the {\it Cartan component} of $\Lambda^2\gt g$. If $\gt g$ is not of type {\sf A}, then $V$ is irreducible. For certain elements $H\in\cS^k(\gt g)$, we have 
${\sf m}(H)\in \gt g\otimes \cS^{k-3}(\gt g)$. 
If 
${\sf m}(H)\in\cS^{k-2}(\gt g)$, 
then 
${\sf m}_{2r+1}(H)={\sf m_{2r-1}}{\circ}{\sf m}(H)$. 
Note that ${\sf m}(\cS^3(\gt g)^{\gt g})=0$, since $(\Lambda^2\gt g)^{\gt g}=0$. 

\subsection{Polarisations and fully symmetrised elements} \label{s-fs}
For elements $y_1,\ldots , y_m \in \gt g$ and a vector $\bar a=(a_1,\ldots,a_m)\in \mathbb Z_{<0}^m$, set 
$\Upsilon[\bar a]=\prod\limits_{i=1}^m y_i[a_i]\in \cS(\widehat{\gt g}^{-})$. 
If we consider the product $Y=\prod_{i} y_i\in\cS^m(\gt g)$, then there is no uniquely defined sequence of factors $y_i$.  
However, the {\it $\bar a$-polarisation}  $Y[\bar a]:=\frac{1}{m!} \sum\limits_{\sigma\in{\tt S}_m} \Upsilon[\sigma(\bar a)]$ of $Y$ is  well-defined. 
We extend this notion to all elements of $\cS^m(\gt g)$ by linearity. 
Linear combinations of the elements 
$$
\varpi\left( Y[\bar a]  \right) \in \U(\widehat{\gt g}^{-})
$$
are said to be {\it fully symmetrised}. Note that  $\varpi(H)$ is fully symmetrised if $H\in\cS^m(\gt g t^{-1})$. 
If $a_i=a$ for all $i$, then $\Upsilon[\bar a]=Y[\bar a]$ and we denote it simply by $Y[a]$. 

The evaluation ${\sf Ev}_1$  at $t=1$ defines an isomorphism ${\sf Ev}_1\!: \cS(\gt g[a])\to \cS(\gt g)$ of $\gt g$-modules. For $F\in\cS(\gt g)$, let $F[a]$ stand for ${\sf Ev}_1^{-1}(F)\in  \cS(\gt g[a])$.    
Then $\varpi(F)[a]:=\varpi(F[a])$ is fully symmetrised.  

\subsection{Segal--Sugawara vectors}\label{ssv} Set $\tau=-\partial_t$.  According to \cite{ff:ak}, $\gt z(\widehat{\ggot})$ 
is a polynomial algebra in infinitely many variables with a distinguished set of ``generators"  
$\{S_1,\ldots,S_{\ell}\}$ such that $\ell=\rk\gt g$ and 
$$
\gt z(\widehat{\ggot})= \mK[\tau^m (S_k) \mid 1\le k\le\ell, m\ge 0].
$$
We have $\gr\!(S_k)=H_k[-1]$ with $H_k\in\cS(\gt g)^{\gt g}$ and 
$\mK[H_1,\ldots,H_\ell]=\cS(\gt g)^{\gt g}$. 
The set $\{S_k\}$ is said to be a {\it complete set of Segal--Sugawara vectors}.  
The symbols of $\tau^m(S_k)$ generate $\cS(\wg^-)^{\gt g[t]}$ in accordance with 
\cite{rt}.  

Suppose that we have $\tilde S_k\in\gt z(\wg)$ with $1\le k\le\ell$ and 
$\gr\!(S_k)=\tilde H_k[-1]$, where $\tilde H_k\in\cS(\gt g)^{\gt g}$, for each $k$. 
The structural properties of $\gt z(\wg)$ imply that  $\{\tilde S_k\}$ is a complete set of Segal--Sugawara vectors if 
and only if the set $\{\tilde H_k\}$ generates $\cS(\gt g)^{\gt g}$.

\subsection{Symmetric invariants} 
For a finite-dimensional Lie algebra $\gt q$, we have $\cS(\gt q)\cong\mK[\gt q^*]$. 
For any reductive Lie algebra, there is an isomorphism of $\gt g$-modules 
$\gt g\cong\gt g^*$.  For $\xi\in(\gt{gl}_n)^*$, write  
\begin{equation} \label{d-gl}
\det(qI_n-\xi)=q^n-\Delta_1(\xi)q^{n-1} + \ldots + (-1)^{k}\Delta_k(\xi) q^{n-k} + \ldots + (-1)^n \Delta_n(\xi)  .
\end{equation}
Then $\cS(\gt{gl}_n)^{\gt{gl}_n}=\mK[\Delta_1,\ldots,\Delta_n]$. 

Let $\gt f\subset\gt g$ be a reductive subalgebra. Then there is an $\gt f$-stable subspace 
$\gt m\subset\gt g$ such that $\gt g=\gt f\oplus\gt m$, whereby  also 
$\gt g^* \cong\gt f^*\oplus\gt m^*$. Identifying $\gt f$ with $\gt f^*$, one  
defines the {\it restriction} $H|_{\gt f}$ of $H\in\cS(\gt g)$ to $\gt f$. This is the image of $H$ in
$\cS(\gt g)/\gt m\cS(\gt g)\cong\cS(\gt f)$. 

In cases $n=2\ell$, $\gt f=\gt{sp}_{2\ell}$ and $n=2\ell+1$, $\gt f=\gt{so}_n$, the 
restrictions ${\Delta_{2k}}|_{\gt f}$ with $1\le k\le \ell$ form a generating set 
in $\cS(\gt f)^{\gt f}$. In case $\gt f=\gt{so}_{n}$ with $n=2\ell$, 
the restriction of the determinant $\Delta_{2\ell}$ is the square of the Pfaffian and 
has to be replaced by the Pfaffian  in the generating set.  

Explicit formulas for basic symmetric invariants of the exceptional Lie algebras 
are less transparent. 

The inclusions $\gt g\subset\cS(\gt g)$ are ruled  by the symmetric invariants. 
The key point here is that $\cS(\gt g)$ is a free module over $\cS(\gt g)^{\gt g}$ \cite{K}. 
If 
$\{H_1,\ldots,H_\ell\}\subset\cS(\gt g)^{\gt g}$ is a generating set consisting of homogeneous elements 
and $\deg H_i=d_i+1$,
then to each $i$ 
corresponds a {\it primitive} copy of $\gt g$ in $\cS^{d_i}(\gt g)$. The non-primitive
copies are obtained as linear combinations of  the primitive ones with coefficients  from $\cS(\gt g)^{\gt g}$.

\subsection{Miscellaneousness}
Let $\gt h\subset\gt g$ be a Cartan subalgebra, we let $\ell$ stand for $\dim\gt h=\rk\gt g$ and
$W=W(\gt g,\gt h)$ stand for the Weyl group of $\gt g$.   
The fundamental weights of $\gt g$ are $\pi_k$ with $1\le k\le\ell$ and 
$V(\lambda)$ stands for an irreducible finite-dimensional $\gt g$-module with the highest weight 
$\lambda=\sum_{k=1}^{\ell} c_k \pi_k$. Please keep in mind that  
the Vinberg--Onishchik numbering \cite[Tables]{VO} of simple roots (and fundamental weights) is used. 
If $\alpha\in\gt h^*$ is a positive root, then $\{e_\alpha,h_\alpha,f_\alpha\}\subset\gt g$ is an $\gt{sl}_2$-triple
associated  with $\alpha$.  

An automorphism $\sigma\in{\rm Aut}(\gt g)$ extends to $\gt g[t^{-1}]$ by setting 
$\sigma(t^{-1})=t^{-1}$. In this context, $\sigma$ stands also for the corresponding 
automorphism of $\cS(\gt g[t^{-1}])$. If we take a $\sigma$-invariant product 
$(\,\,,\,)$, then  $\sigma({\mathcal H}[-1])={\mathcal H}[-1]$. Therefore  
$\sigma$ acts on $\gt z(\wg)$. 

If $\sigma\in{\rm Aut}(\gt g)$ is of finite order 
$m$, then it  leads to a $\Z/m\Z$-grading $\gt g=\gt g_0\oplus\gt g_1\oplus\ldots\oplus\gt g_{m-1}$.  
In case $m=2$, we have $\gt g_1=\{\xi\in\gt g\mid \sigma(\xi)=-\xi\}$.

Throughout the paper, \\[.1ex]
($\diamond$) $\{x_i\}$  is an orthonormal basis of $\gt g$; \\[.1ex]
($\diamond$) in the sums $\sum_{i} x_i$ or 
 $\sum_{i,j} x_i x_j$, the ranges  are from $1$ to $\dim\gt g$ for $i$ and  for $j$; \\[.1ex]
($\diamond$) $\bar b=(b_1,b_2)\in\mathbb Z_{<0}^2$ and ${\mathcal H}[\bar b]$ stands for 
$\sum_{i} x_i[b_1]x_i[b_2]\in \U(\wg^-)$ and also for the symbol of this sum; \\[.1ex]  
($\diamond$) $G_\xi$ stands for the stabiliser of $\xi$ and it is always clear from the context, which $G$-action is considered,  $\gt g_\xi=\Lie G_\xi$\,; \\[.1ex] 
($\diamond$)  $\gt q$ stands for an arbitrary unspecified complex Lie algebra; \\[.1ex]
($\diamond$) if $\ca\subset\U(\gt q)$ is a subalgebra, then $\gr\!(\ca):=\left<\gr\!(x) \mid x\in \ca\right>_{\mK}\subset\cS(\gt q)$. 

\section{Explicit formulas in type {\sf A}} \label{sec-A}

In type {\sf A}, there are several explicit formulas for the Segal--Sugawara vectors 
\cite{ct:qs,cm:ho}, see also \cite[Sect.~7.1]{book2}. 
One of them actually uses symmetrisation. 
One can form the  matrix $E[-1]+\tau = (E_{ij}[-1])+\tau I_n$ with $E_{ij}\in\gt{gl}_n$ and calculate its 
column- and symmetrised determinants. Due to the fact that this matrix is {\it Manin}, see 
\cite[Def.~3.1.1,\, p.~48,\, Lemma~7.2]{book2}, the results are the same. The symmetrised version is more suitable for our purpose. The elements $S_j$ are 
coefficients  of $\tau^{n{-}j}$ in 
$$
\det_{\rm sym}(E[-1]+\tau)=\varpi(\Delta_n[-1])+\varpi(\tau \Delta_{n{-}1}[-1])+\ldots 
+\varpi(\tau^{n-2}\Delta_2[-1])+\varpi(\tau^{n-1}\Delta_1[-1])+\tau^n. 
$$
Assume the conventions that 
$$\tau x[a] - x[a]\tau =[\tau,x[a]]=\tau(x[a])=-ax[a{-}1]$$ and $\tau{\cdot}1=0$. This leads 
for example to  
$\tau x[-1]{\cdot}1= x[-2]$. Note that $\varpi$ acts on the summands of 
$\tau^{n-k} \Delta_{k}[-1]$ as on products of $n$ factors.  It permutes $\tau$ with elements of $\gt{gl}_n[-1]$.  

Let $\theta$ be  the Weyl involution  of $\gt g$. 
As is well-known, $\theta(\Delta_k)=(-1)^k\Delta_k$ in the case  of $\gt{gl}_n$. 
In particular, 
$\theta({\mathcal H}[-1])={\mathcal H}[-1]$ and 
$\theta$ acts on $\gt z(\wg)$.  
Hence one can always modify the Segal--Sugawar vectors in such a way that they become eigenvectors of 
$\theta$.  
The resulting simplified forms are 
\begin{align} 
& S_n=  \varpi(\Delta_n[-1])+\varpi(\tau^2\Delta_{n{-}2}[-1]){\cdot}1+{\ldots} +  \varpi(\tau^{2r}\Delta_{n-2r}[-1]){\cdot 1}+ {\ldots}  \label{S-det} \\ 
&  \qquad \qquad \qquad \qquad \qquad \qquad \qquad \qquad \qquad \qquad \qquad \qquad \qquad
+ \varpi(\tau^{2m-2}\Delta_2[-1]){\cdot}1,   \nonumber \\ 
& S_k=\varpi(\Delta_k[-1])+ \sum_{1\le r<k/2} \binom{n{-}k{+}2r}{2r}\varpi(\tau^{2r}\Delta_{k{-}2r}[-1]){\cdot}1. \label{S-sym-k}
\end{align}

We  will see  that there is a direct connection with the symmetrisaton and that one 
could have  used  $\omega$ instead of $\theta$ in order to simplify the formulas. 
The following two lemmas are valid for all 
Lie algebras.  

\begin{lm} \label{lm-sym-a}
Take $Y=y_1 \ldots  y_m \in \cS^m(\gt g)$ and $\bar a=(a_1,\ldots,a_m)\in \mathbb Z_{<0}^m$.
Then in $\U(\widehat{\gt g}^{-})$, we have  
$$
\mathbb Y[\bar a]:= \sum_{\sigma\in {\tt S}_m} y_{\sigma(1)}[a_1]\ldots y_{\sigma(m)}[a_m] 
=\varpi\left(\sum_{\sigma\in {\tt S}_m} y_{1}[a_{\sigma(1)}]\ldots y_{m}[a_{\sigma(m)}] \right) = m! \varpi(Y[\bar a]) 
$$
in the notation of Section~\ref{s-fs}. 
\end{lm}
\begin{proof}
It suffices to show that $\mathbb Y[\bar a]$ is invariant under all ${\tt t}_i=(i\,i{+}1)\in {\tt S}_m$ with  $1\le i<m$. 
For each $\sigma\in{\tt S}_m$, both monomials 
$$
y_{\sigma(1)}[a_1]{\ldots} y_{\sigma(i)}[a_{i}]y_{\sigma(i{+}1)}[a_{i{+}1}]{\ldots} y_{\sigma(m)}[a_m]  \ \ 
\text{ and } \ \ y_{\sigma(1)}[a_1]{\ldots} y_{\sigma(i{+}1)}[a_{i}]y_{\sigma(i)}[a_{i{+}1}]{\ldots} y_{\sigma(m)}[a_m]
$$
appear in $\mathbb Y[\bar a]$ with the same coefficient $1$. Let $s(\sigma,i)$ stand for their sum. Then 
$$
s(\sigma,i)  -  {\tt t}_i(s(\sigma,i))  = \ldots[y_{\sigma(i)}[a_i],y_{\sigma(i+1)}[a_{i+1}]]\ldots + 
 \ldots[y_{\sigma(i+1)}[a_{i}],y_{\sigma(i)}[a_{i+1}]]\ldots =0, 
$$ 
because $[y_{\sigma(i)}[a_i],y_{\sigma(i+1)}[a_{i+1}]]=[y_{\sigma(i)},y_{\sigma(i+1)}][a_i{+}
a_{i+1}]= - [y_{\sigma(i+1)}[a_{i}],y_{\sigma(i)}[a_{i+1}]]$. 
Since $\mathbb Y[\bar a]=\frac{1}{2}\sum_{\sigma} s(\sigma,i)$ for each $i$, we are done. 
\end{proof}

\begin{lm}\label{side-1}
Take $F\in \cS^m(\gt g)$ and $r\ge 1$. Then $\varpi(\tau^r F[-1]){\cdot}1$ is fully symmetrised and  therefore is
an eigenvector of $\omega$ corresponding to the eigenvalue $(-1)^m$.
\end{lm}
\begin{proof}
Notice that  $\varpi(\tau^r (F{+}F')[-1]){\cdot}1=\varpi(\tau^r F[-1]){\cdot}1+ \varpi(\tau^r F'[-1]){\cdot}1$
for any $F'\in\cS^m(\gt g)$. Hence we may assume that $F=y_1\ldots y_m$ with $y_j\in\gt g$. 
By the construction,  $\varpi(\tau^r F[-1]){\cdot}1$ is the sum of  terms 
$$
\frac{1}{(m+r)!} c(r,\bar a) \sum_{\sigma\in {\tt S}_m} y_{\sigma(1)}[a_1]\ldots y_{\sigma(m)}[a_m] 
\ \ \text{ with } \ \ c(r,\bar a)\in\mathbb N,
$$
taken over all vectors $\bar a=(a_1,\ldots,a_m)\in\Z_{<0}^m$  such that $\sum a_j=-(m+r)$. 
The scalars 
$c(r,\bar a)$ depend on  $(m,r,\bar a)$ in an elementary combinatorial way. 
Each summand here is a fully symmetrised element by Lemma~\ref{lm-sym-a}.
Hence the desired conclusion follows.  
\end{proof}

Let $z=\frac{1}{n}I_n$ be a central element of $\gt g=\gt{gl}_n$ and 
let 
$\tilde\Delta_k$ denote the restriction of $\Delta_k$ to $\gt{sl}_n$. Then 
\begin{equation} \label{z}
\Delta_k=\tilde\Delta_k+(n-k+1)z\tilde\Delta_{k-1}+\binom{n-k+2}{2}z^2\tilde\Delta_{k-2}+\ldots+\binom{n-2}{k-2}z^{k-2}\tilde\Delta_2+\binom{n}{k}z^k.
\end{equation} 
Fix $\gt h=\left<E_{ii} \mid 1\le i\le n\right>_{\mK}$. 
Let $\esi_i\in\gt h^*$ be a linear function such that 
$\esi_i(E_{jj})=\delta_{i,j}$. For $E_{ii}\in\gt g$, set $\tilde E_{ii}=E_{ii}-z$.  

\begin{prop}\label{works}
In type {\sf A}, we have 
$$
{\sf m}_{2r+1}(\tilde\Delta_k) = \frac{(2r)! (k-2r)!}{k!} \binom{n-k+2r}{2r}  \tilde\Delta_{k-2r}
$$
if $k-2r > 1$ and ${\sf m}(\tilde\Delta_3)={\sf m}(\Delta_3)=0$. 
\end{prop} 
\begin{proof}
Notice that the 
centre of $\gt g$  plays a very specific  r\^ole in {\sf m}, since 
$\ad(z)=0$. In particular, 
 ${\sf m}(\cS^3(\gt{gl}_n))={\sf m}(\cS^3(\gt{sl}_n))\subset \Lambda^2\gt{sl}_n$. 
Furthermore, 
$$
{\sf m}(\Delta_k)\in {\sf m}(\tilde\Delta_k)+ \Lambda^2\gt{sl}_n \otimes z\cS^{k-4}(\gt g), 
$$
where one can use the multiplication in either ${\rm End}(\gt{gl}_{n})$ or ${\rm End}(\gt{sl}_{n})$ for the definition of ${\sf m}$.
Therefore we can work either with 
$\gt{sl}_n$ or with $\gt{gl}_n$, whichever is more convenient. 

Suppose that $Y=E_{ij}E_{ls}E_{up}$ is a factor of a monomial  of $\Delta_k$. 
Then $$
i\not\in\{ l,u\},  \ j\not\in\{ s,p\}, \ l\ne u, \ \text{and } \ s\ne p.
$$
The $\gt h$-weight of $Y$ cannot be equal to 
either $2\esi_1-\esi_n-\esi_{n{-}1}$ or $\esi_1+\esi_2-2\esi_n$.   
These are the highest weights of the Cartan component of 
$\Lambda^2\gt{sl}_n$.  
Hence ${\sf m}(\Delta_k)\in (\gt{sl}_n{\otimes}\cS^{k{-}3}(\gt g))^{\gt g}$. 
The image in question 
a polynomial function on $(\gt{sl}_n{\oplus}\gt g)^*\cong \gt{sl}_n\oplus\gt g$ of degree $1$ in $\gt{sl}_n$ 
and degree $k{-}3$ in $\gt g$. 
Note that ${\sf m}(\Delta_3)$ is a $\gt{gl}_n$-invariant in  $\gt{sl}_n$ and is thereby zero. 
Suppose that $n\ge k>3$. 

Fortunately, $G(\gt{sl}_n{\oplus}\gt h)$ is a dense subset of $\gt{sl}_n{\oplus}\gt g$. 
We calculate the restriction $$\lf={\sf m}(\Delta_k)|_{\gt{sl}_n\oplus\gt h}$$  of ${\sf m}(\Delta_k)$ to $\gt{sl}_n\oplus\gt h$.
Write $\lf=\sum\limits_{\nu=1}^L \xi_\nu \otimes \lH_\nu$ with 
$\xi_\nu\in\gt{sl}_n$ and pairwise different monomials $\lH_\nu\in \cS^{k-3}(\gt h)$ in $\{E_{ii}\}$.
Since ${\sf m}(\Delta_k)$ is an element of $\gt h$-weight zero, $\xi_\nu\in\gt h$ for each $\nu$.
Thus one can say  that $\lf$ is an invariant of the Weyl group  $W(\gt g,\gt h)\cong{\tt S}_{n}$.
Without loss of generality assume that 
$\lH_1= y_4\ldots y_k$
with $y_s=E_{ss}$ for all $s\ge 4$. In order to understand $\lf$, it suffices to calculate 
$\xi_1$.  Let $F$ be the polynomial obtained from  $\Delta_3$ 
by 
setting $E_{ij}=0$ for all $(i,j)$ such that  $i$ or $j$ belongs to $\{4,\ldots,k\}$. 
Then $\xi_1=\frac{3!(k-3)!}{k!} {\sf m}(F)$. 

Take now $Y$ as above with $\{i,j,l,s,u,p\}=\{1,2,3\}$. Then 
\begin{itemize}
\item[$\diamond$] ${\sf m}(Y)(E_{14})=0$ if $i=j$ or $l=s$ or $u=p$, 
\item[$\diamond$] ${\sf m}(Y)(E_{14})=\frac{1}{6} E_{14}$ if $Y= E_{13} E_{32} E_{21}$, 
\item[$\diamond$] ${\sf m}(Y)(E_{14})=\frac{1}{6} E_{14}$ if $Y= E_{12} E_{23} E_{31}$. 
\end{itemize}
Besides, ${\sf m}(Y)(E_{vw})=0$ if $v,w\ge 4$. 
In the self-explanatory  notation, 
$\eta={\sf m}\big(\Delta_3^{(1,2,3)}\big)$ is an invariant of 
$(\gt{gl}_3\oplus\gt{gl}_{n-3})$ and $\eta$  acts on 
$\gt{gl}_3=\left< E_{vw} \mid 1\le v,w\le 3\right>_{\mK}$   as zero. 
Since $\Delta_3^{(1,2,3)}$ is a  linear combination of $Y=E_{ij}E_{ls}E_{up}$ with 
$\{i,j,l,s,u,p\}=\{1,2,3\}$,  the element 
 $\eta$ acts on $\gt g$ as $\frac{1}{3}(E_{11}+E_{22}+E_{33})$. 
This implies  that $\eta=\frac{1}{3}(\tilde E_{11}+\tilde E_{22}+\tilde E_{33})$. 
By the construction of $F$, we have now 
${\sf m}(F)=\binom{n-k+2}{2} \sum\limits_{l\not\in\{4,\ldots,k\}} \frac{1}{3} \tilde E_{ll}$ and hence 
$$
\xi_1\otimes\lH_1=\frac{3!(k-3)!}{k!}\frac{1}{3}\binom{n-k+2}{2}\big(\sum_{l\not\in\{4,\ldots,k\}} \tilde E_{ll}\big) \otimes E_{44} \ldots E_{kk}\,.
$$ 
From this one deduces that up to the scalar $\frac{k-2}{3}\frac{3!(k-3)!}{k!}\binom{n-k+2}{2}$, 
the restriction of ${\sf m}(\tilde\Delta_k)$ to $\gt{sl}_n\oplus\gt h$ 
coincides with the restriction $\tilde\Delta_{k-2}|_{\gt{sl}_n \oplus\gt h}$, where we regard 
$\tilde\Delta_{k-2}$ as an element of $\gt{sl}_n \otimes \cS^{k-3}(\gt g)$. 
In particular, ${\sf m}(\tilde\Delta_k)$ is a symmetric invariant. More explicitly, 
$$
{\sf m}(\tilde\Delta_k) = \frac{(k-2)}{3}\frac{3!(k-3)!}{k!} \binom{n-k+2}{2} \tilde\Delta_{k-2}  = \frac{2!(k-2)!}{k!} \binom{n-k+2}{2} \tilde\Delta_{k-2}. 
$$
Iterating  the map {\sf m}, we obtain the result. 
\end{proof}

\begin{rem}
Strictly speaking, ${\sf m}(\Delta_k)$ is not a symmetric invariant.  
\end{rem}

Now we can exhibit  formulas for Segal--Sugawar vectors of $t$-degree $k$ that are independent of $n$, i.e., these formulas are  valid for all 
$n\ge k$. First of all notice that in view of \eqref{z}, Formula~\eqref{S-sym-k} produces an element 
of $\gt z(\widehat{\gt{sl}_n})$ if we replace each $\Delta_{k-2r}$
with $\tilde\Delta_{k-2r}$. (This statement can be deduced from \eqref{S-det} as well.) 
Making use of Proposition~\ref{works}, one obtains that for $H=\tilde \Delta_k$, 
\begin{equation} \label{form-m}
\tilde S_{k-1}=\varpi(H[-1])+\sum\limits_{1\le r<(k-1)/2} \binom{k}{2r} \varpi(\tau^{2r} {\sf m}_{2r+1}(H)[-1]){\cdot}1   
\end{equation}
is a Segal--Sugawara vector. 

\section{Commutators and Poisson brackets} 
In this section, 
we prove that Formula~\eqref{form-m} is universal, i.e., that it is valid in all types,  providing 
${\sf m}_{2r+1}(H)$ is a symmetric invariant for each $r\ge 1$. 

For 
$F\in\cS^m(\gt g)$,
set $\mathbb X_{F[-1]}:=[{\mathcal H}[-1], \varpi(F)[-1]]$ .  Note that 
$$\omega({\mathcal H}[-1] \varpi(F)[-1])=(-1)^{m+2}\varpi(F)[-1]{\mathcal H}[-1].
$$
Hence 
$\omega$ multiplies $\mathbb X_{F[-1]}$ 
with $(-1)^{m{+}1}$.  This implies that the symbol of $\mathbb X_{F[-1]}$ has degree $m+1-2d$ with $d\ge 0$. 
Let ${\mathcal H}[-1]$ stand also for $\sum_{i} x_i[-1]x_i[-1]\in \cS^2(\gt g[-1])$. 
%
The fact that $\{{\mathcal H}[-1], F[-1]\}=0$ for $F\in\cS^m(\gt g)^{\gt g}$ is, of course, known. There is no harm however, in 
recalling a proof. 

\begin{lm} \label{0-Poisson}
Take two arbitrary $\gt g$-invariants $F,F'$ in $\cS(\gt g)$. 
Then $\{F[-1],F'[-1]\}=0$. 
\end{lm}
\begin{proof}
The Poisson bracket of two polynomial  functions can be calculated by
\begin{equation} \label{eq-poisson}
\{f_1,f_2\}(\gamma)=[d_\gamma f_1,d_\gamma f_2](\gamma) \ \ \text{for} \ \ \gamma\in(\wg^-)^*. 
\end{equation}
In case of $F[-1]$ and $F'[-1]$, the differentials $d_\gamma F[-1], d_\gamma F'[-1]$ at $\gamma$ depend only on 
the $(-1)$-part of $\gamma$. More explicitly, if $\gamma(x[-1])=\tilde\gamma(x)$ with 
$\tilde\gamma\in\gt g^*$, then $d_\gamma F[-1]=(d_{\tilde\gamma} F)[-1]$ and the same 
identity hods for 
$F'$. We have $d_{\tilde\gamma} F,d_{\tilde\gamma}F'\in (\gt g_\gamma)^{\gt g_\gamma}$, 
since $F$ and $F'$ are $\gt g$-invariants. Hence
$[d_{\tilde\gamma} F,d_{\tilde\gamma}F']=0$ 
and 
also  $[d_{\gamma}F[-1],d_{\gamma}F'[-1]]=0$ for any $\gamma\in(\wg^-)^*$. 
\end{proof}

If $[\gt g,\gt g]$ is not simple, then the following assumption on the choice of 
the scalar product on $\gt g$ is made in order to simplify the calculations. 
\begin{itemize}
\item[($\blackdiamond$)] There is a constant $C\in\mK$ such that 
$\sum\limits_{i=1}^{\dim\gt g} \ad(x_i)^2(\xi)=C\xi$ for each $\xi\in[\gt g,\gt g]$. 
\end{itemize}
The constant $C$
depends on the scalar product in question. 

From now on, assume that $\gt g$ is semisimple. 
As the next step we examine the difference 
$$
X_{F[-1]}:=  \mathbb X_{F[-1]}-\varpi(\{{\mathcal H}[-1], F[-1]\}) 
$$
and more general expressions, where the commutator is taken with 
${\mathcal H}[\bar b]$. 
Our goal is to present $\mathbb X_{F[-1]}$ in the form~\eqref{sim-alg}.
For any ${\mathcal F}_1,{\mathcal F}_2\in\U(\wg^-)$, the symbol 
$\gr\!([{\mathcal F}_1,{\mathcal F}_2])$ is equal to the Poisson bracket  
$\{\gr\!({\mathcal F}_1),\gr\!({\mathcal F}_2)\}$ if this Poisson bracket is non-zero. 


\subsection{Commutators, the first approximation} 
Fix $m\ge 1$. Then set $\widecheck{j}=m-j$ for $1\le j<m$. 

\begin{lm} \label{lm-com-h}
For $Y=\hat y_1\ldots \hat y_m\in\cS(\wg^-)$ with $\hat y_j=y_j[a_j]$, set 
$$X_Y=X_{Y,\bar b}=[{\mathcal H}[b_1,b_2], \varpi(Y)]- \varpi(\{{\mathcal H}[b_1,b_2], Y\}).$$ 
Then 
\begin{align*}
& X_Y= \sum\limits_{\sigma\in {\tt S}_{m-1}, i,u,l;j<p  } 
({\sf m}(y_{\sigma(p)} \otimes  y_{l} \otimes y_{\sigma(j)}) (x_i), x_{u}) \times \\ 
&  (c_{2,3}(j,p)  
\hat y_{\sigma(1)}\ldots \hat y_{\sigma(j-1)} x_i[b_1{+}a_{\sigma(j)}] \hat y_{\sigma(j+1)}\ldots \hat y_{\sigma(p-1)} x_{u}[b_2{+}a_{l}+a_{\sigma(p)}] \hat y_{\sigma(p+1)} \ldots \hat y_{\sigma(m{-}1)} +  \\
&  c_{2,3}(j,p)  
\hat y_{\sigma(1)}\ldots \hat y_{\sigma(j-1)} x_i[b_2{+}a_{\sigma(j)}] \hat y_{\sigma(j+1)}\ldots \hat y_{\sigma(p-1)} x_{u}[b_1{+}a_{l}+a_{\sigma(p)}] \hat y_{\sigma(p+1)} \ldots \hat y_{\sigma(m{-}1)} +  \\
& 
 -  c_{3,2}(j,p) \hat y_{\sigma(1)}\ldots \hat y_{\sigma(j-1)} x_i[b_1{+}a_{\sigma(j)}{+}a_{l}] \hat y_{\sigma(j+1)}\ldots \hat y_{\sigma(p-1)} x_{u}[b_2{+}a_{\sigma(p)}] \hat y_{\sigma(p+1)} \ldots \hat y_{\sigma(m{-}1)}  +  \\
  & 
 - c_{3,2}(j,p) \hat y_{\sigma(1)}\ldots \hat y_{\sigma(j-1)} x_i[b_2{+}a_{\sigma(j)}{+}a_{l}] \hat y_{\sigma(j+1)}\ldots \hat y_{\sigma(p-1)} x_{u}[b_1{+}a_{\sigma(p)}] \hat y_{\sigma(p+1)} \ldots \hat y_{\sigma(m{-}1)}  ), 
\end{align*}
where $1\le l\le m$, $1\le j<p\le m-1$, and 
$\sigma$ is a map from $\{1,\ldots,m-1\}$ 
to $\{1,\ldots,m\}\setminus\{l\}$. 
%
%
The constants $c_{2,3}(j,p)$, $c_{3,2}(j,p)\in\mathbb Q$ do not depend on $Y$,
they depend only on $m$. Besides, $$c_{2,3}(j,p)=c_{3,2}(\widecheck{p},\widecheck{j}),$$ 
$c_{2,3}(j,p)\le 0$ for all $j<p$, and $c_{2,3}(j,p)< 0$ if in addition  \,$\widecheck{j}\ge p$. 
\end{lm}
\begin{proof}
Set  $\hat Y=\hat y_1\ldots\hat y_m\in\U(t^{-1}\gt g[t^{-1}])$. 
Let $\hat x_i^{(1)}$ stand for $x_i[b_1]$ and  $\hat x_i^{(2)}$ for $x_i[b_2]$. 
Then 
$$
[{\mathcal H}[b_1,b_2],\hat Y]\!\!=\!\! \sum\limits_{j=1,i=1}^{j=m,i=\dim\gt g} (\hat y_1\ldots \hat y_{j{-}1} \hat x_i^{(1)} [\hat x_i^{(2)},\hat y_j] \hat y_{j{+}1}\ldots \hat y_m+ \hat y_1\ldots \hat y_{j{-}1}  [\hat x_i^{(1)},\hat y_j] \hat x_i^{(2)} \hat y_{j{+}1}\ldots\hat y_m).
$$
Furthermore, $[{\mathcal H}[b_1,b_2],\varpi(Y)]= \frac{1}{m!} \sum\limits_{\sigma\in{\tt S}_m} 
 [{\mathcal H}[b_1,b_2],\sigma(\hat Y)]$. 
 One can say that here $\sigma$ defines a sequence of factors $\hat y_j$. 
The symmetrisation 
of $P_Y=\{{\mathcal H}[b_1,b_2], Y\}$ resembles this sum, but  with a rather significant difference: 
the factor 
$\hat x_i^{(\upsilon)}$, which is not involved in 
$[\hat x_i^{(\nu)},\hat y_{\sigma(j)}]$, 
does not have to stay next to 
$[\hat x_i^{(\nu)},\hat y_{\sigma(j)}]$. Therefore the idea is to modify each term of $\varpi(P_Y)$ in such a way that 
$\hat x_i^{(\upsilon)}$ gets back to its place as in $[{\mathcal H}[b_1,b_2],\varpi(Y)]$. 
In this process, other commutators $\pm[\hat x_i^{(\upsilon)},\hat y_l]$ will appear. 
It is convenient to arrange the summands of $\varpi(P_{Y})$ into groups according to the sequence of factors 
$\hat y_j$, including also the one appearing in $[\hat x_i^{(\nu)},\,\,]$. 
Then 
$X_Y=\frac{1}{m!} \sum_{\sigma\in{\tt S}_m} X_{\sigma(\hat Y)}$, where, for  instance,  
$[{\mathcal H}[b_1,b_2],\hat Y]$ is a summand of $X_{\hat Y}$.

At the final step in half of the cases, one has to commute $\hat x_i^{(\upsilon)}$ with $[\hat x_i^{(\nu)},\hat y_j]$. %
For the expression $(m+1)X_{\hat Y}$ and  a place $j$, this leads to the factor  
$$
j[\hat x_i^{(2)},[\hat x_i^{(1)},\hat y_j]] + (m-j+1) [[\hat x_i^{(2)},\hat y_j], \hat x_i^{(1)}]\, .
$$
Observe that 
$$
[\hat x_i^{(2)},[\hat x_i^{(1)},\hat y_j]]+[[\hat x_i^{(2)},\hat y_j],\hat x_i^{(1)}]=[[\hat x_i^{(2)},\hat x_i^{(1)}],\hat y_j]=0. 
$$
Assume that $X_Y$ is presented in the form \eqref{sim-alg}.  
Then in the part of ``degree" $m+1$,  the expressions containing such double commutators  annihilate each other if  we add over pairs of places $(j,m-j+1)$. 
In $X_Y$, the expressions having 
$[[\hat x_i^{(\upsilon)},[\hat x_i^{(\nu)},\hat y_j]],\hat y_l]$ 
as a factor remain for the moment. 
In view of ($\blackdiamond$), 
$$
\sum\limits_i \ad(\hat x_i^{(1)})\ad(\hat x_i^{(2)})(y_j[a_j])=\sum\limits_i \ad(\hat x_i^{(2)})\ad(\hat x_i^{(1)})(y_j[a_j])=Cy_j[a_j+b],
$$
where $b=b_1+b_2$. 
In $X_{\hat Y}$, this results in 
\begin{equation} \label{com-C}
\sum\limits_{i} [[\hat x_i^{(\upsilon)},[\hat x_i^{(\nu)},\hat y_j]],\hat y_l] = 
[Cy_j[a_j+b],\hat y_l]=C[y_j,y_l][a_j+a_l+b]\,. 
\end{equation}
In $X_{{\sigma}(\hat Y)}$ with $\sigma=(j\,l)$, one finds a similar factor 
with a different sign, namely,  $C[y_l,y_j][a_l{+}a_j{+}b]$.
This proves that the expressions  containing   $[\hat x_i^{(\upsilon)},[\hat x_i^{(\nu)},\hat y_j]]$ 
as a factor 
have no contribution to $X_Y$.  

In $(m+1)X_{\hat Y}$,   a term of the form 
$$
X(\id,i,\upsilon,\nu;j,l)=\hat y_1\ldots \hat y_{j{-}1} [\hat x_i^{(\upsilon)},\hat y_j] \hat y_{j{+}1} \ldots \hat y_{l{-}1} [\hat x_i^{(\nu)},\hat y_l] \hat y_{l{+}1} \ldots \hat y_m
$$ 
appears $j$ times with the coefficient $1$ if we have taken the commutator $[\hat x_i^{(\nu)},\hat y_l]$  and 
are moving $\hat x_i^{(\upsilon)}$ to the right, it also appears $(m-l+1)$ times with the coefficient $(-1)$ if the commutator was taken 
with $\hat y_j$ and $\hat x_i^{(\nu)}$ is moving to the left. 
Now we are able to state that 
$$
X_{\sigma(\hat Y)}= \frac{1}{m+1} \sum_{i,\upsilon,\nu;j<l} (j+l-m-1) X(\sigma,i,\upsilon,\nu;j,l). 
$$
Set $j'=m-j+1$. Then $j+j'=m+1$. Assume that $l\ne j'$ and $l>j$. 
For any $\sigma\in{\tt S}_m$, 
the symbol of  $X(\sigma,i,\upsilon,\nu;j,l)$ is the same as the symbol of $X(\sigma,i,\upsilon,\nu;l',j')$ 
and these two expressions appear in $X_{\sigma(\hat Y)}$ with opposite coefficients.
In order to annihilate them in a uniform way, 
we move the commutator that is closer to the middle until the expression obtains a sort of a central symmetry,  
see Example~\ref{m6} below. 
In case $\sigma=\id$, the arising expressions look similar to 
$$
\hat y_1\ldots \hat y_{j{-}1} [\hat x_i^{(\upsilon)},\hat y_j] \hat y_{j{+}1}  \ldots   \hat y_{l{-}1} \hat y_{l{+}1} \ldots  \hat y_{p{-}1} [[\hat x_i^{(\nu)},\hat y_l],\hat y_p] \hat y_{p{+}1} \ldots \hat y_m.
$$
First we deal with these expressions  ``qualitative" and after that describe the coefficients. 

Observe that for $y\in\gt g$ and $a\in\Z_{<0}$, we have  
$y[a]=\sum_{i} (x_i,y) x_i[a]$.  
Assume for simplicity that  $\{j,p,l\}=\{1,2,3\}$, 
disregard for the moment the other factors, and ignore the $t$-degrees of the elements. 
Consider the sum 
\begin{align}
& \sum\limits_i [x_i,y_1] [y_3,[y_2,x_i]] = \sum\limits_{i,j,u} ([x_i,y_1],x_j)x_j ([y_3,[y_2,x_i]],x_u)x_u=   \label{com}
 \\
& \quad = \sum\limits_{i,j,u} (x_i,[ y_1,x_j] )x_j (x_i, \ad(y_2)\ad(y_3)(x_u))x_u = \nonumber  \\
& \quad \quad = \sum\limits_{i,j,u}  (( \ad(y_2)\ad(y_3)(x_u), x_i) x_i, [y_1,x_j]) x_j x_u =   \nonumber \\
& \quad \quad  \quad = \sum\limits_{j,u} (\ad(y_2)\ad(y_3)(x_u),[y_1,x_j])x_jx_u=\sum\limits_{j,u}
(\ad(y_3)\ad(y_2)\ad(y_1)(x_j), x_u)x_j x_u.  \nonumber
\end{align}
Note that $\ad(y_3)\ad(y_2)\ad(y_1)(x_j)={\sf m}(y_3\otimes y_2\otimes y_1)(x_j)$. 
If we recall the $t$-degrees, then the product $x_j x_u$ has to be replaced with 
$x_j[b_\upsilon{+}a_1] x_u[b_\nu{+}a_2{+}a_3]$ in \eqref{com}.  The other factors $\hat y_w$ do no interfere with the 
transformations in \eqref{com}. 

In the process of changing the sequence of factors of  
$X(\id,i,\upsilon,\nu;j,l)$ with $j<l<j'$, 
the term 
$...[\hat x_i^{(\upsilon)},\hat y_j]...[[\hat x_i^{(\nu)},\hat y_l],\hat y_p]...$
appears with the 
negative coefficient $\frac{1}{(m+1)!}(j{+}l{-}m{-}1)$ as long as $l<p\le j'$. 
This shows that indeed the constants $c_{2,3}(j,p)$ 
do not depend on $Y$, they depend only on $m$. Moreover, 
$c_{2,3}(j,p)=0$ if $p>\widecheck{j}$ and  $c_{2,3}(j,p)<0$ if $p\le \widecheck{j}$. 

The symmetry 
$c_{2,3}(j,p)=c_{3,2}(\widecheck{p},\widecheck{j})$ is justified by  the fact that $\omega(X_Y)=(-1)^{m-1}X_Y$. 
A more direct way to see this, is to notice that if $[\hat x_i^{(\nu)},\hat y_l]$ jumps from a place  
$v$ to $j'$, then 
$j<v<j'$ and there is a term with the apposite coefficient, where  $[\hat x_i^{(\nu)},\hat y_l]$ jumps from 
$v'$ to $j$. 
The first type of moves produces 
$$
({\rm coeff.}) ({\sf m}(y_{\sigma(p)}\otimes y_l\otimes y_{\sigma(j)})(x_i),x_u) \ldots x_i[b_{\upsilon}{+}a_{\sigma(j)}] \ldots x_u[b_\nu{+}a_l{+}a_{\sigma(p)}]\ldots 
$$ 
and the second
$$
  ({\rm the\ same\ coeff.}) (x_u,{\sf m}(y_{\sigma(p')}\otimes y_l\otimes y_{\sigma(j')})(x_i)) \ldots x_u[b_\nu{+}a_l{+}a_{\sigma(p')}] \ldots x_i[b_{\upsilon}{+}a_{\sigma(j')}]\ldots \ .
$$
We have $(x_i,{\sf m}(y_{\sigma(j)}\otimes y_l\otimes y_{\sigma(p)})(x_u))= - ({\sf m}(y_{\sigma(p)}\otimes y_l\otimes y_{\sigma(j)})(x_i),x_u)$ and the scalar product $(\,\,,\,)$ is symmetric.  
These facts  confirm the symmetry of the constants 
and justifies the  minus signs in front of $c_{3,2}(j,p)$ in the answer. 
\end{proof}

\begin{ex}\label{m6}
Consider the case $m=6$. One obtains that 
\begin{align*}
&  X_Y= \frac{1}{7!}\!\sum\limits_{\sigma\in {\tt S}_6,i}\!\big( 4 X(\sigma,i,1,2;5,6) - 4 X(\sigma,i,1,2;1,2) + 
3 X(\sigma,i,1,2;4,6) - 3 X(\sigma,i,1,2;1,3) +\\
&  2 X(\sigma,i,1,2;3,6) - 2 X(\sigma,i,1,2;1,4) + 2 X(\sigma,i,1,2;4,5) - 2 X(\sigma,i,1,2;2,3)+ \\
&  X(\sigma,i,1,2;2,6) -  X(\sigma,i,1,2;1,5) +  X(\sigma,i,1,2;3,5) -  X(\sigma,i,1,2;2,4)\big) + \\ 
&  (\ \text{the similar expression for $(\upsilon,\nu)=(2,1)$}\ ).
\end{align*}
Take $\sigma=\id$. In the term $X(\sigma,i,1,2;5,6)$, the commutator 
$[\hat x_i^{(1)},\hat y_5]$ jumps to the first place producing  
commutators wit $\hat y_4, \hat y_3, \hat y_2,\hat y_1$. In the same manner,  
$[\hat x_i^{(2)},\hat y_2]$ jumps to the last place in  $X(\sigma,i,1,2;1,2)$. 
 In the term $X(\sigma,i,1,2;4,5)$, the commutator 
$[\hat x_i^{(1)},\hat y_4]$ jumps to the second place producing  
commutators with $\hat y_3$ and $\hat y_2$.
The non-zero 
constants $c^-(j,p):=-c_{2,3}(j,p)$ are 
$$
c^-(1,2)=\frac{4}{7!},\, c^-(1,3)=\frac{7}{7!},\,  c^-(1,4)=\frac{9}{7!},\, c^-(1,5)=\frac{10}{7!}, 
\ c^-(2,3)=\frac{2}{7!},\, c^-(2,4)=\frac{3}{7!}. 
$$
\end{ex}

Instead of the usual symmetrisation map, one can consider a {\it weighted} ``symmetrisation" or rather shuffle,  
where each permutation is added with a scalar coefficient assigned by a certain function $\Psi$. We will need only a very particular 
case of this construction. Let  $\Psi\!: {\tt S}_{k{+}2}\to \mathbb Q$ 
be a {\it weight function}, satisfying the following assumptions: 
\begin{itemize}
\item[({\tt A})] $\Psi(\sigma)$ depends only on $j=\sigma(k{+}1)$ and $p=\sigma(k{+}2)$, i.e., 
$\Psi(\sigma)=\Psi(j,p)$,
 \item[({\tt B})]  $\Psi(j,p)=\Psi(j',p')$ if $j'=k+3-j$.  
\end{itemize}
Then set 
$$
\varpi_{\rm wt} (y_1\ldots y_k\otimes y_{k+1} \otimes y_{k+2}) = 
\sum_{\sigma\in{\tt S}_{k+2}} \Psi(\sigma)\, y_{\sigma(1)} \otimes \ldots \otimes y_{\sigma(k+2)}
$$
for $y_j\in\gt q$. Let also $\varpi_{\rm wt}$ stand for the corresponding map 
from $\cS^k(\gt q)\otimes\gt q\otimes\gt q$ to $\U(\gt q)$. 
Condition ({\tt B}) guaranties that $\omega(\varpi_{\rm wt}(F))=(-1)^k \varpi_{\rm wt}(F)$ for each 
$F\in \cS^k(\gt q)\otimes\gt q\otimes\gt q$. 
In case  $\Psi(\sigma)=\frac{1}{(k+2)!}$, the map $\varpi_{\rm wt}$ coincides with $\varpi$. 
Keep in mind that each appearing $\varpi_{\rm wt}$ may have its own weight function.

Suppose that $Y\in\cS^m(\gt g)$, $\bar a\in\mathbb Z_{<0}^m$, and we want 
to merge them in order to obtain an element of $\cS^m(\wg^-)$. The only canonical way to do so is to 
replace $\bar a$ with the orbit ${\tt S}_m\bar a$, add over this orbit, and  divide by $|{\tt S}_m\bar a|$ as we have done in 
Section~\ref{s-fs}. The result is the $\bar a$-polarisation  $Y[\bar a]$ of $Y$. 
Set $$X_{Y[\bar a]}=[{\mathcal H[b_1,b_2], \varpi(Y[\bar a])]- \varpi(\{{\mathcal H}[b_1,b_2]}, Y[\bar a]\}).$$
For different numbers $u,v,l\in\{1,\ldots,m\}$, let $\bar a^{u,v,l}\in\mathbb Z_{<0}^{m-3}$ be the vector 
obtained from $\bar a$ by removing $a_u$, $a_v$, and $a_l$.  
 Let $\left<u,v,l\right>$ be 
a triple such that $l<v$ and $u\ne l,v$. 
Write ${\sf m}(Y)= \sum\limits_{w=1}^L \xi_w \otimes R_w$ with 
$\xi_w \in\Lambda^2\gt g$, $R_w\in\cS^{m-3}(\gt g)$.

\begin{prop} \label{prop-fs-c}
The element $X_{Y[\bar a]}$ is equal to
$$
\sum_{\left<u,v,l\right>,i,j,\upsilon,\nu,w} (\xi_w(x_i), x_j)
\varpi_{\rm wt}(R_w[\bar a^{u,v,l}] \otimes x_i[b_\upsilon+a_u] \otimes x_j[b_\nu+a_l+a_v]),
$$
where   $\Psi(j,p)= c_{2,3}(j,p)$ if $j<p$ and 
$\Psi(j,p)=c_{3,2}(p,j)$ if $j>p$ for the weight function $\Psi$. 
\end{prop}
\begin{proof}
Using the linearity, we may assume that $Y=y_1\ldots y_m$.  
The symmetry in $t$-degrees allows one to add the expressions appearing in the 
formulation of Lemma~\ref{lm-com-h} over  the triples 
$(y_{e}[a_{\sigma(p)}], y_{f}[a_l], e_g[a_{\sigma(j)}])$ with $\{e,f,g\}=\{\sigma(p),l,\sigma(j)\}$ 
while keeping $x_i[b_\upsilon{+}a_{\sigma(j)}]$, $x_u[b_\nu{+}a_{l}{+}a_{\sigma(p)}]$ and 
$x_i[b_\upsilon{+}a_{\sigma(j)}{+}a_l]$, $x_u[b_\nu{+}a_{\sigma(p)}]$ at their places. In this way 
the coefficient 
${\sf m}(y_{\sigma(p)}\otimes y_l \otimes y_{\sigma(j)})$ is replaced with 
${\sf m}(y_{\sigma(p)} y_l y_{\sigma(j)})$ and thereby $\xi_w$  with $1\le w\le L$ come into play. 
It remains to count the scalars 
and describe the weight function.

Suppose that $j<p$. Then 
$$
\frac{2}{m!} 3! c_{2,3}(j,p) =\frac{3! (m-3)!} {m!}  \frac{2}{(m-3)!} \Psi(j,p) 
$$
and thereby $\Psi(j,p)=c_{2,3}(j,p)$. Analogously, $\Psi(p,j)=c_{3,2}(j,p)$. 
\end{proof}

\begin{thm} \label{m-sym}
For $F\in\cS^m(\gt g)^{\gt g}$ with $m\ge 4$, the symmetrisation   
$\varpi(F)[-1]$ is an element of the Feigin--Frenkel centre if and only if 
${\sf m}(F)=0$. 
\end{thm}
\begin{proof}
According to \cite{r:un}, $\varpi(F)[-1]\in\gt z(\wg)$ if and only if it commutes with  ${\mathcal H}[-1]$. In view of
Lemma~\ref{0-Poisson}, this is the case if and only if  $X_{F[-1]}=0$. Lemma~\ref{lm-com-h}   
describes this element. It states that 
$c_{2,3}(j,p), c_{3,2}(j,p)\le 0$ and 
$c_{2,3}(j,p)<0$ if $p\le\widecheck{j}$ as well as  $c_{3,2}(j,p)<0$ if $\widecheck{p}\le j$. 
Since $\varpi(F)[-1]$ is fully symmetrised, we can use Proposition~\ref{prop-fs-c}.
It immediately implies that if ${\sf m}(F)=0$, then $X_{F[-1]}=0$. 

Suppose that ${\sf m}(F)\ne 0$. Write   ${\sf m}(F)= \sum\limits_{w=1}^L \xi_w \otimes R_w$ with 
$\xi_w \in\Lambda^2\gt g$ and linearly independent $R_w\in\cS^{m-3}(\gt g)$. 
If $\xi\in\Lambda^2\gt g$ is non-zero, then there are $i,j$ such that $(\xi(x_i),x_j)\ne 0$. 

Set ${\boldsymbol c}=\sum\limits_{j<p} (c_{2,3}(j,p)+c_{3,2}(j,p))$.  
According to Lemma~\ref{lm-com-h}, ${\boldsymbol c}<0$.  Hence
$$
m! \, {\boldsymbol c} \sum_{w=1;i,j}^{w=L} (\xi_w(x_i),x_j) x_i[-2]x_j[-3] R_w[-1]  
$$
is a non-zero element of $\cS(\wg^-)$. 
In view of the same lemma,  this expression is equal to $\gr\!(X_{F[-1]})$. 
Thus  $X_{F[-1]}\ne 0$. This completes the proof. 
\end{proof}

\begin{rem} 
If $\gt g$ is simple, then $\gt g^{\gt g}$ is equal to zero and $\cS^2(\gt g)^{\gt g}$ is spanned by ${\mathcal H}=\sum_{i} x_i^2$. By our convention, ${\sf m}(\cS^m(\gt g))=0$ if $m\le 2$. 
Furthermore, ${\sf m}(\cS^3(\gt g)^{\gt g})\subset (\Lambda^2\gt g)^{\gt g}=0$. Therefore Theorem~\ref{m-sym} holds for $m\le 3$ as well. 
\end{rem}

We will be  using weighted shuffles $\varpi_{\rm wt}$  of Poisson {\it  half-brackets}.  
If $Y=\hat y_1\ldots \hat y_m\in\cS(\wg^-)$, 
then
\begin{equation} \label{h-P}
 \varpi_{\rm wt}(Y, b_1,b_2):=  
 \sum\limits_{j=1,i=1}^{j=m,i=\dim\gt g} 
\varpi_{\rm wt}(Y/\hat y_j \otimes x_i[b_1] \otimes [x_i[b_2],\hat y_j] ). 
\end{equation}
Strictly speaking, here $\varpi_{\rm wt}$ is a linear  map from $\cS^m(\wg^{-})$ in 
$\U(\wg^{-})$ depending on 
$(b_1,b_2)$ 
and the choice of  a  weight 
function $\Psi$. 
The absence of ${\rm wt}$ in the lower index indicates that we are taking the usual symmetrisation. 

\subsection{Iterated shuffling} 

 Another general fact about Lie algebras $\gt q$ will be needed. Suppose that  $Y=y_1\ldots y_m\in\cS^m(\gt q)$ 
 and $x\in\gt q$. Write $Y=\frac{1}{m} \sum\limits_{1\le j\le m} y_j \otimes Y^{(j)}$ with 
$Y^{(j)}=Y/y_j$. Then 
\begin{equation}\label{P-sh}
\sum\limits_{1\le j\le m} [x,y_j] Y^{(j)} = \{x,Y\}. 
\end{equation}

\begin{prop} \label{h-fs}
Let ${\mathcal F}[\widecheck{\overline a}]=\varpi(F[\bar a])\in\U(\widehat{\gt g}^{-})$ be a fully symmetrised 
element corresponding to  $F\in\cS^m(\gt g)$ and a vector $\bar a=(a_1,\ldots,a_m)\in\mathbb Z_{<0}^m$. Suppose that ${\sf m}_{2r+1}(F)\in\cS^{m-2r}(\gt g)$ 
for all $m/2>r\ge 1$. Then \\
{\sf (i)} $X_{F[\bar a]}=[{\mathcal H}[\bar b],{\mathcal F}[\widecheck{\overline a}]] - \varpi(\{ {\mathcal H}[\bar b], F[\bar a]\})$ 
is a sum 
of weighted symmetrisations 
$$
\varpi_{\rm wt}({\sf m}(F)[\bar a^{l,j}], b_1{+}a_l,b_2{+}a_j), \ \ \varpi_{\rm wt}({\sf m}(F)[\bar a^{l,j}], 
b_2{+}a_l,b_1{+}a_j),
$$
where $l\ne j$ and  $\bar a^{l,j}$ is obtained from $\bar a$ by removing $a_l$ and $a_j$; 
\\[.3ex] 
{\sf (ii)} for every weight function $\Psi$, there is a constant $c\in\mathbb Q$, which is independent of $F$, 
such that ${\mathcal P}_{F[\bar a]}=\varpi_{\rm wt}(F[\bar a], b_1,b_2) - c\varpi(F[\bar a],b_1,b_2)$
is a sum of 
$\varpi_{\rm wt}({\sf m}(F)[\bar a^{(1)}], b_\upsilon{+}\gamma_\upsilon, b_\nu{+}\gamma_\nu)$   
with  different weight functions, whereby   
$\bar a^{(1)}$ is a subvector of $\bar a$ with $m{-}2$ entries and $\bar\gamma\in\mathbb Z_{<0}^2$ is constructed from  the complement $\bar a\setminus\bar a^{(1)}$ of $\bar a^{(1)}$; \\[.3ex]
{\sf (iii)} $\mathbb X_{F[a]}=[{\mathcal H}[b_1,b_2],{\mathcal F}[\widecheck{\overline a}]]$ is a sum of  
$$
C(\bar a^{(r)},\bar\gamma)\varpi({\sf m}_{2r+1}(F)[\bar a^{(r)}],  b_\upsilon{+}\gamma_\upsilon, b_\nu{+}\gamma_{\nu}),
$$
where  $0\le r<m/2$,  $\bar a^{(r)}$ is a subvector of $\bar a$ with $m{-}2r$ entries,  $\bar\gamma\in\mathbb Z_{<0}^2$ is constructed from  $\bar a\setminus\bar a^{(r)}$, and the coeffients $C(\bar a^{(r)},\bar\gamma)\in\mathbb Q$
are independent of $F$. 
\end{prop}
\begin{proof}
Since we are working with a fully symmetrised element, 
Proposition~\ref{prop-fs-c} applies. 
In the same notation, write ${\sf m}(F)=\sum\limits_{w=1}^L \xi_w \otimes R_w$. 
By our assumptions, ${\sf m}(F)\in\cS^{m-2}(\gt g)$. 
In particular, $\xi_w\in\gt g$ for each $w$. Observe that 
$$\sum_{i,j} (\xi_w(x_i),x_j) x_i[b_\upsilon] x_j[b_\nu] =  \sum_{i,j}  x_i[b_\upsilon] ([\xi_w,x_i],x_j) x_j[b_\nu] =  \sum_{i}  x_i[b_\nu] [\xi_w,x_i[b_\upsilon]].
$$ 
Thereby part {\sf (i)} follows from Proposition~\ref{prop-fs-c} in view of \eqref{P-sh}.  

\vskip0.5ex

{\sf (ii)} \ Note that $\omega({\mathcal P}_{F[\bar a]})=(-1)^{m+1} {\mathcal P}_{F[\bar a]}$, because of the assumption {\tt (B)} imposed on all 
weight functions.  By the construction, the image of  $\varpi_{\rm wt}(F[\bar a], b_1,b_2)$ in 
$\cS^{m+1}(\wg^-)$ is equal to $c \sum_{i}  \{x_i[b_2],F[\bar a]\}x_i[b_1]$ 
for some $c\in\mathbb Q$. 
This constant $c$ depends only on the weight function $\Psi$.

Let us symmetrise ${\mathcal P}_{F[\bar a]}$ by changing the sequence of factors 
in its summands. Note that there is no need to commute factors 
$\hat y_j=y_j[a_{\sigma(j)}]$ and $\hat y_l=y_l[a_{\sigma(l)}]$ of 
summands of 
${\mathcal F}[\widecheck{\overline a}]$, since ${\mathcal P}_{F[\bar a]}$ is 
symmetric  in the $\hat y_{p}$'s. 
There is no sense in commuting $\hat x_i^{(1)}$ and $\hat x_i^{(2)}$ either. 
Because of the antipode symmetry, the expressions of ``degree" $m$  annihilate each other in ${\mathcal P}_{F[\bar a]}$. Now we have a sum of terms of  degree $m-1$ and each non-zero summand must contain 
certain  factors 
according to one of the types listed below: 
\begin{itemize}
\item[{\sf (1)}] $[\hat x_i^{(\nu)},y_j]$ and $[[\hat x_i^{(\upsilon)},\hat y_l],\hat y_p]$,  
\item[{\sf (2)}] $[\hat y_p,[\hat y_l,[\hat y_j,\hat x_i^{(\upsilon)}]]]=\ad(y_p)\ad(y_l)\ad(y_j)(x_i[b_\upsilon{+}a_{\sigma(p)}{+}a_{\sigma(l)}{+}a_{\sigma(j)}])$ 
and $\hat x_i^{(\nu)}$, 
\item[{\sf (3)}] $[[\hat y_p,\hat x_i^{(\upsilon)}], [\hat y_l,\hat x_i^{(\nu)}]]=[[y_p,x_i], [y_l,x_i]][a_{\sigma(p)}{+} a_{\sigma(l)}{+}b_1{+}b_2]$, 
\item[{\sf (4)}] $[\hat y_p,[\hat x_i^{(\upsilon)}, [\hat x_i^{(\nu)},\hat y_j]]]$, 
\item[{\sf (5)}] $[\hat x_i^{(\upsilon)},[ \hat y_p, [\hat y_j,\hat x_i^{(\nu)}]]]=[[\hat x_i^{(\upsilon)},\hat y_p], [\hat y_j,\hat x_i^{(\nu)}]] - 
[\hat y_p,  [\hat x_i^{(\upsilon)}, [\hat x_i^{(\nu)},\hat y_j]]]$. 
\end{itemize}
The terms of type  {\sf (4)} disappear if we add over all $i$ and permute $p$ and $j$, because of the properties of 
${\mathcal H}[\bar b]$, cf. \eqref{com-C}. The terms of type  {\sf (3)} disappear if we permute  $l$ and $p$. 
Therefore the terms of type  {\sf (5)} disappear as well.  

One can deal with the terms of types  {\sf (1)} and  {\sf (2)} in the same way as in   
Lemma~\ref{lm-com-h} and Proposition~\ref{prop-fs-c}. They lead to 
$\gamma=(a_j,a_l)$ and $\gamma=(a_j{+}a_l,0)$ as well as $\gamma= (0,a_j{+}a_l)$. 
Note that the commutators of type {\sf (2)} are easier to understand, since there is no need
to permute the $t$-degrees, and at the same time they give rise to half-brackets. 

\vskip0.5ex

{\sf (iii)} \  We are presenting $\mathbb X_{F[\bar a]}$ in the form~\eqref{sim-alg} and can state at once that it
has terms of degrees $m+1-2d$ only. Note that $[{\mathcal H}[\bar b],{\mathcal F}[\widecheck{\overline a}]]$ 
can be viewed  as a weighted symmetrisation  $\varpi_{\rm wt}(F[\bar a], b_1, b_2)$ 
if we choose $\Psi(j,j+1)=\Psi(j+1,j)=1$ and $\Psi(j,l)=0$ in case $|l-j|>1$. 
The term of degree $m+1$ is the Poisson bracket $\{{\mathcal H}[\bar b],F[\bar a]\}$. Here 
$r=0$ and $\bar\gamma=0$. 
In degree $m-1$, we obtain images in $\cS^{m-1}(\wg^-)$ 
of the weighted symmetrisations described in part~{\sf (i)}. 
Further terms, which are   of degrees $m-3,m-5,m-7$, and so on, 
are described by the iterated application of part~{\sf (ii)}. 
At all steps, we obtain combinatorially defined rational coefficients, which are independent of $F$. 
\end{proof}

\begin{ex}\label{k4}
Suppose that $F\in\cS^4(\gt g)^{\gt g}$ and that $\gt g$ is simple. 
Here ${\sf m}(F)\in (\Lambda^2\gt g{\otimes}\gt g)^{\gt g}$ and  $\dim(\Lambda^2\gt g{\otimes}\gt g)^{\gt g}=1$. 
This subspace is spanned  by ${\mathcal H}=\sum x_i^2$. Hence ${\sf m}(F)={\mathcal H}$ up to a scalar. 
As we will proof in Section~\ref{sec-hP}, $S=\varpi(F[-1])+6\varpi(\tau^2{\sf m}(F)[-1]){\cdot}1$ is a
Segal--Sugawara vector. Making use of the fact that $\tau^2({\mathcal H}[-1])\in\gt z(\wg)$, 
one can write $S$ as a sum $\varpi(F[-1])+B{\mathcal H}[-2]$
for some scalar $B\in\mathbb\mK$. 

The only possible vector $\bar\gamma$ that can appear in Proposition~\ref{h-fs}{\sf (iii)}
is $(-1,-1)$. 
Therefore 
the commutator $[{\mathcal H}[-1],\varpi(F[-1])]$ 
is equal to $B\varpi(\{{\mathcal H}[-2],{\mathcal H}[-1]\})=B[{\mathcal H}[-2],{\mathcal H}[-1]]$. 
In the orthonormal basis $\{x_i\}$, we have 
\begin{equation} \label{-3}
\{{\mathcal H}[-2],{\mathcal H}[-1]\}=4 \sum\limits_{i,j,s} ([x_i,x_j],x_s) x_s[-3]x_i[-2]x_j[-1]. 
\end{equation} 
\end{ex}

\subsection{Poisson (half-)brackets} \label{sec-hP}

Suppose that $\hat Y=\hat y_1\ldots \hat y_m\in\cS^m(\wg^-)$ and  $\hat y_j=y_j[a_j]$.
Then 
$P_{\hat Y}:= \{ {\mathcal H}[\bar b], \hat Y\} =P_{\hat Y}(b_1,b_2)+P_{\hat Y}(b_2,b_1)$, where 
$$
P_{\hat Y}(b_\upsilon,b_\nu)=  \sum_{j=1,i=1}^{j=m,i=\dim\gt g}  [x_i[b_\nu],\hat y_j] x_i[b_\upsilon]  \hat Y/\hat y_j = 
  \sum_{j,i,u} ([x_u,x_i],y_j)  x_u[b_\nu{+}a_j] x_i[b_\upsilon] \hat Y/\hat y_j. 
$$
Note that in case $b_\nu{+}a_j=b_\upsilon$, each summand $([x_u,x_i],y_j)  x_u[b_\nu{+}a_j] x_i[b_\upsilon] \hat Y/\hat y_j$ is annihilated by $([x_i,x_u],y_j)  x_i[b_\nu{+}a_j] x_u[b_\upsilon] \hat Y/\hat y_j$. Hence
\begin{equation} \label{P2}
P_{\hat Y}(b_1,b_2)=\sum_{i,u;\,j: a_j\ne b_1-b_2} ([x_i,x_u],y_j)  x_i[b_2{+}a_j] x_u[b_1] \hat Y/\hat y_j.
\end{equation}
The  product $(\,\,,\,)$ extends to a non-degenerate $\gt g$-invariant  scalar product on 
$\cS(\wg^-)$. We will assume that $(\gt g[a],\gt g[d])=0$ for $a\ne d$, that 
$(x[a],y[a])=(x,y)$ for $x,y\in\gt g$,  and 
that 
$$
(\xi_1\ldots\xi_{\boldsymbol k},\eta_1\ldots\eta_{\boldsymbol m})=  \delta_{{\boldsymbol k},{\boldsymbol m}}
\sum_{\sigma\in{\tt S_{\boldsymbol k}}} (\xi_1,\eta_{\sigma(1)})\ldots (\xi_{\boldsymbol k},\eta_{\sigma({\boldsymbol k})})
$$
if $\xi_j,\eta_j\in\cS(\wg^-)$, ${\boldsymbol m}\ge {\boldsymbol k}$. 
Let ${\mathcal B}$ be a monomial basis of $\cS(\wg^-)$ consisting of the elements 
$\hat v_1\ldots \hat v_{\boldsymbol k}$, where  $\hat v_j=v_j[d_j]$ and $v_j\in\{x_i\}$. 
Then ${\mathcal B}$ 
is an orthogonal, but not an orthonormal basis. 
For instance, if $\Xi=x_1^{\gamma_1}\!\ldots x_{\boldsymbol k}^{\gamma_{\boldsymbol k}}$, then $(\Xi,\Xi)=\gamma_1!\ldots\gamma_{\boldsymbol k}!$.

Set $M:=m+1$, ${\mathcal B}(M):={\mathcal B}\cap \cS^M(\wg^-)$,  and 
write
$$
P_{\hat Y}(b_1,b_2) = \sum_{\mathbb V\in{\mathcal B}(M)} A(\mathbb V) \mathbb V 
\ \ \text{ with } \ \ A(\mathbb V)\in\mK,
$$
expressing each $\hat Y/\hat y_j$ in the  basis ${\mathcal B}$. 
 
\begin{lm} \label{lm-V}
If 
 $A(\mathbb V)\ne 0$ and $\mathbb V=\hat v_1\ldots\hat v_M$, then   
${\boldsymbol p}:=\{ p \mid d_p=b_1\}\ne\varnothing$. 
Furthermore, 
\begin{equation} \label{P-scalar}
A(\mathbb V) = 
\sum_{p\in{\boldsymbol p}, l\not\in{\boldsymbol p}} \big(\mathbb V,\mathbb V\big)^{-1} \big(\hat Y, \frac{\mathbb V}{\hat v_l\hat v_p}  [v_l,v_p][d_l{-}b_2]\big).
\end{equation}
\end{lm}
\begin{proof}
The first statement is clear, cf.~\eqref{P2}.  It remains to calculate the coefficient of $\mathbb V$ 
in $P_{\hat Y}(b_1,b_2)$. 
Pic a pair $(p,l)$ with $p\in{\boldsymbol p}$ and $l\not\in{\boldsymbol p}$. 
If we take into account  only those summands  of  $P_{\hat Y}(b_1,b_2)$, 
where $\mathbb V^{p,l}=\mathbb V/(\hat v_p\hat v_l)$ is a summand of 
$\hat Y/\hat y_j$ for some $j$, the factor  $\hat v_l$  is  $x_u[a_j+b_2]$, and
$\hat v_p$ appears as $x_i[b_1]$,  
then the coefficient is 
$$
\sum_{j=1}^{m} (\hat y_j,[v_l,v_p][d_l{-}b_2]) (\hat Y/\hat y_j, \mathbb V^{p,l}) (\mathbb V^{p,l},\mathbb V^{p,l})^{-1} = 
(\mathbb V^{p,l},\mathbb V^{p,l})^{-1} (\hat Y, \mathbb V^{p,l} [v_l,v_p][d_l{-}b_2]).
$$
If one  
adds these expressions over the pairs $(p,l)$, then 
certain instances may be counted 
more than once.  If $\hat v_p=\hat v_{p'}$ for
some $p'\ne p$ or $\hat v_l=\hat v_{l'}$ for some $l'\ne l$, then $(p,l')$ or $(p',l)$ has to be omitted from the summation.  
In other words, it is necessary to  divide the contribution of $(p,l)$  by 
the 
multiplicities $\gamma_p$ and $\gamma_l$ of $v_p[d_p]$, $v_l[d_l]$ in $\mathbb V$.  
Since $(\mathbb V,\mathbb V)=\gamma_p\gamma_l (\mathbb V^{p,l},\mathbb V^{p,l})$, the result follows. 
\end{proof}

The Poisson bracket $P_{\hat Y}$ is not multi-homogeneous w.r.t. 
$\wg^-=\bigoplus_{d\le -1} \gt g[d]$. 
If $b_1\ne b_2$, then in general  the  ``halves" of  $P_{\hat Y}$ have different multi-degrees and neither of them 
has to be multi-homogeneous. We need to split $P_{\hat Y}(b_1,b_2)$ into smaller pieces. 
For $\bar a\in\Z_{<0}^m$, set $\cS^{\bar a}(\wg^-)=\prod_{j=1}^m \gt g[a_j]\subset \cS(\wg^-)$. 

Let  $\bar\alpha=\{\alpha_1^{r_1},\ldots,\alpha_s^{r_s}\}$ be a multi-set such that 
$\alpha_i\ne\alpha_j$ for $i\ne j$, $\alpha_j\in \Z_{<0}$ for all $1\le j\le s$, 
$\sum_{j=1}^s r_j=M$, and  $r_j>0$ for all $j$. 
Set $\cS^{\bar\alpha}(\wg^-):=\prod_{j=1}^{s} \cS^{r_j}(\gt g[\alpha_j])$, 
${\mathcal B}(\bar\alpha):={\mathcal B}\cap \cS^{\bar\alpha}(\wg^-)$. 
Fix different $i,j\in\{1,\ldots,s\}$. Assume that a monomial 
$\mathbb V=\hat v_1\ldots \hat v_M\in{\mathcal B}(\bar\alpha)$ with 
$\hat v_l=v_l[d_l]$ is written in such a way that $d_l=\alpha_i$ 
for   $1\le l\le r_i$ and $d_l=\alpha_j$ for  $r_i<l\le r_i+r_j$.  
Finally suppose that $F\in\cS^m(\gt g)$. 
In this notation, set
\begin{align} 
& \mathbb W[F, \bar\alpha, (i,j)] :=\sum_{{\mathbb V}\in{\mathcal B}(\bar\alpha)} A(\mathbb V) \mathbb V     
\ \ \text{ with }  \nonumber  \\
 &  
\qquad \ A(\hat v_1\ldots \hat v_M)= \big(\mathbb V,\mathbb V\big)^{-1}\!\!\!\!\!\!\!\!\!\!\!\!\!\!\!\!\sum_{{\footnotesize \begin{array}{c}1\le l\le r_i, \\ r_i< p \le r_i+r_j \end{array}}}\!\!\!\!\!\!\!\!\!\!\!\!(F,[v_l,v_p] \prod_{u\ne p,l} v_u). 
 \label{W} 
\end{align} 
Clearly,  
$\mathbb W[F, \bar\alpha, (j,i)] = - \mathbb W[F, \bar\alpha, (i,j)]$.

\begin{prop} \label{universal} 
Let $F\in\cS^m(\gt g)^{\gt g}$ be fixed. Then the elements 
$\mathbb W[\bar\alpha, (i,j)]=\mathbb W[F,\bar\alpha, (i,j)]$ 
satisfy the following ``universal" relations:
$$
\sum_{j: j\ne i} \mathbb W[\bar\alpha, (i,j)] =0 \ \text{ for each } \ i\le s. 
$$
These relations  are independent of $F$. 
\end{prop}
\begin{proof}
Follow the notation of \eqref{W}. 
Note that for each $1\le l \le r_i$, 
$$
\big(F, \sum_{w: w\ne l} [v_l,v_w] \prod_{u: u\ne l,w} v_u\big) =\big(\{F,v_l\}, \prod_{w: w\ne l} v_w\big)=0.
$$
Of course, here we are  adding also over the pairs $(l,w)$ with $\hat v_w\in\gt g[\alpha_i]$
if $r_i>1$. 
However, $[v_l,v_w]=-[v_w,v_l]$ and hence the coefficient of $\mathbb V$ in 
$\sum_{j\ne i} \mathbb W[\bar\alpha, (i,j)]$ 
is equal to 
$$
\big(\mathbb V,\mathbb V\big)^{-1} \sum_{1\le l \le r_i} \big(F, \sum_{w: w\ne l} [v_l,v_w] \prod_{u: u\ne l,w} v_u\big)  = \sum_{1\le l \le r_i}  0 =0. 
$$
This completes the proof. 
\end{proof}

\begin{prop} \label{P-fs}
For $Y[\bar a]$ with  $Y=y_1{\ldots} y_m\in\cS^m(\gt g)$ and $\bar a\in\mathbb Z_{<0}^m$, see Section~\ref{s-fs} for the nota\-tion,  
the rescaled Poisson half-bracket  
$$
{\boldsymbol P}[\bar a]:=|{\tt S}_m \bar a| P_{Y[\bar a]}(b_1,b_2) = |{\tt S}_m \bar a| \sum_{u} \{x_u[b_2],Y[\bar a]\} x_u[b_1]
$$ 
is equal to 
the sum of \,$\mathbb W[Y, \bar\alpha, (i,j)]$ with $i<j$ over all multi-sets $\bar\alpha$ as above such that the 
multi-set $\{a_1,\ldots,a_m\}$ of entries of $\bar a$
 can be obtained from $\bar\alpha$ by removing one $\alpha_j=b_1$ and replacing one
$\alpha_i$ with 
$\alpha_i-b_2$.  
\end{prop}
\begin{proof}
For each multi-homogeneous component of ${\boldsymbol P}[\bar a]$, the multi-set of $t$-degrees 
is obtained from the entries of $\bar a$ by appending $b_1$ and replacing one $a_l$ with $a_l{+}b_2$. 
Moreover, here $b_1\ne a_l{+}b_2$, cf.~\eqref{P2}. 
This explains the restrictions on $\bar\alpha$. 

For $\bar\alpha$ and $(i,j)$ satisfying the assumptions of the proposition,  we have to compare
the coefficients $A(\mathbb V)$ of $\mathbb V\in{\mathcal B}(\bar\alpha)$ given by 
\eqref{W} and \eqref{P-scalar}. 
The key point here is the observation that 
$(Y[\bar a], \mathbb V)=|{\tt S}_m \bar a|^{-1} (Y,v_1\ldots v_m)$ for 
any $\mathbb V=\hat v_1\ldots\hat v_m \in {\mathcal B}\cap \cS^{\bar a}(\wg^-)$. 

In a more relevant setup, 
suppose that a summand 
$(y_{\sigma(1)}, [v_l,v_p]) \prod\limits_{w\ne 1,u\ne l,p} (y_{\sigma(w)},v_u)$ 
of the scalar product on the right hand side of 
\eqref{W} is non-zero for some $\sigma\in{\tt S}_m$ and some $l,p$.  
Then  there is exactly one choice, prescribed by $(d_1,\ldots,d_M)$, of 
the $t$-degrees for a monomial of $Y[\bar a]$ such that the corresponding summand 
$$
(y_{\sigma(1)}[\alpha_i-b_2],[v_l,v_p][d_l-b_2]) \prod_{w\ne 1,u\ne l,p} (y_{\sigma(w)}[d_u],\hat v_u)
$$
of the scalar product on 
the right hand side of \eqref{P-scalar} is non-zero as well. 
By our assumptions on the scalar product, these summands are equal. 
\end{proof}

\begin{thm} \label{thm-form}
Suppose that ${\sf m}_{2r+1}(H)$ with $H\in \cS^k(\gt g)^{\gt g}$ is a symmetric invariant for each $r\ge 1$. 
Then 
Formula~\eqref{form-m} provides a Segal--Sugawara vector $S$ associated with $H$. 
\end{thm}
\begin{proof}
Since ${\sf m}_{2l+1}(H)\in\cS(\gt g)$ for any $l$, we can say that  ${\sf m}_{2r+1}(H)={\sf m}^r(H)$. 
By Lemma~\ref{side-1}, each $\varpi(\tau^{2r}{\sf m}^{r}(H)[-1]){\cdot}1$ 
is a fully symmetrised element. It can be written as a sum of 
$\tilde c(r,\bar a)\varpi({\sf m}^{r}(H)[\bar a])$, where $\bar a\in\Z_{<0}^{k-2r}$ and the coefficients 
$\tilde c(r,\bar a)\in\mathbb Q$ depend only on $k$, $r$, and $\bar a$. 
The coefficients of \eqref{form-m} depend only on $k$ and $r$. Combining this observation with 
Propositions~\ref{h-fs}{\sf (iii)} and \ref{P-fs}, 
we obtain that 
$$
[{\mathcal H}[-1],S] = \sum C(r,\bar\alpha,i,j) \varpi(\mathbb W[{\sf m}_{2r+1}(H),\alpha_1^{r_1},\ldots,\alpha_d^{r_s}, (i,j)])
$$
where again the coefficients $C(r,\bar\alpha,i,j)\in\mathbb Q$ do not depend on $H$. 
For  a given degree $k$, one obtains a bunch of $(r,\bar\alpha)$, which depends 
only on $k$, and each appearing coefficient depends on $k$, $r$, $\bar\alpha$, and $(i,j)$. 
In type {\sf A}, for each $k\ge 2$, we find the invariant $\tilde\Delta_k$ such that the corresponding commutator 
$[{\mathcal H}[-1],\tilde S_{k-1}]$ vanishes, cf.~\eqref{form-m}.

For each $F\in\cS^m(\gt g)^{\gt g}$, 
the elements $\mathbb W[F,\bar\alpha, (i,j)]$ are linearly dependent. 
They satisfy the ``universal" relations, see Proposition~\ref{universal}.
 At the same time, for $m=k-2r$, the coefficients $C(r,\bar\alpha,i,j)$ provide a
relation  among $\mathbb W[\tilde\Delta_m,\bar\alpha, (i,j)]$. 
Our goal is to prove that  this relation holds for $\mathbb W[{\sf m}^{r}(H),\bar\alpha, (i,j)]$ as well. 
To this end, 
 it suffices to show that  the terms $\mathbb W[\tilde\Delta_m,\bar\alpha, (i,j)]=\mathbb W[(i,j)]$
with fixed $m$ and fixed $\bar\alpha$ do not satisfy any other linear relation, not generated by the ``universal"  ones. 

We consider the complete simple graph with $s$ vertices $1,\ldots,s$ and identify pairs 
$(i,j)$ with the corresponding (oriented) edges. Now one can say that a linear relation among 
the polynomials  $\mathbb W[(i,j)]$ is  given by  its coefficients on the edges. 
Note that if $s=2$, then $\mathbb W[F,\bar\alpha,(1,2)]=0$ for each $F\in\cS^m(\gt g)^{\gt g}$, cf. Example~\ref{ex-W}{\sf (i)} below. 
Therefore assume that $s\ge 3$. 

Suppose there is a relation and that the coefficient  of  $\mathbb W[(i,j)]$ is non-zero.
We work in the basis 
$$
\{E_{uv}[d],  (E_{11}+\ldots+ E_{ww} - wE_{(w+1)(w+1)})[d] \mid u\ne v,  d<0 \}.
$$ 
Choose $\hat y_1=E_{12}[\alpha_i]$, $\hat y_2=E_{21}[\alpha_j]$ and 
let all other factors $\hat y_l$ with $2<l \le M$ be elements of $t^{-1}\gt h[t^{-1}]$. Assume that 
$\hat y_3=(E_{11}{-}E_{22})[\alpha_p]$ with $p\ne i,j$ and that 
$(E_{11}-E_{22},y_l)=0$ for all $l>3$. 
Then the monomial 
$\hat y_1\ldots \hat y_M$ appears with a non-zero coefficient only in $\mathbb W[(i,j)]$, 
$\mathbb W[(i,p)]$, and $\mathbb W[(p,j)]$. 
This means that in the triangle $(i,j,p)$ at least one of the  edges $(i,p)$ and $(j,p)$ has 
a non-zero coefficient as well. 

We erase all edges with zero coefficients on them. Now the task is to modify the relation or, equivalently, 
the graph, by adding scalar multiplies of the universal relations in such a way that all edges disappear. 

If a vertex $l$ is connected with $j$, remove the edge $(j,l)$ using the universal relation ``at $l$". 
In this way $j$ becomes isolated. 
This means that there is no edge left. 
\end{proof}

\begin{ex}\label{ex-W} Keep the assumption $F\in\cS^m(\gt g)^{\gt g}$. \\[.1ex]
{\sf (i)} Suppose that $s=2$. 
Then 
$\mathbb W(F,\bar\alpha, (1,2))=0$ according to the universal relation. This provides a different proof 
of Lemma~\ref{0-Poisson}. 
Also  in the sase of  $\bar a=(-3,(-1)^{m-1})$, we have $\{{\mathcal H}[-1],F[\bar a]\}\in \gt g[-3]\gt g[-2] \cS^{m-1}(\gt g[-1])$.  \\[.2ex]
{\sf (ii)} Suppose now that $s=3$. Then 
$\mathbb W(F,\bar\alpha, (1,2))=-\mathbb W(F,\bar\alpha, (1,3))=\mathbb W(F,\bar\alpha, (2,3))$.
\end{ex}

\section{Type {\sf C}} \label{sec-C}
There is a very suitable  matrix realisation, where $\gt{sp}_{2n}\subset\gt{gl}_{2n}$ 
is the linear span of the elements   $F_{ij}$ with $i,j\in\{1,\dots,2n\}$ such that 
\begin{equation}
 F_{i j}=E_{i j}-\epsilon_i\ts\epsilon_j\ts E_{j' i'},
\end{equation}
with $i'=2n-i+1$ and $\epsilon_i=1$ for $i\le n$,  
$\epsilon_i=-1$ for $i> n$. Of course, $F_{ij}=\pm F_{j'i'}$. Set $\gt h=\left< F_{jj} \mid 1\le j\le n\right>_{\mK}$. 

The symmetric decomposition $\gt{gl}_{2n}=\gt{sp}_{2n}\oplus \gt p$ leads to 
explicit formulas for symmetric invariants of $\gt{sp}_{2n}$. 
One writes $E_{ij}=\frac{1}{2}F_{ij}+ \frac{1}{2}(E_{i j}+\epsilon_i\ts\epsilon_j\ts E_{j' i'})$,
expands the coefficients $\Delta_k$ of \eqref{d-gl} 
accordingly, and then sets $(E_{i j}+\epsilon_i\ts\epsilon_j\ts E_{j' i'})=0$. 
Up to the multiplication with $2^{k}$, this is equivalent to 
replacing each $E_{ij}$ with $F_{ij}$  in the formulas for $\Delta_{k}\in\cS(\gt{gl}_{2n})$. 
As is well-known, the restriction of $\Delta_{2k+1}$ to $\gt{sp}_{2n}$ is equal to zero for each $k$. 

Until the end of this section, $\Delta_{2k}$ stands for the symmetric invariant of $\gt g=\gt{sp}_{2n}$ that is equal to the sum
of the principal $(2k{\times}2k)$-minors of the matrix $(F_{ij})$.    

\begin{lm} \label{sp-w}
For each $k\ge 2$, we have ${\sf m}(\Delta_{2k})\in (\gt g\otimes \cS^{k-3}(\gt g))^{\gt g}$. 
\end{lm}
\begin{proof}
Note that  $\Lambda^2 \gt g=V(2\pi_1{+}\pi_2)\oplus \gt g$. 
In the standard notation,  $2\pi_1+\pi_2=3\esi_1+\esi_2$. 
Assume that $y_1y_2y_3$ is a factor of a summand of $\Delta_{2k}$ of weight 
$3\esi_1+\esi_2$ and $y_s\in\{F_{i j}\}$ for each $s$.   
Then 
\begin{itemize}
\item[$\rhd$]
either $F_{1(2n)}\in\{y_1,y_2,y_3\}$ and some  $y_j\ne F_{1(2n)}$  
lies in  the first row or the last column 
\item[$\rhd$]or 
all three elements $y_s$ must lie in the union of the first row and the last column. 
\end{itemize}
Each of the two possibilities  contradicts the 
definition of $\Delta_{2k}$. Thus, indeed ${\sf m}(\Delta_k)\in (\gt g \otimes \cS^{2k-3}(\gt g))^{\gt g}$. 
\end{proof}

\begin{lm} \label{sp-k6}
We have {\sf (i)} ${\sf m}(\Delta_{6})=\frac{(n-2)(2n-3)}{15}\Delta_4$  and \\[.2ex]
{\sf (ii)} ${\sf m}(\partial_{F_{11}}\Delta_4)= \frac{-1}{3} (2n-1)(2n-2) F_{11}$. 
\end{lm}
\begin{proof}
According to Lemma~\ref{sp-w}, ${\sf m}(\Delta_6)\in (\gt g \otimes \cS^{3}(\gt g))^{\gt g}$. 
Observe that $\cS^3(\gt g)$ contains  exactly two linearly independent copies of $\gt g$, one is equal to 
$\{\xi\Delta_2 \mid \xi\in\gt{sp}_{2n}\}$, the other is primitive. 
Therefore 
$$
(\gt g\otimes\cS^3(\gt g))^{\gt g}=\cS^4(\gt g)^{\gt g} = \left<\Delta_4,\Delta_2^2\right>_{\mK}.
$$ 
By the construction, $\Delta_2^2$ contains the summand $F_{11}^4$. 
Since $F_{1 1}^3$ cannot be a factor of a summand 
of $\Delta_6$, we conclude that ${\sf m}(\Delta_6)$ is proportional  to $\Delta_4$. 

Let $y\otimes F_{11}^2 F_{22}$ be a summand of ${\sf m}(\Delta_6)$. Then 
$y\in\mK F_{22}$. Also 
$y=\frac{-3!3!}{6!}{\sf m}(\partial_{F_{22}} \Delta_4^{[1]})$, where  
 $\Delta_4^{[1]}\in\cS^4(\gt{sp}_{2n-2})$ stands for  $\Delta_4$ of 
 $\gt{sp}_{2n-2}\subset \gt g_{F_{11}}$. 
Next we compute $\eta={\sf m}(\partial_{F_{22}} \Delta_4^{[1]})(F_{12})$. This will settle part {\sf (ii)}. 
So far we have shown that 
\begin{equation} \label{4}
\varpi(\partial_{F_{11}}\Delta_4)\in\U(\gt g) \ \text{ acts \ as } \ cF_{11} \ \text{ on } \ \gt g \ \text{ and \ on } \ \mK^{2n}  
\end{equation} 
and part {\sf (ii)} describes this constant $c$, which is to be computed. 

Recall that $F_{12}=- F_{(2n{-}1) 2n}$.  The terms of  $\Delta_4^{[1]}$  have neither 
$1$ nor $2n$ in the indices. Therefore a non-zero action on $F_{12}$ comes only from 
the following summands of $\Delta_4^{[1]}$
\begin{align}
& F_{22}  F_{(2n{-}1) j} F_{js}       F_{s (2n{-}1)}, \ \  - F_{22}  F_{(2n{-}1) s} F_{s' s'}       F_{s (2n{-}1)}  \qquad \ \ \ 
\text{and } \label{l1} \\ 
& F_{(2n{-}1)(2n{-}1)} F_{j2} F_{sj} F_{2s}, \ \ - F_{(2n{-}1)(2n{-}1)} F_{s2} F_{s's'} F_{2s}.  \label{l2} 
\end{align}
One easily computes that 
$$
{\sf m}(F_{(2n{-}1) j} F_{js}       F_{s (2n{-}1)})(F_{12})= \left\{ \begin{array}{l}
\frac{1}{6} F_{12}  \  \  \text{ if } \ \ j\ne s', \\
 \frac{1}{3} F_{12}  \  \ \text{ if } \ \ j = s',  \ \ \text{  because } \ \  F_{s' s}=2E_{s' s}\,;
 \end{array} \right.
 $$
and that ${\sf m}(F_{(2n{-}1) s} F_{s' s'}       F_{s (2n{-}1)})(F_{12})= -\frac{1}{6}F_{12}$.  
There are $2n{-}4$ choices  for $s$ in  line~\eqref{l1}. If $s$ is fixed, then  there are $2n-5$ possibilities for $j$, since 
$j\ne s$, but the choice $j=s'$ has to be counted twice. 
Applying the symmetry $F_{uv}=\pm F_{v'u'}$, we see that the terms in line~\eqref{l2}  
are the same as in~\eqref{l1}. 
Now $$
\eta=\frac{1}{3}( (2n-4)^2+(2n-4))F_{12}=\frac{(2n-4)(2n-3)}{3}F_{12}.
$$ 
Hence $y= \frac{(n-2)(2n-3)}{30}F_{22}$.   Since $\frac{(n-2)(2n-3)}{30} F_{11}\otimes F_{11}F_{22}^2$ 
is also a summnad of 
${\sf m}(\Delta_6)$, we obtain ${\sf m}(\Delta_6)=\frac{(n-2)(2n-3)}{15}\Delta_4$.
\end{proof}

\begin{prop} \label{sp-m}
We have ${\sf m}_{2r+1}(\Delta_{2k})=\frac{(2k-2r)!(2r)!}{(2k)!}\binom{2n-2k+2r+1}{2r} \Delta_{2k-2r}$. 
\end{prop}
\begin{proof}
First we have to show that ${\sf m}(\Delta_{2k})\in\cS^{2k-2}(\gt g)$. 
By Lemma~\ref{sp-w}, ${\sf m}(\Delta_{2k})$ is a $G$-inva\-riant polynomial function on 
$\gt g\oplus\gt g$. We use again the fact that $G(\gt g\oplus\gt h)$ is dense in $\gt g\oplus\gt g$.

Examine first the summands of $\Delta_{2k}$ that lie in 
$\cS^3(\gt g)\cS^{2k-3}(\gt h)$. Such a summand has the form 
$y_1y_2y_3 (F_{i_1 i_1}\ldots F_{i_s i_s})^2 F_{j_1 j_1}\ldots F_{j_u j_u}$.  Here 
$j_a\ne j_b,j_b'$ if $a\ne b$ and the product 
$y_1y_2y_3$ is an element of weight zero lying in $\cS^3(\gt f)$, where 
$\gt f$ is a subalgebra of $\gt g$ isomorphic either to  
$\gt{sp}_6$
or $\gt{sp}_4$. 
Furthermore, the numbers $i_b, i_b'$ with $1\le b\le s$ do not appear among the indices of the elements of $\gt f$ and 
at most three different numbers $j_b$  with $1\le b\le u$  can appear among the indices of the elements of $\gt f$.

If $u>3$, we can change at least one $F_{j_b j_b}$ to $F_{j_b' j_b'}=-F_{j_b j_b}$ without altering 
the other factors and produce a different summand of $\Delta_{2k}$. These two expressions annihilate 
each other. 
Therefore $u=3$ or $u=1$. 
Suppose that $u=3$ and that there is no way to annihilate the term 
via $F_{j_b j_b}\mapsto F_{j_b' j_b'}$.  Then 
$\gt f\cong\gt{sp}_6$ and
$y_1y_2y_3 F_{j_1 j_1} F_{j_2 j_2} F_{j_3 j_3}$ 
is a summand of the determinant $\Delta_6^{(\gt f)}\in\cS^6(\gt f)$. 
Since ${\sf m}(\Delta_6)$ is proportional to $\Delta_4$ by Lemma~\ref{sp-k6} and $F_{j_1 j_1} F_{j_2 j_2} F_{j_3 j_3}$
cannot appear in $\Delta_4$, we conclude that terms  
$y\otimes (F_{i_1 i_1}\ldots F_{i_s i_s})^2 F_{j_1 j_1}\ldots F_{j_u j_u}$
with $u>1$ 
does not appear in ${\sf m}(\Delta_{2k})$. 

Fix an element  $\lH=F_{11}F_{22}^2\ldots F_{ll}^2$ with $l=k-1$.  
We compute the coefficient of $\lH$ in ${\sf m}(\Delta_{2k})$. 
This coefficient is equal to $(-1)^{k}\frac{3!(2k-3)!}{(2k)!} {\sf m}(\partial_{F_{11}} \Delta_{4}^{(l)})$, where 
$\Delta_4^{(l)}$ is $\Delta_4$ of  the $\gt{sp}_{2n-2k+4}$-subalgebra generated by $F_{ij}$ with $i,j\not\in\{2,2',\ldots,l,l'\}$.  
According to Lemma~\ref{sp-k6}, ${\sf m}(\partial_{F_{11}} \Delta_{4}^{(l)})= \frac{-1}{3} (2n-2k+3)(2n-2k+2) F_{11}$.
Making use of the action  of the Weyl group $W(\gt g,\gt h)$ on $\gt h\oplus\gt h$, we conclude that 
${\sf m}(\Delta_{2k})$ is proportional to $\Delta_{2k-2}$, more explicitly 
$$
{\sf m}(\Delta_{2k})= \frac{2(k-1)(2n-2k+3)(2n-2k+2) (2k-3)!}{(2k)! } \Delta_{2k-2}
$$
and with some simplifications 
$$
{\sf m}(\Delta_{2k})=  \frac{(2n-2k+3)(2n-2k+2)}{ 2k(2k-1)  } \Delta_{2k-2} = 
\binom{2k}{2}^{\!\!\!-1}\!\binom{2n-2k+3}{2} \Delta_{2k-2}.
$$
Iterating the map ${\sf m}$, one obtains the result. 
\end{proof}

\begin{thm} \label{thm-C}
For $\gt g=\gt{sp}_{2n}$ and $1\le k\le n$, 
$$
S_{k} = \varpi(\Delta_{2k}[-1]) + 
\sum_{1\le r<k} \binom{2n{-}2k{+}2r{+}1}{2r}\varpi(\tau^{2r}\Delta_{2k{-}2r}[-1]){\cdot}1
$$
is a Segal--Sugawara vector. \qed 
\end{thm}

\section{Several exceptional examples}%
\label{sec-Spec}

There are instances, where our methods work very well.

\begin{ex}\label{so8}
Suppose that $\gt g=\gt{so}_8$. This Lie algebra has ${\tt S}_3$ as the group of the outer automorphisms. 
There are two Segal--Sugawara vectors, say $S_2$ and $S_3$, such that their symbols are 
$\gt g$-invariants of degree $4$ in $\cS(\gt g[-1])$. Assume that $S_2$ and $S_3$ are 
fixed vectors of $\omega$.  Then each of them is a sum 
$\varpi(Y_4)+\varpi(Y_2)$, cf.~\eqref{sim-alg}. Each element in  
$\cS^2(\wg^-)^{\gt g}$ is proportional to ${\mathcal H}[\bar b]$ for some $\bar b$. 
Hence it  is also an invariant of ${\tt S}_3$.  
Without loss of generality we may assume that the symbols of $S_2$ and $S_3$ are 
Pfaffians ${\rm Pf}_2[-1]$, ${\rm Pf}_3[-1]$ 
related to different matrix realisations of $\gt{so}_8$. Then for each of them there is 
an involution $\sigma\in{\tt S}_3$ such that $\sigma(Y_4)=-Y_4$.  
Replacing $S_j$ with $S_j-\sigma(S_j)$, we see that 
$\tilde S_2=\varpi({\rm Pf}_2[-1])$ and $\tilde S_3=\varpi({\rm Pf}_3[-1])$ are also Segal--Sugawara vectors.  
In view of Theorem~\ref{m-sym}, this implies that ${\sf m}({\rm Pf}_2)={\sf m}({\rm Pf}_3)=0$.  
\end{ex} 

Automorphisms of $\gt g$ make themselves extremely useful. We will see the full power of this devise in Section~\ref{sec-ort}, which deals with the orthogonal case. 
At the moment  notice the following thing, any $\sigma\in{\rm Aut}(\gt g)$ acts on $\cS(\gt g)$ in the natural way and induces a map 
$\sigma_{(m)}\!: \cS^m(\gt g)\to \cS^m(\gt g)$.   Let 
$\gt v_m\otimes \gt g$ be the isotypic component of $\cS^m(\gt g)$ corresponding to $\gt g$. 
Then $\sigma$ acts on $\gt v_m$ and  for this 
action, we have $\sigma(v)\otimes\sigma(x)=\sigma_{(m)}(v\otimes x)$, where  $v\in\gt v_m$, $x\in\gt g$.

An interesting story is related to Pfaffians in higher ranks.

\begin{ex}[The Pfaffians]\label{paf} 
Take $\gt g=\gt{so}_{2n}$. If $n<4$, these algebras appear in type {\sf A}. 
In case $n=4$, the Pfaffian-like  Segal--Sugawara vectors are examined in Example~\ref{so8}.
Suppose that $n>4$ and  that $\gt{so}_{2n}\subset\gt{gl}_{2n}$ consists of the skew-symmetric w.r.t. the 
antidiagonal matrices.  
The highest weight of the Cartan component of $\Lambda^2\gt g$ is
$\pi_1{+}\pi_3=2\esi_1+\esi_2+\esi_3$. 
Assume that $\pi_1{+}\pi_3$ appears as the weight of 
a factor $y_1y_2y_3$ for a summand $y_1\ldots y_n$ of the Pfaffian ${\rm Pf}$. 
Then 
up to the change of indices,  we must have 
$$
y_1=(E_{1i}-E_{i' 1'}), \ y_2=(E_{1j}-E_{j'1'}), 
$$
where $i'=2n+1-i$. 
If this is  really the case, then 
the determinant $\Delta_{2n}\in\cS^{2n}(\gt{gl}_{2n})$ has a summand
$E_{1i}E_{1j}\ldots $, a contradiction. 

Thus  ${\sf m}({\rm Pf})\in (\gt g\otimes\cS^{n{-}3}(\gt g))^{\gt g}$. 
If $n$ is odd, then there is no copy of $\gt g$ in $\cS^{n{-}3}(\gt g)$ and 
we conclude at once that the image of the Pfaffian under {\sf m} is zero. 

Suppose that $n$ is even. Then we can rely on  the fact that 
$G(\gt g\oplus \gt h)$ is dense in $\gt g\oplus\gt h$. 
Fix a factor ${\boldsymbol H}\in\cS^{n-3}(\gt h)$ of a summand of 
${\rm Pf}$. Without loss of generality assume that 
${\boldsymbol H}= \prod\limits_{s>3} (E_{ss}-E_{s's'})$. 
Let ${\rm Pf}^{(3)}$ be the Pfaffian of the subalgebra spanned by 
$$
E_{ij}-E_{j'i'}, \  E_{ij'}-E_{ji'}, \ E_{i'j}-E_{j'i}  \ \ \text{  with } \ \ i,j\le 3. 
$$
Since this subalgebra is isomorphic 
to $\gt{so}_6\cong \gt{sl}_4$ write also $\tilde\Delta_3^{(4)}$ for  ${\rm Pf}^{(3)}$. 
By the construction, $\frac{3!(n-3)!}{n!}{\sf m}({\rm Pf}^{(3)})\otimes {\boldsymbol H}$ is a summand 
of ${\sf m}({\rm Pf})$. For the Weyl involution $\theta$ of $\gt{sl}_4$, we have 
$\theta(\tilde\Delta_3^{(4)})=-\tilde\Delta_3^{(4)}$. Therefore 
$\varpi(\tilde\Delta_3^{(4)})$ acts as zero on any irreducible self-dual $\gt{sl}_4$-module,
in particular, on $\gt{sl}_4$ and on $\Lambda^2\mK^4=\mK^6$. 
Now we can conclude that ${\sf m}({\rm Pf}^{(3)})=0$ and hence 
${\sf m}({\rm Pf})=0$. Thus $\varpi({\rm Pf})[-1]$ is a Segal--Sugawara vector for 
each $n$. 

Keep the assumption that $n$ is even.  
Another way to see that ${\sf m}({\rm Pf})=0$ 
is to use 
the outer involution $\sigma\in {\rm Aut}(\gt{so}_{2n})$ such that $\sigma({\rm Pf})=-{\rm Pf}$. 
Here $\sigma(v)=-v$ for $v\in\gt v_{n-1}$ such that   $v\otimes\gt g$ 
is
the primitive copy of  $\gt g$ 
that gives rise to ${\rm Pf}$ and 
 also $\sigma({\sf m}({\rm Pf}))=-{\sf m}({\rm Pf})$. 
At the same time, $\sigma$ acts as $\id$ on $\gt v_{n-3}$. 
Therefore $\sigma$ acts as $\id$ on $(\gt g\otimes\cS^{n-3}(\gt g))^{\gt g}$. 
Since ${\sf m}({\rm Pf})\in (\gt g\otimes\cS^{n-3}(\gt g))^{\gt g}$, it must be zero. 
\end{ex}

One finds explicit formulas for the Pfaffian-type  Segal--Sugawara vector
${\rm Pf}\, F[-1]\in\U(\gt{so}_{2n}[t^{-1}])$ in \cite{m:ff,Natasha}. In the basis
$\{F_{ij}^{\circ}=E_{ij}-E_{ji}\mid 1\le i<j\le 2n\}$ for $\gt{so}_{2n}$, the vector ${\rm Pf}\, F[-1]$ 
is written as a sum of monomials with pairwise commuting factors, see \cite{m:ff} and \cite[Eq.~(8.11)]{book2}. 
Hence it coincides with the symmetrisation of its symbol, in our notation, 
${\rm Pf}\, F[-1] = \varpi({\rm Pf})[-1]$. 
Example~\ref{paf} provides a different proof for \cite[Prop.~8.4]{book2}.  

Another easy to understand instance is provided by the invariant of degree $5$ in type~${\sf E}_6$.

\begin{ex}\label{E6-5}
Suppose that $\gt g$ is a simple Lie algebra of type ${\sf E}_6$. Let $H\in\cS(\gt g)^{\gt g}$ be a homogeneous invariant  
of degree $5$. Then 
${\sf m}(H)\in (\Lambda^2\gt g\otimes\cS^2(\gt g))^{\gt g}$. Here 
$\Lambda^2\gt g=V(\pi_3)\oplus\gt g$ and $\cS^2(\gt g)=V(2\pi_6)\oplus V(\pi_1{+}\pi_5)\oplus\mathbb C$.
Therefore ${\sf m}(H)=0$.
\end{ex}

Recall that we are considering only semisimle $\gt g$ now and that $(\,\,,\,)$ is fixed in such a way that 
${\mathcal H}\in\U(\gt g)$ acts on $\gt g$ as $C\id_{\gt g}$ for some $C\in\mK$. 

\begin{lm} \label{H-sc}
There is $c_1\in\mK$ depending on the scalar product $(\,\,,\,)$ such that 
$\sum_{i} x_i[\xi,x_i] = c_1 \xi$ in $\U(\gt g)$
for any $\xi\in\gt g$. Furthermore, ${\mathcal H}(\xi)=\sum_{i} \ad(x_i)^2 (\xi)=-2c_1 \xi$, i.e., $C=-2c_1$.  
\end{lm}
\begin{proof}
For each  $\xi\in\gt g$, set
$\psi_1(\xi)=\sum_{i} x_i[\xi,x_i] = \sum_i x_i\xi x_i - {\mathcal H} \xi$. 
Since
$[{\mathcal H},\xi]=0$, we have $\omega(\psi_1(\xi))=-\psi_1(\xi)$. 
$$
\psi_1(\xi)= \sum_{i} [x_i,[\xi,x_i]] + \sum_{i} [\xi,x_i]x_i = - {\mathcal H}(\xi)  +\omega(\psi_1(\xi))=
- {\mathcal H}(\xi)  - \psi_1(\xi). 
$$
Thereby $2\psi_1(\xi)=-{\mathcal H}(\xi)=-C\xi$. 
\end{proof}

\begin{lm} \label{dva}
Let $\gt g$ be a simple Lie algebra of rank at least $2$. Then 
${\sf m}({\mathcal H}^3)\not\in \gt g\otimes\cS^3(\gt g)$. 
\end{lm}
\begin{proof}
Choose an orthogonal basis of $\gt h$ such that at least one element in it is equal to $h_\alpha$ for 
a simple root $\alpha$. 
If the root system of $\gt g$ is not simply laced,  suppose  that $\alpha$ is a 
long root.  Suppose further that either $\alpha=\alpha_1$ or $\alpha=\alpha_\ell$. 
Changing the scalar product if necessary, we may  assume that $h_\alpha\in\{x_i\}$. 
Consider  the summand  $\xi\otimes h_\alpha^3$ of ${\sf m}({\mathcal H}^3)$. We have 
$$
\xi  = \frac{3! 3!}{6!} \left(   6{\sf m}(h_\alpha{\mathcal H})   + 8 {\sf m} (h_\alpha^3) \right). 
$$
Set $\xi_0={\sf m}(h_\alpha{\mathcal H})$. Note that $h_\alpha{\mathcal H}\in\U(\gt g)$ acts on $\gt g$ as a 
scalar multiple of $\ad(h_\alpha)$. In view of Lemma~\ref{H-sc}, the sum $\sum_{i} x_i h_\alpha x_i$ acts on 
$\gt g$ as another multiple of $\ad(h_\alpha)$. Hence $\xi_0\in\gt g$. 

It remains to show that $\eta=\ad(h_\alpha)^3$ is not an element of $\gt g\subset\gt{so}(\gt g)$. 
Let $\alpha'$ be the unique simple root not orthogonal to  $\alpha$. 
Observe that $\eta(e_\alpha)=8 e_\alpha$ and 
$\eta(e_{\alpha'})= - e_{\alpha'}$. Set $\gamma=\alpha+\alpha'$. Then 
$e_\gamma\ne 0$ and $\eta(e_\gamma)=e_\gamma$. Since $1\ne 8-1$, we conclude that indeed 
$\eta\not\in\gt g$. 
\end{proof}

\begin{prop} \label{inv6}
Let $\gt g$ be an exceptional simple Lie algebra. 
Suppose that  $H\in\cS^6(\gt g)^{\gt g}$. Then there is $\lb\in\mathbb C$ such that 
 ${\sf m}(H-\lb{\mathcal H}^3)\in \cS^4(\gt g)$. 
\end{prop}
\begin{proof}
Let $V$ be the Cartan component of $\Lambda^2\gt g$. 
A straightforward calculation shows that $V$ appears in $\cS^3(\gt g)$ with multiplicity one: 
\begin{center}
\begin{tabular}{c|c|l}
Type & the highest weight of $V$ & $\cS^3(\gt g)$ \\
\hline
{\sf E}$_6$ &  $\pi_3$ &   $V(3\pi_6)\oplus V(\pi_1{+}\pi_5{+}\pi_6)\oplus V(\pi_3)\oplus V(\pi_1{+}\pi_5) \oplus \gt g$ \\
{\sf E}$_7$ &  $\pi_5$ & $V(3\pi_5)\oplus V(\pi_2{+}\pi_6)\oplus V(\pi_5) \oplus V(2\pi_1)  \oplus \gt g$ \\
{\sf E}$_8$ & $\pi_2$  & $V(3\pi_1)\oplus V(\pi_1{+}\pi_7) \oplus V(\pi_2) \oplus \gt g$ \\
{\sf F}$_4$ & $\pi_3$ & $V(3\pi_4)\oplus V(2\pi_1{+}\pi_4)\oplus V(\pi_3)\oplus V(\pi_2)\oplus\gt g $    \\
{\sf G}$_2$  & $3\pi_1$ & $V(3\pi_1)\oplus V(2\pi_1{+}\pi_2)\oplus V(3\pi_2)\oplus V(\pi_1) \oplus \gt g$  \\
\hline
\end{tabular}
\end{center}
We have ${\sf m}(H)\in (V\otimes\cS^3(\gt g))^{\gt g}\oplus (\gt g\otimes\cS^3(\gt g))^{\gt g}$. 
The first summand here is one-dimensional. Since ${\sf m}({\mathcal H}^3)\not\in \gt g\otimes \cS^3(\gt g)$ 
by Lemma~\ref{dva}, there is $\lb\in\mK$ such that ${\sf m}(\tilde H)\in \gt g\otimes \cS^3(\gt g)$ for 
$\tilde H=H-\lb{\mathcal H}^3$. 

The degrees of  basic symmetric invariants $\{H_k\mid 1\le k\le l\}$ indicate that 
$\cS^3(\gt g)$ contains exactly one copy of $\gt g$. (This is also apparent in the table above.) 
Hence $(\gt g\otimes\cS^3(\gt g))^{\gt g}=\cS^4(\gt g)^{\gt g}=\mK {\mathcal H}^2$. 
\end{proof}

\begin{cl} \label{cl-6}
Keep the assumption that $\gt g$ is exceptional. Then   $\tilde H=H-\lb{\mathcal H}^3$ of Proposition~\ref{inv6} 
satisfies \eqref{property} and there are $R(1),R(2)\in\mK$  such that 
\begin{equation} \label{S2-G2}
S_2=\varpi(\tilde H)[-1] + R(1) \varpi(\tau^2 {\mathcal H}^2[-1]){\cdot}1+ R(2) \varpi(\tau^4{\mathcal H}[-1]){\cdot}1
\end{equation}
is an element of $\gt z(\wg)$. 
\end{cl}
\begin{proof}
The first  statement follows from Proposition~\ref{inv6} and Example~\ref{k4}. 
More explicitly, ${\sf m}(\tilde H)\in\mK {\mathcal H}^2$, since there is no other symmetric invariant of degree four. Now the existence of $R(1)$ and $R(2)$ follows from Theorem~\ref{thm-form}. 
\end{proof}

\section{Type {\sf G}$_2$} \label{sec-g2}

Let $\gt g$ be a simple Lie algebra of  type {\sf G}$_2$. 
Then $\ell=2$. The algebra $\cS(\gt g)^{\gt g}$ has two generators, 
${\mathcal H}$ and $\Delta_6\in\cS^6(\gt g)$. 
In this section, we compute
the constant $\lb$ of Proposition~\ref{inv6} for $H=\Delta_6$ and $R(1), R(2)$ of \eqref{S2-G2}. 
All our computations are done by hand. 
A  computer-aided explicit formula for a Segal--Sugawara vector of $t$-degree $6$  is obtained in  \cite{mrr}.

First we fix a matrix realisation of $\gt g\subset\gt{so}_7$.
It suffices to describe the complement of 
$\gt{sl}_3 = (\gt{so}_6 \cap\gt g)$, which is isomorphic to $\mK^3\oplus(\mK^3)^*$. 
\begin{equation}   \label{m-g2}
\left(
\begin{array}{cccccc|c} 
 & & & -\tbe & \tga & 0 & \sqrt{2}\ta \\
 & & & \tal & 0 & -\tga &  \sqrt{2}\tb \\
 & & & 0 & -\alpha & \tbe &  \sqrt{2}\tc \\
 \tb & -\ta & 0 & & &  &  \sqrt{2}\tga \\
 -\tc & 0 & \ta & & & &   \sqrt{2}\tbe \\
 0 & \tc & -\tb & & & &  \sqrt{2} \tal \\
 \hline
 -\sqrt{2}\tal & -\sqrt{2}\tbe & -\sqrt{2}\tga & -\sqrt{2}\tc & -\sqrt{2}\tb & -\sqrt{2}\ta & 0 
\end{array}
\right)
\end{equation}
Matrix~\eqref{m-g2} presents an element of $\mK^3\oplus(\mK^3)^*\subset\gt{so}_7$. 
With a certain abuse of notation, we denote the elements of 
the  corresponding basis 
by the same symbols, for instance, 
$$
\ta=\sqrt{2}E_{17}-E_{42}+E_{53}-\sqrt{2} E_{76}
$$ 
as a vector of $\mK^3$. 
The embedding $\iota\!:\gt{sl}_3\to \gt{so}_7$ is fixed by 
$\iota(E_{ij})=E_{ij}-E_{(7-j)(7-i)}$ for $i\ne j$.
Choose a basis of $\gt h\subset\gt{sl}_3$ as $\{h_1,h_2\}$ with $h_1=\diag(1,-1,0)$, $h_2=\diag(1,1,-2)$ and extend it to a basis of $\gt{sl}_3$ by adding $e_i,f_i$ with $1\le i\le 3$ in the semi-standard notation, e.g. 
$e_3=E_{13}$, $f_1=E_{21}$, $f_3=E_{31}$.  
Let $\esi_i\in\gt h^*$ with $1\le i\le 3$ be the same as in Section~\ref{sec-A}.  
The scalar product $(\,\,,\,)$ is such that  $ {\mathcal H}=\Delta_2$ with 
$$
\Delta_2 = 2 e_1f_1 + 2 e_2f_2 + 2e_3f_3 + \frac{1}{2} h_1^2 + \frac{1}{6} h_2^2-\frac{2}{3}(\ta\tal + \tb\tbe + \tc\tga).
$$
The basic invariant of degree $6$, $\Delta_6$, is chosen as the restriction to $\gt g$ of the coefficient $\Delta_6^{(7)}$ 
of degree 6 in \eqref{d-gl} written for  $\gt{gl}_7$. In this case, 
the restriction of $\Delta_6$ to $\gt{sl}_3$ is equal to $-\tilde\Delta_3^2$, where $\tilde\Delta_3$ is the determinant of $\gt{sl}_3$. 
For the future use, we record 
\begin{align*}
& [\ta, \tal]=\diag(-2,1,1)\in\gt{sl}_3, \ \ [\tb,\tbe]=\diag(1,-2,1)\in\gt{sl}_3, \ \ [\tc,\tga]=h_2, \\
&  [\tal,\tc]=3f_3, \ \ [\tbe,\tc]=3f_2,   \ \ [\ta,\tb]=-2\tga,   \ \  [\tga,\tbe]=2\ta, \ \  
  [\tb,\tc]=-2\tal, \ \ [\tbe,\ta]=3e_1. 
\end{align*}

The decomposition $\gt g=(\mK^3)^*\oplus\gt g\oplus\mK^3$ is a $\Z/3\Z$-grading induced by an (inner) automorphism 
$\sigma$ of $\gt g$. Note that our basis for $\gt g$ consists of eigenvectors of $\sigma$. 

Recall that $S_2$ is given by \eqref{S2-G2} and that we are computing the constants occurring there. 
There is an easy part of the calculation.   It concerns  the projection of ${\sf m}(\Delta_2^3)$
on $(V\otimes \cS^3(\gt g))^{\gt g}$. As we already know, the highest weight of $V$ is $3\pi_1$. 
Next choose a monomial of  weight $3\pi_1$, for instance,  $e_3^2 f_1$. 

\begin{lm} \label{G1}
Let $\xi\otimes e_3^2 f_1$ be a summand of ${\sf m}(\Delta_2^3)$. 
Then $\xi(e_3)=\frac{6}{5} f_2$.
\end{lm}
\begin{proof}
Observe that in $\Delta_2^3$, the factor $e_3^2 f_1$ appears only in the summand $24 e_3^2 f_3^2 e_1f_1$. 
By the construction
$$
\xi= \frac{24\times  3! \times 3!}{6!}   {\sf m}(f_3^2e_1) = \frac{6}{5} {\sf m}(f_3^2e_1).
$$ 
Note that $[e_1,e_3]=0$. Hence 
$$
\begin{array}{l}
\frac{5}{6}\xi(e_3)=\frac{1}{6}(2\ad(e_1)\ad(f_3)^2+2\ad(f_3)\ad(e_1)\ad(f_3))(e_3)=\\
\qquad = \frac{1}{6}(-2\ad(f_2)\ad(f_3)-4\ad(f_3)\ad(f_2))(e_3)=-[f_2,[f_3,e_3]]=f_2
\end{array}
$$
and the result follows. 
\end{proof}

In the above computation,  we did not see the projection of ${\sf m}(\Delta_2^3)$ on 
$(\gt g\otimes \cS^3(\gt g))^{\gt g}$, which is equally important. Set $h_3=[e_3,f_3]$.
Note that $\{f_3,h_3,e_3\}$ is an $\gt{sl}_2$-triple associated with the highest root of $\gt g$. 
In the following lemma, $\gt{sl}_2$ means $\left<f_3,h_3,e_3\right>_{\mK}$. 

\begin{lm} \label{G2}
Let $\eta\otimes e_3^2f_3$ be a summand  of ${\sf m}(\Delta_2^3)$. 
Then  $\eta$ acts as $\frac{48}{5}\ad(f_3)$ on $\left<e_3,f_3,h_3\right>_{\mK}$ and 
as $\frac{42}{5}\ad(f_3)$ on the  $11$-dimensional $\gt{sl}_2$-stable complement of this 
subspace. 
\end{lm}
\begin{proof} 
In this case, one has to pay a special attention to the summand $8e_3^3f_3^3$ of $\Delta_2^3$.  
In the product $(e_3f_3)(e_3f_3)(e_3f_3)$, there are $6$ choices of $(e_3,e_3,f_3)$ such that 
$f_3$ and one of the elements $e_3$ belong to one and the same copy of $\Delta_2$; there are also $3$ other choices. 
These first $6$ choices are absorbed in $\frac{3!3!}{6!} 24 {\sf m}(f_3\Delta_2)$.  
Note that $\frac{3!\times 3! \times 24}{6!}=\frac{6}{5}$. 
The contribution to $\eta$ of the three other choices is $\frac{3!3!}{6!} 24 {\sf m}(e_3f_3^2)$. 
Hence 
$\eta=\frac{6}{5}({\sf m}(f_3\Delta_2)+{\sf m}(e_3f_3^2))$. 

The element $\varpi(\Delta_2)\in\U(\gt g)$ acts on $\gt g$ as a scalar. That scalar is $8$
in our case. 
According to Lemma~\ref{H-sc}, the sum $\sum_{i} x_i f_3 x_i\in \U(\gt g)$ is equal to 
${\mathcal H}f_3 -4 f_3$.   Thus, 
${\sf m}(f_3\Delta_2) = (8-\frac{4}{3})\ad(f_3)$ and 
$\frac{6}{5}{\sf m}(f_3\Delta_2)=8\ad(f_3)$. 

Consider now $\eta_0=(e_3f_3^2+f_3^2e_3+f_3e_3f_3)\in\U(\gt{sl}_2)$. 
Clearly, ${\eta}_0$ acts as zero on a trivial $\gt{sl}_2$-module;   for the defining representation on  
$\mK^2=\left<v_1,v_2\right>_{\mK}$ with $e_3v_1=0$, one obtains ${\eta}_0(v_1)=v_2$ and
${\eta}_0(v_2)=0$. This suffices to state that $\frac{6}{5}{\sf m}(e_3f_3^2)$ acts as 
$\frac{2}{5}\ad(f_3)$ on the $\gt{sl}_2$-stable complement of $\left<f_3,h_3,e_3\right>_{\mK}$. 
Finally, ${\eta}_0(f_3)=0$ by the obvious reasons,  
${\eta}_0(e_3)=-2h_3-2h_3=4\ad(f_3)(e_3)$ and ${\eta}_0(h_3)=4\ad(f_3)(h_3)$ as well, 
since ${\eta}_0$ acts on $\gt g$ as an element of $\gt{so}(\gt g)$. 
All computations are done now and the proof is finished.  
\end{proof}

Let ${\sf pr}\!\!:\gt{so}_7\to\gt g$ be the orthogonal projection. 
In order to work with $\Delta_6$, one needs to compute 
the images under  ${\sf pr}$
of  $F_{ij}=E_{ij}-E_{(8-j)(8-i)}\in\gt{so}_7$.  For the elements of $\gt{gl}_3\subset\gt{so}_6$, this is easy, 
the task reduces to $F_{ii}$ with $1\le i\le 3$, where we have 
$$
{\sf pr}(F_{11}) = \frac{1}{6}(3h_1+h_2), \ \ 
{\sf pr}(F_{22}) = \frac{1}{6}(-3h_1+h_2), \ \ 
{\sf pr}(F_{33}) = \frac{-1}{3}h_2.
$$
The elements of $F_{ij}\in\gt{so}_6$ with $1\le i\le 3$, $4\le j\le 6$ project with the coefficient $\frac{1}{3}$ on the corresponding letters in \eqref{m-g2}, e.g, 
${\sf pr}(F_{14})=\frac{-1}{3}\tbe$, ${\sf pr}(F_{15})=\frac{1}{3}\tga$, and so on.  The elements
$F_{i7}$ project with the coefficient $\frac{\sqrt{2}}{3}$ on the corresponding letters, e.g, 
${\sf pr}(F_{17})=\frac{\sqrt{2}}{3}\ta$. Finally,  the elements
$F_{7i}$ project with the coefficient $\frac{-\sqrt{2}}{3}$ on the corresponding letters, e.g, 
${\sf pr}(F_{71})=\frac{-\sqrt{2}}{3} \tal$. An explicit formula for $\Delta_6$ can be obtained by 
replacing first $E_{ij}$ with $F_{ij}$ in $\Delta_6^{(7)}\in\cS^6(\gt{gl}_7)$ and then 
replacing $F_{ij}$ with ${\sf pr}(F_{ij})$.
We write down some of the terms of $\Delta_6$:
$$
\Delta_6=-\tilde\Delta_3^2 - \frac{4}{27} \tc^3e_3^2f_1 - \frac{4}{9} \tc\tbe f_3 e_3^2 f_1 + \frac{2}{3} \tc \tal f_2 e_3^2 f_1 + \frac{1}{9} \tc\tal h_1 f_3 e_3^2 -\frac{1}{27} \tc \tal h_2 f_3 e_3^2  - \frac{4}{9} \tb\tal f_2f_3e_3^2+\ldots 
$$
With this knowledge we can attack the computation of ${\sf m}(\Delta_6)$. 
The first challenge is to understand the term $\tilde\xi\otimes e_3^2 f_1$. 

\begin{lm} \label{g2-3}
For $\tilde\xi$ as above, we have $\tilde{\xi}(e_3)=\frac{5}{18}f_2$.
\end{lm}
\begin{proof}
Once again, we rely on a direct computation.  
The  terms of  $\tilde\Delta_3$ containing $e_3$ as a factor  are $e_3f_1f_2$ and $e_3f_3(\frac{1}{2}h_1-\frac{1}{6}h_2)$. 
Thereby the  contribution of $-\tilde\Delta_3^2$ to $\tilde\xi$ is 
$$
\frac{1}{20} {\sf m}(-2f_1f_2^2-2f_2f_3 (\frac{1}{2}h_1-\frac{1}{6}h_2))
$$
and this element of ${\rm End}(\gt g)$ maps $e_3$ to $\frac{3}{20}  f_2$. 

Since $\sigma(\Delta_6)=\Delta_6$, the summands of $\Delta_6$ that contain $e_3^2 f_1$ as a factor 
are 
of tri-degrees 
$(3,3,0)$ or $(1,4,1)$ w.r.t. the $\Z/3\Z$-grading $\gt g=(\mK^3)^*\oplus\gt{sl}_3\oplus\mK^3$.
By the weight considerations, the first possibility occurs only for the monomial $\tc^3 e_3^2 f_1$. 
Record that
$$
\ad(\tc)^3(e_3)=\ad(\tc)^2(-\ta)=[\tc,2\tbe]=-6 f_2.
$$
The coefficient of $\tc^3 e_3^2 f_1$ in $\Delta_6$ is equal to $\frac{-4}{27}$. The monomials of the tri-degree 
$(1,4,1)$ are $\tc\tbe f_3 e_3^2 f_1$ and $\tc \tal f_2 e_3^2 f_1$.  Their coefficients are  
$\frac{-4}{9}$ and $\frac{2}{3}$, respectively. 

Next 
$$
{\sf m}(\tc\tbe f_3) (e_3) = \frac{1}{6}(\ad([\tbe,\tc])\ad(f_3)+2\ad(f_3)\ad([\tbe,\tc]))(e_3) = \frac{3}{2} \ad(f_2)\ad(f_3)(e_3)= -\frac{3}{2} f_2
$$
and 
$$
{\sf m}(\tc\tal f_2)(e_3)=\frac{1}{2}(\ad(f_2)\ad(\tc)\ad(\tal) +  \ad(f_2)\ad(\tal)\ad(\tc))(e_3)= \frac{1}{2}[\diag(1,-2,1),f_2]=\frac{3}{2}f_2. 
$$
Summing up 
$$
\tilde{\xi}(e_3)=\frac{1}{20}\left(3+\frac{8}{9}+\frac{2}{3}+1\right)f_2=\frac{1}{20}\left(5+\frac{5}{9}\right)f_2=
\frac{1}{4}\left(1+\frac{1}{9}\right)f_2=\frac{5}{18}f_2
$$ 
and we are done. 
\end{proof}

\begin{cl} \label{g2-4}
We have $\lb= \frac{25}{108}$ and the invariant $\tilde H$ of Proposition~\ref{inv6} is equal to 
$\Delta_6 - \frac{25}{108} \Delta_2^3$. 
\end{cl} 
\begin{proof}
By the definition of $\lb$, we must have $(\tilde\xi-\lb\xi)(e_3)=0$. From Lemmas~\ref{G1} and \ref{g2-3} we obtain that $\lb=\frac{5}{18}{\times} \frac{5}{6}=\frac{25}{108}$. 
\end{proof}

Next we deal with $\tilde\eta$ for the summand $\tilde\eta\otimes e_3^2f_3$ of ${\sf m}(\Delta_6)$.  
\begin{lm} \label{g2-5}
For $\tilde\eta$ as above, we have 
$$
\tilde\eta(\ta)=\frac{1}{20}\left(\frac{-2}{9} -\frac{28}{9}-\frac{4}{3}+\frac{2}{9}\right) \ad(f_3)(\ta)
=  \frac{-2}{9} \ad(f_3)(\ta)
$$
and 
$$
\tilde\eta(h_3)=\frac{1}{20}\left( \frac{1}{3} +\frac{8}{27}+\frac{4}{9} + \frac{1}{27}
\right) \ad(f_3)(h_3) = \frac{10}{9\times 20} \ad(f_3)(h_3) = \frac{1}{18} \ad(f_3)(h_3). 
$$
\end{lm}
\begin{proof}
We go through the relevant summands of $\Delta_6$. In $-\tilde\Delta_3^2$, these are 
$\frac{-1}{36}e_3^2f_3^2(3h_1-h_2)^2$ and $\frac{-1}{3} e_3^2 f_1 f_2 f_3 (3h_1-h_2)$. The corresponding contributions to $\tilde\eta$ are:
$$
\frac{-1}{18}{\sf m}(f_3(3h_1-h_2)^2) \ \ \text{ and } \ \ 
\frac{-1}{3}{\sf m}(f_1f_2(3h_1-h_2))
$$ 
multiplied by $\frac{1}{20}$.  
We are going to keep the factor $\frac{1}{20}$ in the background.  
Note that  ${\sf m}(f_3(3h_1-h_2)^2)$ acts on $\ta$ as  $4\ad(f_3)$, hence we add $\frac{-2}{9}$. Since 
$2-4+2=0$,  the second  of the above elements acts on $\ta$ as zero. If we consider the action on $h_3$ instead,  then the contribution of the first term is zero and   ${\sf m}(f_1f_2(3h_1-h_2))$ acts as $\ad([f_1,f_2])=-\ad(f_3)$. 

On account of $\sigma$, the other relevant  terms have tri-degrees $(3,3,0)$, $(0,3,3)$, $(1,4,1)$, where the former two 
possibilities occur for $\frac{4}{27} \tc^2\tb e_3^2f_3$ and $\frac{4}{27} \tal^2\tbe e_3^2f_3$.  Here 
\begin{align*}
& {\sf m}(\tc^2 \tb)(h_3)=\frac{1}{3}(\ad(\tb)\ad(\tc)\ad(\tc)+\ad(\tc)\ad(\tb)\ad(\tc))(h_3)= \\
& = \frac{1}{3}(-2\ad(\tal)\ad(\tc)-4\ad(\tc)\ad(\tal))(h_3)=(-2\ad(f_3)-2\ad(\tc)\ad(\tal))(h_3).
\end{align*}
Since $[\tal,h_3]=\tal$ and $[\tc,\tal]=-3f_3$, the contribution in question is $\frac{4}{27}\ad(f_3)$. Similarly,  
\begin{align*}
& {\sf m}(\tal^2 \tbe)(h_3)=\frac{1}{3}(\ad(\tbe)\ad(\tal)\ad(\tal)+\ad(\tal)\ad(\tbe)\ad(\tal))(h_3)= \\
& = \frac{1}{3}(2\ad(\tc)\ad(\tal)+4\ad(\tal)\ad(\tc))(h_3)=(2\ad(\tal)\ad(\tc)-2\ad(f_3))(h_3).
\end{align*}
Since $[\tc,h_3]=\tc$,  $[\tal,\tc]=3f_3$, we obtain again $\frac{4}{27}\ad(f_3)$. 
A slightly different story happens at $a$. Namely, 
\begin{align*}
& {\sf m}(\tc^2 \tb)(\ta)=(-2\ad(f_3)-2\ad(\tc)\ad(\tal) +\ad(\tc)\ad(\tc)\ad(\tb))(\ta) = \\
& = (-2\ad(f_3) -2\ad(f_3)+ 4\ad(f_3))(\ta)=\left(-4+4\right)\ad(f_3)(\ta)=0; \\
& {\sf m}(\tal^2 \tbe)(\ta)=(2\ad(\tal)\ad(\tc) -2\ad(f_3)+\ad(\tal)\ad(\tal)\ad(\tbe))(\ta)= \\
& =(8-2-6)\ad(f_3) (\ta)=0.
\end{align*}

Consider now the terms of the tri-degree $(1,4,1)$. Let $\esi_i-\esi_j$ be the weight of 
the fourth element from $\gt{sl}_3$. 
Assume first that  $i\ne j$. 
Then $\esi_3-\esi_1=(\esi_i-\esi_j)+\esi_s-\esi_l$  for some $s$ and $l$.   
One of the possibilities is $i=3, j=1$, and $s=l$. The other two come from the decomposition $\esi_1-\esi_3=(\esi_1-\esi_2)+(\esi_2-\esi_3)$. 

In case $s=l$, the relevant term is  $\frac{4}{9} e_3^2f_3^2 \tb\tbe$ and its 
contribution to $\tilde\eta$ is $\frac{8}{9}{\sf m}(f_3 \tb\tbe)$. Since both $\tb$ and $\tbe$ commute with $f_3$ and $h_3$, we see that ${\sf m}(f_3 \tb\tbe)(h_3)=0$. Furthermore, 
$${\sf m}(f_3 \tb\tbe)(\ta)=\ad(\tb)\ad(\tbe)(\tc)-\frac{1}{2}\tc= -3\tc-\frac{1}{2}\tc = \frac{-7}{2}\ad(f_3)(\ta).
$$
In this way the summand $\frac{-28}{9}$ appears in the first formula of the lemma.

In case $s\ne l$, 
the relevant terms are $\frac{-4}{9} \tb\tal f_2 f_3 e_3^2$ and $\frac{-4}{9} \tc\tbe f_1 f_3 e_3^2$. 
On $h_3$, each of the elements ${\sf m}(\tb\tal f_2)$,  ${\sf m}(\tc\tbe f_1)$ acts as $\frac{-1}{2}\ad(f_3)$.  
Thus, $\frac{4}{9}$ appears in the second formula. Further, 
\begin{align*}
& {\sf m}(\tb\tal f_2)(\ta)=\frac{1}{6}(\ad(f_2)\ad(\tb)\ad(\tal)+2\ad(\tal)\ad(f_2)\ad(\tb))(\ta)= \\
 &  \frac{1}{6}(\ad(\tc)\ad(\tal)+2\ad(\tal)\ad(\tc))(\ta)= \frac{1}{2}\ad(f_3)(\ta)+\ad(f_3)(\ta)=\frac{3}{2}\ad(f_3)(\ta).
\end{align*}
In the same fashion ${\sf m}(\tc\tbe f_1)(\ta)=\frac{3}{2}\ad(f_3)(\ta)$. This justifies  $\frac{-4}{3}$ in the first formula.

The final term, which is $\frac{1}{27} \tc \tal (3h_1{-}h_2) f_3 e_3^2$, fulfills  the  case $i=j$, $s=3$, $l=1$.
Here we have 
 ${\sf m}((3h_1{-}h_2)\tc\tal)(\ta)=6\ad(f_3)(\ta)$, hence the last summand in the first formula is 
$\frac{2}{9}$. 
Similarly, 
$${\sf m}((3h_1-h_2)\tc\tal)(h_3)=\frac{1}{6}(2\ad((\tal)\ad(\tc)-2\ad(\tc)\ad(\tal))(h_3)=\ad(f_3)(h_3).$$
This justifies  $\frac{1}{27}$ in the second formula.   
\end{proof}

Lemma~\ref{g2-5} provides a different way to compute $\lb$. Namely, 
$\tilde\eta-\lb\eta$ has to act on $\gt g$ as a scalar  multiple of $\ad(f_3)$. Check,  
\begin{align} 
& \left(\tilde\eta - \frac{25}{108}\eta\right)(\ta)=\left(-\frac{2}{9}-\frac{25}{108}\times\frac{42}{5}\right) \ad(f_3)(\ta) = \frac{-13}{6}\ad(f_3)(\ta); \label{g2-6} \\
&  \left(\tilde\eta - \frac{25}{108}\eta\right)(h_3)=\left(\frac{1}{18}-\frac{5\times 48}{108}\right) \ad(f_3)(h_3) = \frac{-39}{18}\ad(f_3)(h_3) =\frac{-13}{6}\ad(f_3)(h_3). \nonumber
\end{align}
In order to compute $R(1)$ and $R(2)$, state first that according to \eqref{g2-6}, 
$\frac{-13}{6} \ad(f_3) \otimes e_3^2 f_3$ is a summand of 
${\sf m}(\tilde H)$.  
This indicates that if ${\sf m}(\tilde H)$  is written as an element of $\cS^4(\gt g)$, then it 
has a term $\frac{-13}{3} e_3^2 f_3^2$, which is a summand of $\frac{-13}{12} \Delta_2^2$.  
Thus ${\sf m}(\tilde H)=\frac{-13}{12} \Delta_2^2$. 
In terms of Lemma~\ref{H-sc}, we have 
$$
{\sf m}({\mathcal H}^2) = \frac{3!}{4!} \times 4 \left(-2c_1+\frac{1}{3}c_1\right) {\mathcal H} =  \frac{20}{3} {\mathcal H},
$$
since $c_1=-4$ in our case. Making use of Theorem~\ref{thm-form}, we obtain the main result of this section: 
\begin{equation} \label{g2-7}
S_2=\varpi(\Delta_6-\frac{25}{108}\Delta_2^3)[-1]-\frac{65}{4}\varpi(\tau^2 \Delta_2^2[-1]){\cdot}1
- \frac{325}{3} \varpi(\tau^4 \Delta_2[-1]){\cdot}1 
\end{equation} 
is an element of $\gt z(\wg)$. Furthermore,  $S_1={\mathcal H}[-1]$ and $S_2$ form a complete set 
of Segal--Sugawara vectors for $\gt g$.


\section{The orthogonal case} \label{sec-ort}

Suppose now that $\gt g=\gt{so}_n\subset\gt{gl}_n$ with $n\ge 7$. 
A  suitable  matrix realisation of $\gt g$ uses  the elements   $F_{i j}=E_{i j}-E_{j' i'}$ with $i,j\in\{1,\dots,n\}$,   $i'=n-i+1$. A rather  unfortunate thing is that ${\sf m}(\Delta_{2k}|_{\gt g})$ is not a symmetric invariant 
for $k>2$. Therefore we will be working with the coefficients $\Phi_{2k}\in\cS^{2k}(\gt g)^{\gt g}$ of 
$$
\det(I_n- q(F_{ij}))^{-1} = 1 + \Phi_2 q^2+ \Phi_4 q^4 +\ldots + \Phi_{2k} q^{2k} +\ldots\, . 
$$
The generating invariants of this type appeared in \cite[Sect.~3]{my}
in connection with the symmetrisation map and 
they can be used in \eqref{form-m} as well. In \cite{book2,my}, 
the elements $\Phi_{2k}$ 
are called {\it permanents}, but  they are not the permanents of matrices in the usual sense. 
Set $\gt h=\left<F_{jj} \mid 1\le j\le \ell\right>_{\mK}$.

In general, $\det(I_n- qA)^{-1}=\det(I_n+qA+q^2 A^2+\ldots)$ for $A\in\gt{gl}_n$. In particular, 
${\Phi_{2k}}|_{\gt h}$ is  equal to the homogeneous part of degree $2k$ of 
$$
\prod_{j=1}^{\ell} (1+ F_{jj}^2+ F_{jj}^4+  F_{jj}^6 + \ldots) \,.
$$
By the construction, ${\sf m}(\Phi_{2k})$ is a polynomial function on 
$(\Lambda^2\gt g\oplus\gt g)^*\cong\Lambda^2\gt g\oplus\gt g$.  
Set 
$$
\lf={\sf m}(\Phi_{2k})|_{\Lambda^2\gt g\oplus\gt h} \ \ \text{ and write } \ \
\lf=\sum_{\nu}^L \xi_\nu\otimes \lH_{\nu},
$$
where $\lH_{\nu}\in\cS^{2k-3}(\gt h)$ are linearly independent 
monomials in $\{F_{jj}\}$ and $\xi_\nu\in\Lambda^2\gt g$. 
Note that each $\Phi_{2k}$ is an invariant of ${\rm Aut}(\gt g)$.  Since $\Phi_{2k}$ is an
element of $\gt h$-weight zero, each $\xi_\nu$ is also of weight zero. 
Hence one can say  that $\lf$ is an invariant of $W(\gt g,\gt h)$.

Let $\sigma\in{\rm Aut}(\gt g)$ be an involution such that $\gt g_0=\gt g^{\sigma}\cong\gt{so}_{n-1}$, 
$\sigma(F_{11})=-F_{11}$, i.e., $F_{11}\in\gt g_1$, 
and
$\sigma(F_{ss})=F_{ss}$ for  $\ell\ge s>1$. 
Then $\gt g_{0,F_{11}}:=(\gt g_0)_{F_{11}}\cong\gt{so}_{n-2}$. 
Such an involution $\sigma$ is not unique and we fix it by assuming that 
\begin{equation} \label{g1}
\gt g_1=\left<F_{1i}+F_{i'1} \mid 1<i<n\right>_{\mK} \oplus \mK F_{11}.
\end{equation}
The centraliser $\gt g_{1,F_{11}}$ of $F_{11}$ in $\gt g_1$ is equal to $\mK F_{11}$. This property defines 
an involution of {\it rank one}. Set $\gt h_0=\left<F_{ss} \mid \ell\ge s>1\right>_{\mK}$.

By the construction, the map ${\sf m}$ is ${\rm Aut}(\gt g)$-equivariant. Here the group 
${\rm Aut}(\gt g)\subset\GL(\gt g)$ acts on $\gt{so}(\gt g)\subset\gt{gl}(\gt g)$ via konjugation. 
In particular,  $\sigma$ acts as $-\id$ on 
$\gt g_0\wedge \gt g_1\subset \gt{so}(\gt g)$ and as $\id$ 
on the subspaces $\Lambda^2\gt g_0$ and
$\Lambda^2\gt g_1$.  
For the future use, record: ${\sf m}(F_{ii}^3)=F_{ii}$ and if $i\ne j,j'$, then 
${\sf m}(F_{ii}F_{jj}^2)$ acts as $\id$ on $F_{ij}=-F_{j'i'}$, $F_{ij'}=-F_{ji'}$, as $-\id$ on 
$F_{ji}=-F_{i'j'}$, $F_{j'i}=-F_{i'j}$, and as zero on 
all other elements $F_{uw}$. 
In particular, ${\sf m}(F_{ii}F_{jj}^2)\not\in\gt g$ if $i\not\in\{j,j'\}$. 

\begin{lm} \label{odd}
Suppose that $\lH_\nu=F_{11}^{\beta_1}\ldots F_{\ell\ell}^{\beta_\ell}$ and $\xi_\nu\ne 0$. Then there is exactly one odd $\beta_j$ with 
$1\le j\le \ell$. Furthermore, if $\beta_1$ is odd, then 
$$\xi_\nu\in\left<(F_{1i}-F_{i'1})\wedge (F_{1i'}+F_{i1}) \mid 1< i < n \right>_{\mK}\oplus 
\left<F_{11}\wedge F_{ss}\mid 1<s\le \ell \right>_{\mK}.$$ 
\end{lm}
\begin{proof}
Without loss of generality assume that $\beta_j$ is odd for $j\le u$ and is even for $j>u$. 
Let $\sigma_j\in{\rm Aut}(\gt g)$ with $2\le j\le u$ be an involution of rank one such that $\sigma_j(F_{jj})=-1$
and $\sigma_j(F_{ss})=F_{ss}$ for $s\ne j,j'$. Following the case of $\sigma_1=\sigma$, fix $\sigma_j$ by setting 
$$
\sigma_j(F_{ji}+F_{i'j})=-F_{ji}-F_{i'j} \ \ \text{ for } \ \ i\not\in\{j,j'\}.
$$  
As we have already mentioned, $\sigma_j(\Phi_{2k})=\Phi_{2k}$ for each $j$. 
Thereby ${\sf m}(\Phi_{2k})$ is a $\sigma_j$-invariant as well. 
At the same time $\sigma_j({\boldsymbol H}_\nu)=-{\boldsymbol H}_{\nu}$ by the construction. 
Hence $\sigma_j(\xi_\nu)=-\xi_\nu$ for each $1\le j\le u$. 

The above discussion has 
clarified, how involutions $\sigma_j$ act on $\Lambda^2\gt g={\sf m}(\cS^3(\gt g))$. 
In particular, we must have $\xi_\nu\in\gt g_0 \wedge\gt g_1$. We know also that $\xi_\nu$ is an element of 
$\gt h$-weight zero and that $F_{11}\in\gt h$. Recall that $\gt g_{1,F_{11}}=\mK F_{11}$. The decomposition 
$\gt g_0=\gt g_{0,F_{11}}\oplus [F_{11},\gt g_1]$ indicates that
$\xi \in  \gt g_{0,F_{11}}\wedge F_{11} \oplus  [F_{11},\gt g_1] \wedge\gt g_1$. 
Both summands here are $\gt h$-stable. 
Furthermore, $( \gt g_{0,F_{11}}\wedge F_{11})^{\gt h}$ is spanned by $F_{ss}\wedge F_{11}$ 
with $\ell\ge s>1$. 

The subspace $[F_{11},\gt g_1]$ 
is spanned by $F_{1i}-F_{i'1}$, 
where $1<i<n$.
For each $i$,  the element of 
the opposite  $\gt h_0$-weight in $\gt g_1$ is $F_{1i'}+F_{i1}$. 
Note that $(F_{1i}+F_{i'1})\wedge (F_{1i'}-F_{i1})$ is an  eigenvector of
$F_{11}$ if and only if $i=i'$. Thus, $([F_{11},\gt g_1] \wedge\gt g_1)^{\gt h}$ is a linear span  of 
$$
\Xi(i):=(F_{1i}+F_{i'1})\wedge (F_{1i'}-F_{i1}) + (F_{1i'}+F_{i1})\wedge (F_{1i}-F_{i'1}) 
$$
with $1<i\le i'$.

If $u>1$, then $u\ge 3$. The involution $\sigma_2$ acts on $ F_{1i}\pm F_{i'1}$ as $\id$ if 
$2<i<n-1$. Therefore $\xi_\nu$ has to be a linear combination of 
$F_{22} \wedge F_{11}$ and $\Xi(2)$. At the same time, 
$\sigma_3$ acts as $\id$ on both these vectors.  This contradiction proves that $u=1$.  
\end{proof}

\begin{rem}
Lemma~\ref{odd} is valid for any homogeneous $\Phi\in\cS(\gt g)^{{\rm Aut}(\gt g)}$.  
\end{rem}

Fix now  ${\boldsymbol H}=\lH_\nu=F_{11}^{2b_1-1} F_{22}^{2b_2} \ldots F_{\ell \ell}^{2b_\ell}$ 
with $b_j\in\Z_{\ge 0}$ and $b_1\ge 1$. The task is to compute $\xi=\xi_\nu$.  
Set $b_{j'}=b_j$ for $j\le\ell$. In type {\sf B}, set also $b_{\ell{+}1}=0$.  
Below we list the terms $Y_3$ such that $Y_3 \lH$ is a summand
of  $\Phi_{2k}$: 
\begin{equation} \label{parts}
\begin{array}{l} 
F_{11}^3, \ F_{11}F_{jj}^2,  \ 2(b_1+1)(b_j+1)F_{11} F_{1j} F_{j1}, \ (b_i+1)(b_j+1)F_{11} F_{ij} F_{ji}, \\
\enskip 2b_1(b_j+1)F_{1j}F_{j1} F_{jj}, \ 2b_1(b_i+1)(b_j+1)F_{1i} F_{ij} F_{j1}, 
\end{array}
\end{equation} 
where $1<i,j<n$ and $i\not\in\{j,j'\}$, also  
in 
$F_{jj}$, we have $1< j\le\ell$. 
When computing ${\sf m}$, one has to take into account the additional coefficients appearing from 
the powers of $F_{ii}$. For instance, in case of $F_{11}^3$, this coefficient is  $\binom{2b_1+2}{3}$, 
for   $2b_1(b_j+1)F_{1j}F_{j1} F_{jj}$, the additional scalar factor is $2b_j+1$. 

We will show that $\xi$ acts on $F_{ij}$ as  $c(i,j)F_{11}$ for some constant 
$c(i,j)\in\mK$, compute these constants and see that all of  them are equal. 
Note  that $[F_{11},F_{ij}]=0$  if $i,j\not\in \{1,n\}$. 

\begin{lm} \label{lm-F11}
We have $\xi(F_{11})=0$, furthermore $\xi(F_{ij})=0$ if $i,j\not\in \{1,n\}$. 
\end{lm}
\begin{proof}
By a direct computation, we show that indeed $\xi(F_{11})=0$. 
Some expressions in \eqref{parts} act on $F_{11}$ as zero by obvious reasons. 
If one takes into account that $F_{1j}F_{11}F_{j1}+F_{j1}F_{11}F_{1j}$ acts as 
$[F_{j1},F_{j1}]$, this covers the first line of \eqref{parts}. The same argument
takes care of ${\sf m}(F_{jj}F_{1j}F_{j1})$.  It remains to 
calculate $\eta={\sf m}(F_{1i}F_{ij}F_{j1})(F_{11})$. Here we have 
$6\eta=(F_{11}-F_{ii}) + (F_{jj}-F_{11})$. If we switch $i$ and $j$, then the total sum is zero. 

Since $(F_{ss}\wedge F_{11})(F_{11})=F_{ss}$ up to a non-zero scalar, 
Lemma~\ref{odd} implies now that  
$$\xi \in \left<(F_{1i}-F_{i'1})\wedge (F_{1i'}+F_{i1}) \mid 1< i < n \right>_{\mK}.
$$ 
Hence $\xi(F_{ij})=0$ if $i,j\not\in \{1,n\}$. 
\end{proof}

\begin{lm} \label{lm-D}
Suppose that $n=2\ell$. Assume that $1<u<n$.  Then 
$\xi(F_{1u})=\frac{3!(2k-3)!}{(2k)!} C(1) F_{1u}$ and $C(1)$ is equal to 
$$
\binom{2b_1{+}2}{3}  +\frac{2}{3} b_1 
 \sum_{j=2}^{\ell} (2b_j+1)(b_j+1) + \frac{8}{3} b_1(b_1+1) \sum_{j=2}^{\ell} (b_j+1)+ \frac{8}{3} b_1\!\!\!\sum_{1<i<j\le\ell}\!\!\!(b_i+1)(b_j+1)\,. 
$$
\end{lm}
\begin{proof}
Recall that ${\sf m}(F_{11}^3)=F_{11}$. This leads to the summand  $\binom{2b_1+2}{3}$ of $C(1)$. 
Consider 
$$\xi_{1}^{(j)}=  {\sf m}(F_{11}F_{jj}^2) +   {\sf m}(F_{1j} F_{j1} F_{jj}-F_{1j'} F_{j'1} F_{jj})$$ 
with $1<j\le\ell$.  Here $\xi_1^{(j)}(F_{1u})= \frac{1}{3}F_{1u}$ for $u\not\in\{j,j'\}$ and 
$\xi_1^{(j)}(F_{1j})=(1-\frac{2}{3})F_{1j}$ as well as 
$\xi_1^{(j)}(F_{1j'})=(1-\frac{2}{3})F_{1j'}$. 
In $C(1)$, we have to add $\frac{1}{3}$ with the coefficients  
$$2b_1\binom{2b_j+2}{2} = 2b_1(b_j+1)(2b_j+1).$$ 
The next terms are $\xi_2^{(j)}={\sf m}(F_{11} F_{1j} F_{j1})$ with $1<j<n$. 
Here $\xi_2^{(j)}(F_{1u})=\frac{1}{3}F_{1u}$ for $u\not\in\{j,j'\}$. Furthermore, 
$\xi_2^{(j)}(F_{1j})=\frac{2}{3}F_{1j}$ and $\xi_2^{(j)}(F_{1j'})=0$. 
Adding $\xi_2^{(j)}$ and $\xi_2^{(j')}$ with $j\le\ell$ and recalling the coefficient of $\xi_2^{(j)}$, we obtain the summands  
$\frac{8}{3}(b_1+1)b_1(b_j+1)$.

Fix $1<i,j < n$ with $i\not\in\{j,j'\}$ and consider 
$$
\xi_3^{i,j}={\sf m}(F_{1i} F_{ij} F_{j1}), \ \ \ \xi_4^{i,j}={\sf m}(F_{11} F_{ij} F_{ji}).
$$
An  easy observation is that $\xi_4^{i,j}(F_{1u})=0$ if $u\not\in\{i,i',j,j'\}$. 
Also $\xi_3^{i,j}(F_{1u})=\frac{1}{6} F_{1u}$ in this case. 
Furthermore, $\xi_4^{i,j}(F_{1i})=\frac{1}{2}F_{1i}$ and  $\xi_4^{i,j}(F_{1i'})=\frac{1}{2}F_{1i'}$.  
A more lengthy calculation brings 
$$
\begin{array}{l}
\xi_3^{i,j}(F_{1i})=\frac{1}{6}F_{1i} - \frac{1}{6}((\ad(F_{1i})\ad(F_{j1})+\ad(F_{j1})\ad(F_{1i}))(F_{1j})=\frac{1}{6}F_{1i}-\frac{1}{6}F_{1i}=0; \\
\xi_3^{i,j}(F_{1j})=\frac{1}{6}F_{1j} +\frac{1}{6} \ad(F_{ij})(F_{1i}) =0; \quad 
\xi_3^{i,j}(F_{1j'})= \frac{1}{6} \ad(F_{1i})\ad(F_{j1})(F_{1i'})=-\frac{1}{6}F_{1j'}; \\
\xi_3^{i,j}(F_{1i'})= \frac{1}{6} \ad(F_{ij})\ad(F_{1i})(F_{ji'}) =\frac{1}{6} \ad(F_{ij}) (F_{j n})=-\frac{1}{6} F_{1i'}.
\end{array}
$$
Note that $\xi_4^{i,j}=\xi_4^{j,i}=\xi_4^{i',j'}$. Fix now $1<i<j\le \ell$ and consider
$$
\xi_5^{i,j}=\xi_3^{i,j}+\xi_3^{j,i}+\xi_3^{i',j}+\xi_3^{j,i'}+\xi_3^{i,j'}+\xi_3^{j',i}+\xi_3^{i',j'}+\xi_3^{j',i'}+
\xi_4^{i,j}+\xi_4^{i',j}+\xi_4^{i,j'}+\xi_4^{i',j'}.
$$
Here $2b_1(b_i+1)(b_j+1)\xi_5^{i,j}$ is a summand of $\frac{(2k)!}{3!(2k-3)!} \xi$. 
Moreover, $\xi_5^{i,j}(F_{1s})=\frac{4}{3}F_{1s}$ for each $1<s<n$. 
This justifies the last summand of $C(1)$. 
\end{proof}

Rearranging  the expression for $C(1)$, one obtains 
\begin{equation} \label{C1}
C(1)=\frac{2}{3} b_1 \big(\sum_{j=1}^{\ell} (b_j+1)(2b_j+1) + 4\sum_{1\le i<j\le\ell} (b_i+1)(b_j+1) \big).
\end{equation}

\begin{lm} \label{lm-B}
Suppose that $n=2\ell+1$. Assume that $1<u<n$.  Then 
$\xi(F_{1u})=\frac{3!(2k-3)!}{(2k)!} \tilde C(1) F_{1u}$ and $\tilde C(1)$ is equal to 
$$
C(1)+ \frac{4}{3}b_1(b_1+1) + \frac{4}{3}b_1\sum_{1<i\le\ell}(b_i+1)=C(1)+ \frac{4}{3} b_1 
 \sum_{1\le j\le\ell} (b_j+1). 
$$
\end{lm}
\begin{proof}
We have to take care of the instances, where $j=\ell+1=j'$. Here $F_{jj}=0$, thereby also $\xi_1^{(j)}=0$. 
By a direct calculation, $\xi_2^{(j)}(F_{1u})=\frac{1}{3}F_{1u}$ for each $u$. 
Recall that $\xi_2^{(j)}$ corresponds to $Y_3=2(b_1+1)(b_j+1)F_{11}F_{1j}F_{j1}$ in \eqref{parts} and that the additional scalar factor in this case is $2b_1$. 
Since $b_{\ell+1}=0$, we have to add $\frac{4}{3}b_1(b_1+1)$  to $C(1)$. 

The calculations for $\xi_3^{i,j}, \xi_3^{j,i}$, and $\xi_4^{i,j}$ have to be altered.
The modifications are:
$$
\xi_4^{i,j}(F_{1j})=F_{1j}, \ \ \ \xi_3^{i,j}(F_{1j})=-\frac{1}{3} F_{1j}, 
\ \ \ \xi_3^{j,i}(F_{1j})=-\frac{1}{3} F_{1j},
$$  
and $\xi_5^{i,j}$ with $1<i<j=l+1$ has a simpler form, here 
$$
\xi_5^{i,j}=\xi_3^{i,j}+\xi_3^{j,i}+\xi_3^{i',j}+\xi_3^{j,i'} + \xi_4^{i,j}+\xi_4^{i',j}. 
$$
The coefficient of this $\xi_5^{i,j}$ in $\frac{(2k)!}{3!(2k-3)!} \xi$ is 
$2b_1(b_i+1)$ and $\xi_5^{i,j}(F_{1u})=\frac{2}{3} F_{1u}$ for all $u$. This justifies  the 
second additional summand. 
\end{proof}

\begin{prop}\label{prop-D}
For $\gt g=\gt{so}_{n}$, we have 
${\sf m}(\Phi_{2k})=R(k) \Phi_{2k-2}$, where 
$$
R(k)=\frac{1}{k(2k-1)} \left(\binom{n}{2} + 2n(k-1) + (k-1)(2k-3) \right). 
$$
\end{prop}
\begin{proof}
According to Lemmas~\ref{lm-D} and \ref{lm-B}, there is $c(1)\in\mK$ such that $\xi(F_{1u})=c(1)F_{1u}$ for each $1<u<n$. 
Since $\xi\in\gt{so}(\gt g)$,  we have also $\xi(F_{u1})= -c(u) F_{u1}$. Taking into account Lemma~\ref{lm-F11}, 
we conclude that $\xi=c(1)F_{11}$. 

Simplifying \eqref{C1} and using Lemma~\ref{lm-B}, we obtain that 
$$
\begin{array}{l}
c(1)=\frac{2}{3}b_1\frac{3!(2k-3)!}{(2k)!} \left(2\big(\sum\limits_{j=1}^{\ell} b_j)^2 + (4\ell-1)\big(\sum\limits_{j=1}^{\ell} b_j\big)+
\ell+ 2\ell(\ell-1)\right) = \\
\qquad \qquad = \frac{b_1}{k(2k-1)(k-1)} \left(2(k-1)^2+(4\ell-1)(k-1)+\ell(2\ell-1) \right)
\end{array}
$$
in type {\sf D} and that 
$$
c(1)= \frac{b_1}{k(2k-1)(k-1)} \left(2(k-1)^2+(4\ell-1)(k-1)+\ell(2\ell-1) + 2(k-1)+2\ell  \right)
$$
in type {\sf B}. In both cases, the scalars $c(1)/b_1$ depend only on $k$ and $\ell$. 
Making use of the action of $W(\gt g,\gt h)$, we can conclude now that 
${\sf m}(\Phi_{2k})$ is a symmetric invariant and that it is equal to $R(k) \Phi_{2k-2}$ with $R(k)\in\mathbb Q$. 
More explicitly, $R(k)$ is equal to $2(k-1)\frac{c(1)}{2b_1}=\frac{(k-1)c(1)}{b_1}$.

In type {\sf D}, we have $2(k-1)^2+(4\ell-1)(k-1) = 2n(k-1)+(k-1)(2k-3)$ and 
$\ell(2\ell-1)=\binom{n}{2}$.  Quite similarly, in type {\sf B}, we have 
$\ell(2\ell-1)+2\ell=\ell(2\ell+1)=\binom{n}{2}$ and 
$$
2(k-1)^2+(4\ell-1)(k-1) + 2(k-1)=2n(k-1) + (k-1)(2k-3).
$$
Therefore multiplying $c(1)$ with $(k-1)/b_1$ we obtain the desired formula for $R(k)$. 
\end{proof}

It does not look like there is a nice way to iterate ${\sf m}$.

\begin{thm} \label{thm-ort}
For any $k\ge 2$, \ 
$S_k=\varpi(\Phi_{2k})[-1]+\sum\limits_{1\le r<k} R(k,r) \varpi(\tau^{2r} \Phi_{2k-2r}[-1]){\cdot 1}$ 
is a Segal--Sugawara vector if $R(k,r)=\frac{2^r}{(2r)!} \prod\limits_{u=1}^r \left( \binom{n}{2} + 2n(k-u) +
(k-u)(2k-2u-1) \right)$.
\qed 
\end{thm} 

\section{Applications and open questions}%
\label{open}

The Feigin--Frenkel centre can be used in order to construct commutative subalgebras of the enveloping 
algebra in finite-dimensional cases. There are two most remarkable instances.  

\subsection{Quantum Mishchenko--Fomenko subalgebras}  \label{qMF} 
Recall the construction from \cite{r:si}. 
For any $\mu\in\gt g^*$ and
a non-zero $u\in\mathbb C$, the map 
\begin{equation} \label{ff-map}
\varrho_{\mu,u}\!:\U\big(t^{-1}\gt g[t^{-1}]\big)\to \U(\gt g),
\qquad x t^d \mapsto  u^d x +\delta_{d,-1}\mu(x),\quad x\in\gt g,
\end{equation}
defines  a $G_\mu$-equivariant  algebra homomorphism. The image of $\gt z(\widehat{\gt g})$ 
under $\varrho_{\mu,u}$
is a commutative subalgebra $\tilde\ca_{\mu}$ of $\U(\gt g)$, which does not depend
on $u$ \cite{r:si,fftl}. Moreover, $\gr\!(\tilde\ca_\mu)$ contains the {\it Mishchenko--Fomenko
subalgebra} $\ca_\mu\subset\cS(\gt g)$ associated with $\mu$,
which is generated by
all $\mu$-{\it shifts}  $\partial_\mu^m H$ of the $\gt g$-invariants  $H\in\cS(\gt g)$.   
The main property of $\ca_\mu$ is that it is {\it Poisson-commutative}, i.e., $\{\ca_\mu,\ca_\mu\}=0$ \cite{mf:ee}. 
If $\mu\in\gt g^*\cong\gt g$ is {\it regular}, i.e., if $\dim\gt g_\mu=\rk\gt g$, then $\ca_\mu$ is a maximal w.r.t. inclusion Poisson-commutative 
subalgebra of $\cS(\gt g)$ \cite{codim3} and hence 
$\gr\!(\tilde\ca_\mu)=\ca_\mu$.  Several important properties and applications of  (quantum) MF-subalgebras are discussed e.g. in~\cite{FFR,v:sc}. 

Mishchenko--Fomenko subalgebras were introduced in \cite{mf:ee}, before the appearance of 
the  Feigin--Frenkel centre.   In \cite{v:sc}, Vinberg posed a problem of finding a {\it quantisation} of $\ca_\mu$.  
A natural idea is to look for a solution given by the symmetrisation map $\varpi$. 
For $\gt g=\gt{gl}_n$, 
the elements $\varpi(\partial_\mu^m\Delta_k)\in\U(\gt g)$ with $1\le k\le n$, 
$0\le m<k$ commute and therefore produce a solution to Vinberg's quantisation problem 
\cite{t:cs,fm:qs,my}. 

Consider ${\mathcal F}[\bar a]=\varpi(F)[\bar a]\in\U(\wg^-)$ 
corresponding to $F\in\cS^m(\gt g)^{\gt g}$ in the sense of \eqref{a}. Set $p=|\{i \mid a_i=-1\}|$. 
Then 
\begin{equation} \label{MF1}
\left< \varrho_{\mu,u} ({\mathcal F}[\bar a]) \mid u\in\mK\setminus\{0\} \right>_{\mK} = 
\left< \varpi(\partial^l_\mu F) \mid 0\le l\le p\right>_{\mK}.
\end{equation}
Combining \eqref{MF1} with \eqref{S-sym-k}, we conclude  immediately that for $\gt g=\gt{gl}_n$, 
the algebra $\tilde\ca_\mu$ is generated by $\varpi(\partial_\mu^m\Delta_k)$.  
This observation is not new, see \cite[Sect.~3]{my} and in particular Sect.~3.2 there 
for a historical overview and a more elaborated proof. 

In \cite[Sect.~3.3]{my}, sets of generators $\{H_i\mid 1\le i\le \ell\}$ of $\cS(\gt g)^{\gt g}$ such that 
\begin{equation} \label{MF-sym}
\tilde\ca_\mu={\sf alg}\left<\varpi(\partial_\mu^m H_i) \mid 1\le i\le \ell, \ 0\le m<\deg H_i\right>
\end{equation} 
are exhibited in types {\sf B}, {\sf C}, and {\sf D}. We rejoice to say that   in type {\sf C}, 
$H_k=\Delta_{2k}$ in the notation of Section~\ref{sec-C}. In the even orthogonal case, the set $\{H_i\}$ includes the Pfaffian. 
The other generators in types {\sf B} and {\sf D} are $\Phi_{2k}$ of Section~\ref{sec-ort}. 
Thus Theorems~\ref{thm-C} and \ref{thm-ort} provide a new 
proof of  \cite[Thm~3.2]{my}. 

 
In conclusion, we show that Proposition~\ref{inv6} 
confirms Conjecture~3.3 of \cite{my} 
in type~{\sf G}$_2$.

\begin{prop}\label{g2-sym}
Let $\gt g$ be a simple Lie algebra of type {\sf G}$_2$. Let $\tilde H\in\cS^6(\gt g)$ 
be a $\gt g$-invariant  satisfying  \eqref{property}, cf. Corollaries~\ref{cl-6} and \ref{g2-4}. Then $\tilde\ca_\mu$ is generated by 
$\mu, {\mathcal H}$, and $\varpi(\partial_\mu^m \tilde H)$ with $0\le m\le 5$.
\end{prop}
\begin{proof}
We work with $\mu$ as with an element of $\gt g$. 
Clearly $\partial_\mu {\mathcal H}=2 \sum (x_i,\mu) x_i=2\mu$ and 
$$
\varrho_{\mu,u}({\mathcal H}[b_1,b_2])=u^{b_1+b_2}{\mathcal H}+(u^{b_2}\delta_{b_1,-1}+u^{b_1}\delta_{b_2,-1}) \mu +
\delta_{b_1,-1}\delta_{b_2,-1} (\mu,\mu). 
$$ 
Let $S_2$ be the Segal--Sugawara vector provided by \eqref{S2-G2} and \eqref{g2-7}.
Set also $S_1={\mathcal H}[-1]$. Then $\{S_1,S_2\}$ is a complete set of Segal--Sugawara vectors. 
A general observation is that 
$\tilde\ca_\mu$ is generated by $\{\varrho_{\mu,u}(S_\nu) \mid u\in\mK^{^\times}\!, \nu=1,2\}$ \cite[Cor.~9.2.3]{book2}. We have already computed the images of $S_1$ and also of $\varpi(\tau^4{\mathcal H}[-1]){\cdot}1$. 
Similarly to \eqref{MF1}, the images of $\varpi(\tilde H)[-1]$ span 
$\left<\varpi(\partial_\mu^m H)\right>_{\mK}$. It remains to deal with 
$$
{\mathcal Y}=\varrho_{\mu,u}(\varpi(\tau^2 {\mathcal H}^2[-1]){\cdot}1) = 
u^{-6} {\mathcal Y}_4+u^{-5}  {\mathcal Y}_3 + u^{-4} {\mathcal Y}_2 + u^{-3} {\mathcal Y}_1.
$$ 
Here ${\mathcal Y}_1$ is proportional to $\mu$; the term ${\mathcal Y}_2$
 is a linear combination of ${\mathcal H}$ and $\mu^2$. Furthermore, 
${\mathcal Y}_3$ is a linear combination of  $\mu{\mathcal H}$ and 
$\sum_{i} x_i\mu x_i$, therefore of $\mu{\mathcal H}$ and  $\mu$, cf. Lemma~\ref{H-sc}.  
Finally, ${\mathcal Y}_4$ is a linear combination of ${\mathcal H}^2$ and  
$\sum_{i,j} x_i x_jx_j x_i=\sum_{i} x_i {\mathcal H} x_i={\mathcal H}^2$, 
$$
\sum_{i,j} x_ix_j x_i x_j = {\mathcal H}^2 +  \sum_{j} c_1 x_j x_j = {\mathcal H}^2+c_1{\mathcal H}.
$$ 
This completes the proof. 
\end{proof}

\subsection{Gaudin algebras} \label{G-A}

Recall that elements $S\in\gt z(\wg)$ give rise to higher Hamiltonians of the Gaudin models, which  describe 
completely integrable quantum spin chains \cite{FFRe}. 

A Gaudin model consists of $n$-copies of $\gt g$ and the Hamiltonians
$$
{\mathcal H}_k=\sum_{j\ne k} \frac{\sum_i x_i^{(k)} x_i^{(j)}}{z_k-z_j}  ,  \enskip 1\le k\le n, 
$$
where $z_1,\ldots,z_n$ are pairwise different complex numbers. 
Here $\{x_i^{(k)}\mid 1\le i\le\dim\gt g\}$ is an orthonormal  basis for the $k$'th copy of $\gt g$. 
These Gaudin Hamiltonians can be regarded as 
elements of $\U(\gt g)^{\otimes n}$ or of $\cS(\gt g\,{\oplus}\ldots{\oplus}\,\gt g)$. 
They commute (and hence Poisson-commute) with each other. 
Higher Gaudin Hamiltonians are elements of $\U(\gt g)^{\otimes n}$ that commute with all ${\mathcal H}_k$. 

The construction of \cite{FFRe} produces a {\it Gaudin subalgebra} ${\mathcal G}$, which consists of Gaudin Hamiltonians and
contains ${\mathcal H}_k$ for each $k$. Let $\Delta\U(\wg^-)\cong\U(\wg^-)$ be the diagonal of $\U(\wg^-)^{\otimes n}$.
Then a vector $\bar z=(z_1,\ldots,z_n)$ defines 
a  natural homomorphism 
$\rho_{\bar z}\!:\Delta\U(\wg^-) \to \U(\gt g)^{\otimes n}$. 
In this notation, ${\mathcal G}={\mathcal G}(\bar z)$
is the image of $\gt z(\wg)$ under $\rho_{\bar z}$. 
By the construction, ${\mathcal G}\subset(\U(\gt g)^{\otimes n})^{\gt g}$. 
Since $\gt z(\wg)$ is homogeneous in $t$, it is clear that 
$\Gz=\cG(c\bar z)$ for any non-zero complex number $c$. 

Gaudin subalgebras have attracted a great deal of attention, see e.g. \cite{G-07} and references therein. 
They are closely related to  quantum Mishchenko--Fomenko  subalgebras and share some of their properties. 
In particular, for a generic $\bar z$, the action of $\Gz$ on an irreducible finite-dimensional 
$(\gt g\,{\oplus}\ldots{\oplus}\,\gt g)$-module 
$V(\lambda_1)\otimes\ldots\otimes V(\lambda_n)$
is diagonalisable and has a simple spectrum on the subspace of highest weight vectors of the diagonal $\gt g$ 
\cite{L-Gau}. 
Applying $\rho_{\bar z}$, one obtains explicit formulas for higher Gaudin Hamiltonians
from explicit formulas for the generators of $\gt z(\wg)$.  
In the following, the discuss which generators of  $\gt z(\wg)$ one has to consider.

Let $\{S_1,\ldots,S_\ell\}$ with $\gr\!(S_k)=H_k[-1]$ be a complete set of Segal--Sugawara vectors 
as in Section~\ref{ssv}. 
Set $\deg H_k=:d_k$. 
Assume that $z_k\ne z_j$ for $k\ne j$ and that $z_k\ne 0$ for all $k$. 
According to \cite[Prop.~1]{G-07}, ${\mathcal G}(\bar z)$ has a set of algebraically independent 
generators $\{F_1,\ldots,F_{\bb(n)}\}$, where 
$\bb(n):=\frac{n-1}{2}(\dim\gt g+\ell)+\ell$, see also \cite[Thms~2\&3]{r:si}.
Moreover, exactly $(n-1)d_k +1$ elements among the $F_j$'s belong to 
$\left<\rho_{\bar z}(\tau^m(S_k)) \mid m\ge 0\right>_{\mK}$ \cite[Prop.~1]{G-07}. 
Note that $\sum_{k=1}^{\ell} d_k = (\dim\gt g+\ell)/2$. 
Furthermore,  the symbols $\gr\!(F_j)$ are algebraically independent as well and 
$\deg\gr\!(F_j)=d_k$ if $F_j\in \left<\rho_{\bar z}(\tau^m(S_k)) \mid m\ge 0\right>_{\mK}$, 
see the proof of \cite[Prop.~1(2)]{G-07} and \cite[Sect.~4]{r:si}.

\begin{rmk}
If $\ca\subset(\cS(\gt g)^{\otimes n})^{\gt g}$  is a Poisson-commutative algebra, then 
${\rm tr.deg}\,\ca\le \bb(n)$ by \cite[Prop.~1.1]{my}. Combining this with \cite[Satz~5.7]{BK-GK}, we obtain that  
${\rm tr.deg}\,\tilde\ca\le\bb(n)$ for a commutative subalgebra $\tilde\ca\subset (\U(\gt g)^{\otimes n})^{\gt g}$. 
Thus $\cG$ has the maximal possible  transcendence degree. 
Arguing in the spirit of \cite{codim3} and using the results of \cite{r:si}, one can show that 
$\cG$ is also a maximal commutative subalgebra of $(\U(\gt g)^{\otimes n})^{\gt g}$ w.r.t. inclusion. 
\end{rmk}

In case $n=2$, the application of our  
result looks particularly nice. Besides, this {\it two points} case has several  features.  
Suppose that $n=2$. 
Set $\gt l=\gt g\oplus\gt g$.
For $H\in\cS^d(\gt l)$, $\xi^{(1)}$ in the first copy of $\gt g$, $\eta^{(2)}$ in the second, and
a non-zero $c\in\mK$, write
\begin{equation}\label{2-dec}
H(\xi^{(1)}{+}c\eta^{(2)})=H_{d,0}(\xi^{(1)})+cH_{d-1,1}(\xi^{(1)},\eta^{(2)})+\ldots+c^{d-1}H_{1,d-1}(\xi^{(1)},\eta^{(2)})+c^dH_{0,d}(\eta^{(2)})
\end{equation}
Here $H_{d,0}$ belongs to the symmetric algebra of the first copy of $\gt g$.  
The symbol of   $\rho_{\bar z}(\tau^m(S_k))$ lies in $\left<(H_k)_{d_k-j,j} \mid 0\le j\le d_k\right>_{\mK}$. 
Since we must have $d_k+1$ linearly independent elements among these symbols, 
$\gr\!(\cG)$ is freely generated by $(H_k)_{d_k-j,j}$ with $1\le k\le\ell$, $0\le j\le d_k$. 

The Lie algebra $\gt l$ has the following symmetric decomposition 
\begin{equation} \label{dec-sym}
\gt l= \gt l_0 \oplus \gt l_1,   \ \text{ where } \  \ \gt l_1=\{(\xi,-\xi)\mid\xi\in\gt g\} 
\end{equation} 
and $\gt l_0=\Delta\gt g=\{(\xi,\xi)\mid\xi\in\gt g\}$ is the diagonal. 
Similarly to \eqref{2-dec}, one polarises $H\in\cS^d(\gt l)$ w.r.t. the decomposition  \eqref{dec-sym}. 
Let $H_{[j,d-j]}$ with $0\le j\le d$ be the arising components. 
Then 
$\left< H_{j,d-j} \mid 0\le j\le d\right>_{\mK}= \left< H_{[j,d-j]} \mid 0\le j\le d\right>_{\mK}$. 

On the one side, the polynomials $(H_k)_{j,d_k-j}$ generate $\gr\!(\cG)$, on the other,
the polynomials $(H_k)_{[j,d_k-j]}$ generate a Poisson-commutative subalgebra 
$\eus Z\subset \cS(\gt l)$ related to the symmetric pair $(\gt l,\gt l_0)$, 
which has many interesting properties \cite{OY}. 
Thus, our discussion results in the following statement. 

\begin{cl}[{cf. \cite[Example~6.5]{OY}}] \label{G-sym}
The two points Gaudin subalgebra $\cG(z_1,z_2)$ is a quantisation of the 
Poisson-commutative subalgebra $\eus Z$ associated with the symmetric pair 
$(\gt g{\oplus}\gt g,\Delta\gt g)$. 
\end{cl}

Let us give more information on the issue of Corollary~\ref{G-sym}. 
Observe that $\cG(z_1,z_2)=\cG(z_1{-}b,z_2{-}b)$ if 
$b\in \mK\setminus\{z_1,z_2\}$, see \cite[Prop.~1]{G-07}. 
Hence $\Gz=\cG(\frac{z_1-z_2}{2},\frac{z_2-z_1}{2})=\cG(1,-1)$. 
For $\rho=\rho_{1,-1}$, we have $\rho(\xi t^k)=\xi^{(1)}+(-1)^k \xi^{(2)}$, i.e., 
$\gt g[-1]$, as well as any $\gt g[-2k{-}1]$, is mapped into $\gt l_1$
and each $\gt g[-2k]$ is mapped into  $\Delta\gt g$.  
One can understand $\rho$ as the map from $\U(\wg^-)$ to $\U(\wg^-)/(t^2-1)\cong\U(\gt l)$. 

It is not difficult to see that $\gr\!(\rho(S_k))=(H_k)_{[0,d_k]}$, $\gr\!(\rho(\tau(S_k)))=(H_k)_{[1,d_k-1]}$, and 
in general 
$$
\gr\!(\rho(\tau^m(S_k)))\in m!(H_k)_{[m,d_k-m]} + \left<(H_k)_{[j,d_k-j]}\mid j<m\right>_{\mK}
$$
as long as $m\le d_k$. This shows that indeed $\gr\!(\cG)=\eus Z$ and that 
\begin{equation} \label{gen-G}
\cG={\sf alg}\left<\rho(\tau^m(S_k))\mid  1\le k\le\ell,  \ 0\le m \le d_k\right>.
\end{equation}

Suppose  that $H_1,\ldots,H_\ell$ are homogeneous generators of $\cS(\gt g)^{\gt g}$  
and for each $k$ there is $0\le k'<k$ such that 
${\sf m}(H_{k'})\in \mK H_{j}$, where  $H_0=0$.   If $\gt g$ is  simple and classical, then explicit descriptions of such sets are contained in 
Sections~\ref{sec-A}, \ref{sec-C}, \ref{sec-ort}. 

\begin{thm}\label{2-Gau}
If we keep the above assumption on the set $\{H_k\}$, then the two points Gaudin subalgebra $\cG\subset\U(\gt l)$ is generated by  $\varpi((H_k)_{d_k-j,j})$ with $1\le k\le\ell$, $0\le j\le d_k$. 
\end{thm}
\begin{proof}
For each $k$, let $S_k$ be the Segaal--Sugawara vector obtained from $H_k$ by~\eqref{form-m}.
We will show that 
$V_{\rm sym}{:=}\left<\varpi((H_k)_{d_k-j,j}) \mid 1\le k\le\ell, \, 0\le j\le d_k\right>_{\mK}$ is equal to   
$$
V_{\cG}{:=} \left<\rho(\tau^m(S_k))\mid  1\le k\le\ell,  \, 0\le m \le d_k\right>_{\mK}. 
$$
Since $\dim V_{\cG} =\dim V_{\rm sym}$, it suffices to prove the inclusion 
$V_{\cG}\subset V_{\rm sym}$. 
We argue by induction on $k$. If $k=1$, then $S_1=\varpi(H_1[-1])$. Hence 
$\left<\rho(\tau^m(S_1))\mid  0\le m \le d_1\right>_{\mK}$ is equal to
$\left<\varpi((H_1)_{d_1-j,j}) \mid  0\le j\le d_1\right>_{\mK}$. 
If $k\ge 2$ and $m\le d_k$, then according to the structure of \eqref{form-m} and our condition on $\{H_1,\ldots,H_\ell\}$, we have 
$$
\begin{array}{l}
\rho(\tau^m(S_k))\in m!\varpi((H_k)_{[m,d_k-m]})  + V_{m,k}, \ \ \text{ where } \\
\qquad 
V_{m,k}=  \left<\varpi((H_k)_{[j,d_k-j]})\mid j<m\right>_{\mK}
\oplus \left<\varpi((H_{k'})_{j,d_{k'}-j})\mid  k'<k, \, 0\le j<d_{k'}\right>_{\mK}. 
\end{array}
$$
Thus $\rho(\tau^m(S_k))\in V_{\rm sym}$ and we are done. 
\end{proof}

\subsection{Further directions}

For all classical  types, 
we find families of generators $\{H_k\}$ that behave well in terms of 
\eqref{property}.  In the orthogonal case, the result was  obtained by a brute force combined with a long calculation. 
It would be most 
desirable to obtain a better understanding of the map ${\sf m}$ and the relevant representation/invariant theory of 
$\gt g$. Such an understanding can replace tedious calculations. 

We say that a set of generators $\{H_k\}\subset\cS(\gt g)^{\gt g}$ satisfies~\eqref{property} if  each $H_k$ 
satisfies it. 
The general picture is not complete yet, since the following question remains open. 

\begin{qn}
Does any exceptional Lie algebra $\gt g$ poses a set of generators $\{H_k\}\subset\cS(\gt g)^{\gt g}$ satisfying \eqref{property}? 
\end{qn}

Proposition~\ref{inv6} takes care of type ${\sf G}_2$. We have seen  also some partial positive answers in other types. 

\begin{qn}
Are there homogeneous generators $\{H_k\}$ of $\cS(\gt g)^{\gt g}$ such that ${\sf m}(H_k)=0$ for each $k$? 
\end{qn}

The calculations in Section~\ref{sec-g2} prove that in type {\sf G}$_2$, the answer is negative. 
I would expect that the answer is negative in general. 

As  Example~\ref{k4} shows,  a set of generators $\{H_k\}$ satisfying \eqref{property} is not unique. 
For the classical Lie algebras, 
there is a freedom of choice in degree $4$ and there is also  some freedom in degree $6$, but we rather would not dwell on it. 

It is quite possible that the condition~\eqref{MF-sym}  on the set $\{H_k\}$ is less restrictive than \eqref{property}. However, we have no convincing evidence to this point. 

\begin{rmk} \label{rm-comb}
There are some intricate  combinatorial identities hidden in Formula~\eqref{form-m}. 
In order to  reveal them, 
one has to understand the natural numbers $c(r,\bar a)$ appearing in the proof of Lemma~\ref{side-1},  
the rational constants $c_{2,3}(j,p)$ of Lemma~\ref{lm-com-h} as well as the scalars $C(\bar a^{(r)},\bar\gamma)$ of
Proposition~\ref{prop-fs-c}. 
It is a miracle that the formula takes care of all these numbers. 
\end{rmk}

\vskip2ex

\noindent 
{\bf Acknowledgement.} I am grateful to Alexander Molev for his enlightening explanations on the subjects of  Segal--Sugawara vectors and vertex algebras. A special thank you is due to Leonid Rybnikov for bringing Gaudin
subalgebras to my attention. 

Part of this work was done during my visit to the Mathematical Institute of the Cologne University. 
It is a pleasure to thank Peter Littelmann for his hospitality and many inspiring conversations about 
mathematics in general as well as Kac--Moody  algebras in particular.

\end{document}